\documentclass[12pt]{amsart}
\usepackage{amsmath,amssymb,amsthm,bm,bbm}
\usepackage{graphicx,enumerate}
\usepackage{layout}

\theoremstyle{plain}
\newtheorem{theorem}{Theorem}[section]
\newtheorem{lemma}[theorem]{Lemma}
\newtheorem{proposition}[theorem]{Proposition}
\newtheorem{corollary}[theorem]{Corollary}
\theoremstyle{definition}
\newtheorem{definition}[theorem]{Definition}

\newtheorem*{acknowledgments}{Acknowledgments}

\newcommand{\bZ}{\mathbbm{Z}}\newcommand{\bQ}{\mathbbm{Q}}
\newcommand{\bC}{\mathbbm{C}}
\newcommand{\Br}{\mathrm{Br}}

\newcommand{\ta}{\tau}\newcommand{\la}{\lambda}
\newcommand{\al}{\alpha}\newcommand{\be}{\beta}\newcommand{\ga}{\gamma}
\newcommand{\de}{\delta}
\newcommand{\ca}{\sigma_{3{\rm A}}}
\newcommand{\cb}{\sigma_{3{\rm B}}}
\newcommand{\caa}{\sigma_{4{\rm A}}}
\newcommand{\cbb}{\sigma_{4{\rm B}}}
\newcommand{\cC}{\mathcal{C}}\newcommand{\cD}{\mathcal{D}} 
\newcommand{\cA}{\mathcal{A}}\newcommand{\cS}{\mathcal{S}} 

\title{Three-dimensional purely quasi-monomial actions}

\author{Akinari Hoshi} 
\address{Department of Mathematics, Niigata University, Niigata 950-2181, Japan}
\email{hoshi@math.sc.niigata-u.ac.jp}

\author{Hidetaka Kitayama}
\address{Department of Mathematics, Wakayama University, Wakayama 640-8510, Japan}
\email{hkitayam@center.wakayama-u.ac.jp}

\thanks{This work was supported by JSPS KAKENHI Grant Numbers 
24740014, 25400027, 15K17511. 
Some part of this work was done during 
the authors visited the National Center for Theoretic Sciences 
(Taipei Office), whose support is gratefully acknowledged.}

\begin{document}
\maketitle
\begin{abstract}
Let $G$ be a finite subgroup of $\mathrm{Aut}_k(K(x_1, \ldots, x_n))$ 
where $K/k$ is a finite field extension and $K(x_1,\ldots,x_n)$ is the 
rational function field with $n$ variables over $K$. 
The action of $G$ on $K(x_1, \ldots, x_n)$ is called quasi-monomial 
if it satisfies the following three conditions
{\rm (i)} $\sigma(K)\subset K$ for any $\sigma\in G$; 
{\rm (ii)} $K^G=k$ where $K^G$ is the fixed field under the action of $G$; 
{\rm (iii)} for any $\sigma\in G$ and $1 \leq j \leq n$, 
$\sigma(x_j)=c_j(\sigma)\prod_{i=1}^n x_i^{a_{ij}}$ 
where $c_j(\sigma)\in K^\times$ and $[a_{i,j}]_{1\le i,j \le n} \in GL_n(\bZ)$. 
A quasi-monomial action is called purely quasi-monomial 
if $c_j(\sigma)=1$ for any $\sigma \in G$, any $1\le j\le n$. 
When $k=K$, a quasi-monomial action is called monomial. 
The main problem is that, under what situations, 
$K(x_1,\ldots,x_n)^G$ is rational (= purely transcendental) over $k$. 
For $n=1$, the rationality problem was solved by Hoshi, Kang and Kitayama. 
For $n=2$, the problem was solved by Hajja 
when the action is monomial, 
by Voskresenskii 
when the action is faithful on $K$ and purely quasi-monomial, 
which is equivalent to the rationality problem of $n$-dimensional 
algebraic $k$-tori which split over $K$, 
and by Hoshi, Kang and Kitayama 
when the action is purely quasi-monomial. 
For $n=3$, the problem was solved by Hajja, Kang, Hoshi and Rikuna 
when the action is purely monomial, 
by Hoshi, Kitayama and Yamasaki 
when the action is monomial except for one case and by Kunyavskii 
when the action is faithful on $K$ and purely quasi-monomial. 
In this paper, we determine the rationality when $n=3$ and 
the action is purely quasi-monomial except for few cases 
using a technique of conjugacy classes move. 
%which is introduced by the second-named author. 
As an application, we will show the rationality of some 
$5$-dimensional purely monomial actions which are decomposable. 
\end{abstract}

\tableofcontents

%\newpage

%%%%%%%%%%%%%%%%%%%%%%%%%%%%%%%%%%%%%%%%%%%%%%%%%%%%%%%%%%%%%
\section{Introduction}

Let $k$ be a field and  
$K$ be a finitely generated field extension of $k$. 
$K$ is called {\it $k$-rational} (or {\it rational over $k$}) if $K$ is
purely transcendental over $k$, i.e. $K$ is isomorphic to $k(x_1,
\ldots, x_n)$, the rational function field with $n$ variables over
$k$ for some integer $n$. 
$K$ is called {\it stably $k$-rational} if 
$K(y_1,\ldots,y_m)$ is $k$-rational for some $y_1,\ldots,y_m$ 
such that $y_1,\ldots,y_m$ are algebraically independent over $K$. 
When $k$ is an infinite field, 
$K$ is called {\it retract $k$-rational} if $K$ is the quotient field of
some integral domain $A$ where $k\subset A\subset K$ satisfying
the conditions that there exist a polynomial ring
$k[X_1,\ldots,X_m]$, some non-zero element $f\in
k[X_1,\ldots,X_m]$, and $k$-algebra morphisms $\varphi:A\to
k[X_1,\ldots,X_m][1/f]$, $\psi: k[X_1,\ldots,X_m][1/f]\to A$ such
that $\psi\circ\varphi =1_A$. 
$K$ is called {\it $k$-unirational} if 
$k \subset K \subset k(x_1, \ldots, x_n)$ for some integer $n$. 
It is not difficult to see that 
``$k$-rational"$\Rightarrow$``stably $k$-rational"
$\Rightarrow$``retract $k$-rational"$\Rightarrow$``$k$-unirational". 
The reader is referred to the papers \cite{MT86,CTS07,Swa83}
for surveys of the various rationality problems, e.g. Noether's
problem. 
We will restrict ourselves to consider 
the rationality problem of the fixed field 
$K(x_1,\ldots,x_n)^G$ 
under the following special kind of actions:  

\begin{definition}[{\cite[Definition 1.1]{HKK14}}]\label{def}
Let $G$ be a finite subgroup of ${\rm Aut}_k(K(x_1,\ldots,x_n))$.\\ 
(1) The action of $G$ on $K(x_1,\ldots,x_n)$ is called {\it quasi-monomial} 
if it satisfies the following three conditions:\\
(i) $\sigma(K)\subset K$ for any $\sigma\in G$;\\
(ii) $K^G=k$, where $K^G$ is the fixed field under the action of $G$;\\
(iii) for any $\sigma\in G$ and any $1 \le j \le n$,
\begin{align*}
\sigma(x_j)=c_j(\sigma)\prod_{i=1}^n x_i^{a_{ij}}\label{acts1}
\end{align*} 
where
$c_j(\sigma)\in K^\times$ and $[a_{ij}]_{1\le i,j \le n} \in
GL_n(\bZ)$.\\
(2) The quasi-monomial action is called {\it purely quasi-monomial} action 
if $c_j(\sigma)=1$ for any $\sigma \in G$ and any $1\le j\le n$ in (iii). \\
(3) The quasi-monomial action is called {\it monomial} action 
if $G$ acts trivially on $K$, i.e.\ $k=K$.\\ 
(4) The quasi-monomial action is called {\it purely monomial} action 
%if purely and $k=K$. 
if it is purely quasi-monomial and monomial. 
\end{definition}
We have the following implications: 
\begin{center}
quasi-monomial action\ \ $\Leftarrow$\ \ 
purely quasi-monomial action\\
$\Uparrow$\hspace*{3.5cm}$\Uparrow$~~~~~~~~\\
\ monomial action\ \ $\Leftarrow$\ \ purely monomial action.
\end{center}

\vspace*{5mm}
%%%%%%%%%%%%%%%%%%%%%%%%%%%%%%%%%%
{\it Algebraic torus case.} 

When $G\simeq {\rm Gal}(K/k)$, i.e. $G$ acts faithfully on $K$, and 
$G$ acts on $K(x_1,\ldots,x_n)$ by purely quasi-monomial $k$-automorphisms, 
the rationality problem of $K(x_1,\ldots,x_n)^G$ coincides with 
that of algebraic $k$-tori of dimension $n$ which split over $K$ 
(see \cite[Section 1]{HKK14}, \cite{Hos14}, \cite{HY17}). 
The rationality problem of algebraic $k$-tori of dimension two and three 
was solved by Voskresenskii \cite{Vos67} 
and Kunyavskii \cite{Kun87} respectively:

\begin{theorem}[Voskresenskii \cite{Vos67}] 
\label{thm:Vos} 
Let $k$ be a field. All the two-dimensional algebraic
$k$-tori are $k$-rational. 
In particular, $K(x_1,x_2)^G$ is 
$k$-rational if $G$ acts on $K$ faithfully and on 
$K(x_1,x_2)$ by purely quasi-monomial $k$-automorphisms.
\end{theorem}

\begin{theorem}[{Kunyavskii \cite{Kun87}, see also Kang \cite[Section 1]{Kan12}}] 
\label{thm:Kun} 
Let $k$ be a field. 
All the three-dimensional algebraic
$k$-tori are $k$-rational except for the $15$ cases in the
list of \cite[Theorem 1]{Kun87}. 
In particular, $K(x_1,x_2,x_3)^G$ is 
$k$-rational if $G$ acts on $K$ faithfully and on 
$K(x_1,x_2,x_3)$ by purely quasi-monomial $k$-automorphisms 
except for the $15$ cases.
For the exceptional $15$ cases, they are not $k$-rational; 
in fact, they are even not retract $k$-rational.
\end{theorem}

%%%%%%%%%%%%

The rationality problem for algebraic $k$-tori 
depends on the conjugacy class of 
%$\bZ$-class of 
$G\leq GL_n(\bZ)$. 
%Note that 
There exist $13$ (resp. $73$, $710$, $6079$) 
%$\bZ$-classes 
conjugacy classes of finite subgroups 
%forming $227$ $\bQ$-classes 
in $GL_2(\bZ)$ (resp. $GL_3(\bZ)$, $GL_4(\bZ)$, $GL_5(\bZ)$). 
A stably rational classification of 
the algebraic $k$-tori of dimension four and five 
was 
given by \cite{HY17}. 
\begin{theorem}[Hoshi, Yamasaki {\cite[Theorem 1.9]{HY17}}]\label{th1}
Let $K/k$ be a Galois extension and $G\simeq 
{\rm Gal}(K/k)$ be a finite subgroup of $GL_4(\bZ)$ 
which acts on $K(x_1,x_2,x_3,x_4)$ by purely quasi-monomial 
$k$-automorphisms.\\% via $(\ref{acts})$. \\
{\rm (i)} 
$K(x_1,x_2,x_3,x_4)^G$ is stably $k$-rational 
if and only if 
$G$ is conjugate to one of the $487$ groups which are not in 
{\rm \cite[Tables $2$, $3$ and $4$]{HY17}}.\\
{\rm (ii)} 
$K(x_1,x_2,x_3,x_4)^G$ is not stably but retract $k$-rational 
if and only if $G$ is conjugate to one of the $7$ groups which are 
given as in {\rm \cite[Table $2$]{HY17}}.\\
{\rm (iii)} 
$K(x_1,x_2,x_3,x_4)^G$ is not retract $k$-rational if and only if 
$G$ is conjugate to one of the $216$ groups which are given as 
in {\rm \cite[Tables $3$ and $4$]{HY17}}.
\end{theorem}
\begin{theorem}[Hoshi, Yamasaki {\cite[Theorem 1.12]{HY17}}]\label{th2}
Let $K/k$ be a Galois extension and $G\simeq 
{\rm Gal}(K/k)$ be a finite subgroup of $GL_5(\bZ)$ 
which acts on $K(x_1,x_2,x_3,x_4,x_5)$ by purely quasi-monomial 
$k$-automorphisms.\\% via $(\ref{acts})$. \\\\% via $(\ref{acts})$.\\
{\rm (i)} 
$K(x_1,x_2,x_3,x_4,x_5)^G$ is stably $k$-rational if and only if 
$G$ is conjugate to one of the $3051$ groups which are not in 
{\rm \cite[Tables $11$, $12$, $13$, $14$ and $15$]{HY17}}.\\
{\rm (ii)} 
$K(x_1,x_2,x_3,x_4,x_5)^G$ is not stably but retract $k$-rational 
if and only if $G$ is conjugate to one of the $25$ groups which are given as 
in {\rm \cite[Table $11$]{HY17}}.\\
{\rm (iii)} 
$K(x_1,x_2,x_3,x_4,x_5)^G$ is not retract $k$-rational if and only if 
$G$ is conjugate to one of the $3003$ groups which are given as 
in {\rm \cite[Tables $12$, $13$, $14$ and $15$]{HY17}}.
\end{theorem}
However, we do not know the $k$-rationality %of algebraic $k$-tori 
in Theorem \ref{th1} (i) and Theorem \ref{th2} (i). 
%of dimension four and five. 

\vspace*{5mm}
%%%%%%%%%%%%%%%%%%%%%%%%%%%%%%%%%%
{\it Monomial action case.} 

When $G$ acts on $K$ trivially, i.e. $K=k$, 
the rationality of $k(x_1,x_2)^G$ and $k(x_1,x_2,x_3)^G$ 
under the monomial actions of $G$ is known as follows: 

\begin{theorem}[{Hajja \cite{Haj87}}]\label{thHaj87}
Let $k$ be a field and $G$ be a finite group acting on $k(x_1,x_2)$ 
by monomial $k$-automorphisms. 
Then $k(x_1,x_2)^{G}$ is $k$-rational.
\end{theorem}

\begin{theorem}[{Hajja, Kang \cite{HK92,HK94}, Hoshi, Rikuna \cite{HR08}}]\label{thHKHR}
Let $k$ be a field and $G$ be a finite group acting on 
$k(x_1,x_2,x_3)$ by purely monomial $k$-automorphisms. 
Then $k(x_1,x_2,x_3)^G$ is $k$-rational.
\end{theorem}

Let $G$ be a finite group acting on $K(x_1, \ldots, x_n)$ by 
quasi-monomial $k$-automorphisms. 
There exists a group homomorphism 
$\rho_{\underline{x}}: G\rightarrow GL_n(\bZ)$ defined by 
$\rho_{\underline{x}}(\sigma)=[a_{i,j}]_{1\leq i,j\leq n}\in GL_n(\bZ)$ 
for any $\sigma\in G$ where 
$[a_{i,j}]_{1\leq i,j\leq n}$ is given in (iii) of Definition \ref{def}. 

%For quasi-monomial actions, the following results are already known 
%(see a survey \cite{Hos14}):

\begin{theorem}[Hoshi, Kang, Kitayama {\cite[Proposition 1.12]{HKK14}}]\label{thp12}
Let $k$ be a field and 
$G$ be a finite group acting on $K(x_1,\ldots,x_n)$ by
quasi-monomial $k$-automorphisms. Then there is a normal subgroup
$N$ of $G$ satisfying the following conditions:\\
{\rm (i)}
$K(x_1,\ldots,x_n)^N=K^N(y_1,\ldots,y_n)$ where each $y_i$ is of
the form $ax_1^{e_1}x_2^{e_2}\cdots x_n^{e_n}$ with $a\in
K^{\times}$ and $e_i\in \bZ$ \rm{(}we may take $a=1$ if the
action is a purely quasi-monomial action\rm{)};\\
{\rm (ii)} $G/N$
acts on $K^N(y_1,\ldots,y_n)$ by quasi-monomial $k$-automorphisms;\\
{\rm (iii)} $\rho_{\underline{y}}: G/N\to GL_n(\bZ)$ is an
injective group homomorphism.
\end{theorem}

By Theorem \ref{thp12}, 
we may assume that $\rho_{\underline{x}}: G\rightarrow GL_n(\bZ)$ 
is injective and thus $G$ may be regarded as a finite subgroup of $GL_n(\bZ)$. 
The following theorem was already proved by Prokhorov \cite{Pro10} 
when $k=\bC$. %$k$ is an algebraically closed field. 
\begin{theorem}[{Hoshi, Kitayama, Yamasaki \cite{HKY11, Yam12}}]\label{thHKY}
Let $k$ be a field with {\rm char} $k\neq 2$ 
and $G$ be a finite subgroup of $GL_3(\bZ)$ acting on $K(x_1,x_2,x_3)$ 
by monomial $k$-automorphisms. 
Then $k(x_1,x_2,x_3)^G$ is $k$-rational 
except for the $8$ cases contained in \cite{Yam12} and one
additional case. For the last exceptional case, 
$k(x_1,x_2,x_3)^G$ is also $k$-rational except for a minor
situation. 
Moreover, there exists $L=k(\sqrt{a})$ with $a\in k^\times$ such that 
$L(x_1,x_2,x_3)^G$ is $L$-rational. 
In particular, if $k$ is %an algebraically 
a quadratically closed field, %with {\rm char} $k \neq 2$, 
then $k(x_1,x_2,x_3)^G$ is $k$-rational.
\end{theorem}
Indeed, for the exceptional $8$ cases in Theorem \ref{thHKY}, 
the necessary and sufficient condition for the $k$-rationality 
of $K(x_1,x_2,x_3)^G$ was given in terms of $k$ and $c_j(\sigma)$. 
In particular, if it is not $k$-rational, then it is not retract 
$k$-rational (see \cite{Yam12}). 
We also note that %for dimension four case 
there exists %a fixed field 
$\bC(x_1,x_2,x_3,x_4)^G$ 
under the monomial action of $G$ 
which is not retract $\bC$-rational (see \cite[Example 5.11]{CHKK10}). 

\vspace*{5mm}
%%%%%%%%%%%%%%%%%%%%%%%%%%%%%%%%%%
{\it Quasi-monomial action case.} 

When we consider general quasi-monomial actions, 
it turns out that non-unirational fields $K(x)^G$ appear 
%for general quasi-monomial actions 
even in the one-dimensional case. 
Let $(a,b)_{k}$ (resp. $[a,b)_k$) 
be the norm residue symbol of degree two over $k$ 
when char $k\neq 2$ (resp. char $k=2$), see \cite[Chapter 11]{Dra83}. 
For dimension one (resp. two), the rationality problem 
for quasi-monomial (resp. purely quasi-monomial) actions was solved by 
Hoshi, Kang and Kitayama \cite{HKK14}. 
\begin{theorem}[Hoshi, Kang, Kitayama {\cite[Proposition 1.13]{HKK14}}]\label{thHKKp113}
Let $k$ be a field.\\
{\rm (1)} Let $G$ be a finite group acting on $K(x)$ by purely 
quasi-monomial $k$-automorphisms.
Then $K(x)^G$ is $k$-rational.\\
{\rm (2)} Let $G$ be a finite group acting on $K(x)$ by
quasi-monomial $k$-automorphisms. Then $K(x)^G$ is $k$-rational
except for the following case: 
There is a normal subgroup
$N$ of $G$ such that {\rm (i)} $G/N=\langle \sigma \rangle \simeq
\cC_2$, {\rm (ii)} $K(x)^N=k(\alpha)(y)$ with $\alpha^2=a\in
K^{\times}$, $\sigma(\alpha)=-\alpha$ \rm{(}if 
${\rm char}$ $k\neq 2$\rm{)}, and $\alpha^2+\alpha=a\in K$, $\sigma(\alpha)=\alpha+1$
\rm{(}if ${\rm char} k=2$\rm{)}, {\rm (iii)} $\sigma\cdot y=b/y$
for some $b\in k^{\times}$.

For the exceptional case, $K(x)^G=k(\alpha) (y)^{G/N}$ is
$k$-rational if and only if the norm residue $2$-symbol
$(a,b)_k=0$ (if ${\rm char}$ $k\neq 2$), and $[a,b)_k=0$ (if
${\rm char}$ $k=2$).

Moreover, if $K(x)^G$ is not $k$-rational, then $k$ is an infinite
field, the Brauer group $\Br(k)$ is non-trivial, and $K(x)^G$
is not $k$-unirational.
\end{theorem}

Let $\cS_n$ (resp. $\cA_n$, $\cD_n$, $\cC_n$) 
be the symmetric group (resp. the alternating group, the dihedral group, the cyclic group) 
of degree $n$ of order $n!$ (resp. $n!/2$, $2n$, $n$). 

\begin{theorem}[Hoshi, Kang, Kitayama {\cite[Theorem 1.14]{HKK14}}] 
\label{thm:HKK2}
Let $k$ be a field and 
$G$ be a finite group acting on $K(x,y)$ by purely
quasi-monomial $k$-automorphisms. 
Define $N=\{\sigma\in G:\sigma(x)=x,~ \sigma(y)=y\}$, 
$H=\{\sigma\in G:\sigma(\alpha)=\alpha$ for all $\alpha\in K\}$. 
Then $K(x,y)^G$ is $k$-rational except possibly for the following situation: 
\rm{(}1\rm{)} ${\rm char}$ $k\ne 2$ and \rm{(}2\rm{)}
$(G/N,HN/N)\simeq (\cC_4,\cC_2)$ or $(\cD_4,\cC_2)$.

More precisely, in the exceptional situation we may choose $u,v\in
k(x,y)$ satisfying that $k(x,y)^{HN/N}=k(u,v)$ (and therefore
$K(x,y)^{HN/N}=K(u,v)$) such that

{\rm (i)} when $(G/N,HN/N)\simeq (\cC_4,\cC_2)$,
$K^N=k(\sqrt{a})$ for some $a\in k\backslash k^2$, $G/N=\langle
\sigma \rangle \simeq \cC_4$, then $\sigma$ acts on $K^N(u,v)$ by
$\sigma: \sqrt{a} \mapsto -\sqrt{a}$, $u\mapsto \frac{1}{u}$,
$v\mapsto -\frac{1}{v}$; or 

{\rm (ii)} when
$(G/N,HN/N)\simeq (\cD_4,\cC_2)$, $K^N=k(\sqrt{a},\sqrt{b})$ is a
biquadratic extension of $k$ with $a,b\in k\backslash k^2$,
$G/N=\langle \sigma,\tau \rangle \simeq \cD_4$, then $\sigma$ and
$\tau$ act on $K^N(u,v)$ by $\sigma:\sqrt{a} \mapsto -\sqrt{a}$,
$\sqrt{b}\mapsto \sqrt{b}$, $u\mapsto\frac{1}{u}$, $v\mapsto
-\frac{1}{v}$, $\tau: \sqrt{a}\mapsto \sqrt{a}$, $\sqrt{b}\mapsto
-\sqrt{b}$, $u\mapsto u$, $v\mapsto -v$.

For Case {\rm (i)}, $K(x,y)^G$ is $k$-rational if and only if the
norm residue $2$-symbol $(a,-1)_k=0$. For Case {\rm (ii)},
$K(x,y)^G$ is $k$-rational if and only if $(a,-b)_k=0$.

Moreover, if $K(x,y)^G$ is not $k$-rational, then $k$ is an
infinite field, the Brauer group $\Br(k)$ is non-trivial, and
$K(x,y)^G$ is not $k$-unirational.
\end{theorem}

The following definition gives an equivalent definition of
quasi-monomial actions. 
This definition follows the approach of Saltman's
definition of twisted multiplicative actions \cite{Sal90a, Sal90b}, 
\cite[Definition 2.2]{Kan09}.

\begin{definition}[{\cite[Definition 1.15]{HKK14}}] 
Let $G$ be a finite group. A $G$-lattice $M$ is a finitely
generated $\bZ[G]$-module which is $\bZ$-free as an abelian
group, i.e.\ $M=\bigoplus_{1\le i\le n} \bZ\cdot x_i$ with a
$\bZ[G]$-module structure. 
Let $K/k$ be a field extension such
that $G$ acts on $K$ with $K^G=k$. Consider a short exact sequence
of $\bZ[G]$-modules $\alpha: 1\to K^{\times}\to M_\alpha \to
M\to 0$ where $M$ is a $G$-lattice and $K^{\times}$ is regarded as
a $\bZ[G]$-module through the $G$-action on $K$. 
The $\bZ[G]$-module structure (written multiplicatively) of
$M_\alpha$ may be described as follows: 
For each $x_j\in M$ (where $1\le j\le n$), 
take a pre-image $u_j$ of $x_j$. 
As an abelian group, $M_\alpha$ is the direct product of $K^{\times}$
and $\langle u_1, \ldots, u_n \rangle$. If $\sigma \in G$ and 
$\sigma\cdot x_j=\sum_{1\le i\le n} a_{ij} x_i \in M$, we find 
that $\sigma\cdot u_j=c_j(\sigma) \cdot \prod_{1\le i\le n} 
u_i^{a_{ij}} \in M_\alpha$ for a unique $c_j(\sigma)\in K^{\times}$ 
determined by the group extension $\alpha$. 

Using the same idea, once a group extension 
$\alpha:1\to K^{\times}\to M_\alpha \to M\to 0$ is given, we may define a 
quasi-monomial action of $G$ on the rational function field 
$K(x_1,\ldots,x_n)$ as follows: 
If $\sigma\cdot x_j=\sum_{1\le i\le n} a_{ij} x_i \in M$, 
then define $\sigma\cdot 
x_j=c_j(\sigma)\prod_{1\le i\le n} x_i^{a_{ij}} \in 
K(x_1,\ldots,x_n)$ and $\sigma\cdot \alpha =\sigma(\alpha)$ for 
$\alpha\in K$ where $\sigma(\alpha)$ is the image of $\alpha$ 
under $\sigma$ via the prescribed action of $G$ on $K$. 
This quasi-monomial action is well-defined (see \cite[page 538]{Sal90a} 
for details). The field $K(x_1,\ldots,x_n)$ with such a $G$-action 
will be denoted by $K_\alpha(M)$ to emphasize the role of the 
extension $\alpha$; its fixed field is denoted as $K_\alpha(M)^G$. 
We will say that $G$ acts on $K_\alpha(M)$ by quasi-monomial $k$-automorphisms.

If $k=K$, then $k_\alpha(M)^G$ is nothing but the fixed field 
associated to the monomial action. 

If the extension $\alpha$ splits, then we may take 
$u_1,\ldots,u_n\in M_\alpha$ satisfying that 
$\sigma\cdot u_j=\prod_{1\le i\le n} u_i^{a_{ij}}$. 
Hence the associated quasi-monomial action of $G$ on $K(x_1,\ldots,x_n)$ 
becomes a purely quasi-monomial action. 
In this case, we will write 
$K_\alpha(M)$ and $K_\alpha(M)^G$ as $K(M)$ and $K(M)^G$
respectively (the subscript $\alpha$ is omitted because the
extension $\alpha$ plays no important role). We will say that $G$
acts on $K(M)$ by purely quasi-monomial $k$-automorphisms. 
Again $k(M)^G$ is the fixed field associated to the 
purely monomial action. 
\end{definition}

As an application of Theorem \ref{thm:HKK2}, we have the following theorems: 

\begin{theorem}[Hoshi, Kang, Kitayama {\cite[Theorem 1.16]{HKK14}}]
Let $k$ be a field, $G$ be a finite group and $M$ be a $G$-lattice 
with ${\rm rank}_{\bZ} M=4$ such that $G$ acts on $k(M)$ by 
purely monomial $k$-automorphisms. 
If $M$ is decomposable, i.e. $M=M_1\oplus M_2$ as $\bZ[G]$-modules where 
$1\le {\rm rank}_{\bZ} M_1 \le 3$, then $k(M)^G$ is $k$-rational.
\end{theorem}

\begin{theorem}[Hoshi, Kang, Kitayama {\cite[Theorem 6.2]{HKK14}}]\label{th62}
Let $k$ be a field, $G$ be a finite group and $M$ be a $G$-lattice 
such that $G$ acts on $k(M)$ by purely monomial $k$-automorphisms. 
Assume that 
{\rm (i)} $M=M_1\oplus M_2$ as $\bZ[G]$-modules where 
${\rm rank}_{\bZ}M_1=3$ and ${\rm rank}_{\bZ} M_2=2$, 
{\rm (ii)} either $M_1$ or $M_2$ is a faithful $G$-lattice. 
Then $k(M)^G$ is $k$-rational except the following situation: 
{\rm char} $k\ne 2$, $G=\langle\sigma,\tau\rangle \simeq \cD_4$ 
and $M_1=\bigoplus_{1\le i\le 3} 
\bZ x_i$, $M_2=\bigoplus_{1\le j\le 2} \bZ y_j$ such that 
$\sigma:x_1\leftrightarrow x_2$, $x_3\mapsto -x_1-x_2-x_3$, 
$y_1\mapsto y_2\mapsto -y_1$, $\tau: x_1\leftrightarrow x_3$, 
$x_2\mapsto -x_1-x_2-x_3$, $y_1\leftrightarrow y_2$ where the 
$\bZ[G]$-module structure of $M$ is written additively. 
For the exceptional case, $k(M)^G$ is not retract $k$-rational. 
\end{theorem}

The aim of this paper is to investigate the rationality problem 
of $K(x_1,x_2,x_3)^G$ for purely quasi-monomial $k$-automorphisms 
of dimension three. 
The followings are main results of this paper. 

Note that (i) by Theorem \ref{thp12}, 
we may assume $N=\{\sigma\in G:\sigma(x_i)=x_i$ $(i=1,2,3)\}=1$ 
and hence $G$ may be regarded as a finite subgroup of $GL_3(\bZ)$; 
(ii) when $H=1$ the answer to the rationality problem in dimension $3$ 
was given by Kunyavskii (see Theorem \ref{thm:Kun}). 

There exist $73$ finite subgroups $G_{i,j,k}$ $(1\leq i\leq 7)$ 
contained in $GL_3(\bZ)$ up to conjugation 
which are classified into $7$ crystal systems 
(see Section \ref{sec:notation} for details). 

\begin{theorem}[The groups $G$ do not belong to 
the $7$th crystal system in dimension $3$]\label{thmain1} 
Let $k$ be a field with {\rm char} $k\neq 2$ 
and $G$ be a finite subgroup of $GL_3(\bZ)$ acting on $K(x_1,x_2,x_3)$ 
by purely quasi-monomial $k$-automorphisms. 
Assume that $G$ does not belong to 
the $7$th crystal system in dimension $3$ 
and $H=\{\sigma\in G\mid \sigma(\alpha)=\alpha\ {\rm for\ any\ } \alpha\in K\}\neq 1$.\\
{\rm (1)} If $G$ does not belong to 
the $4$th crystal system in dimension $3$, i.e. 
$G$ is either not a $2$-group or a $2$-group of exponent $2$, 
then $K(x_1,x_2,x_3)^G$ is $k$-rational;\\
{\rm (2)} When $G$ belongs to the $4$th crystal system 
in dimension $3$, i.e. $G$ is a $2$-group of exponent $4$, 
$G$ is $\bQ$-conjugate to one of the following $8$ groups: 
$\langle\pm\caa\rangle\simeq\cC_4$, 
$\langle\caa,-I_3\rangle\simeq\cC_4\times\cC_2$, 
$\langle\pm\caa,\pm\la_1\rangle\simeq\cD_4$, 
$\langle\caa,\la_1,-I_3\rangle\simeq\cD_4\times\cC_2$ where 
\[
\caa=\begin{bmatrix} 0&-1&0\\1&0&0\\0&0&1 \end{bmatrix},\ 
\la_1=\begin{bmatrix} -1&0&0\\0&1&0\\0&0&-1 \end{bmatrix} 
\]
and $I_3$ is the $3\times 3$ identity matrix. 
Then $K(x_1,x_2,x_3)^G$ is $k$-rational except possibly for the 
following cases 
with $H=\langle\caa^2\rangle$ or $\langle\caa^2,-I_3\rangle$:

{\rm (i)} 
The case where $H=\langle\caa^2\rangle\simeq \cC_2$. 

{\rm (i-1)} 
When $G$ is $\bQ$-conjugate to $\langle\pm\caa\rangle$, 
$($resp. $\langle\caa,-I_3\rangle)$
we may take $K=k(\sqrt{a})$ 
$($resp. $K^{\langle -I_3\rangle}=k(\sqrt{a}))$ 
on which $G$ acts by $\pm\caa : \sqrt{a}\mapsto -\sqrt{a}$, 
and $K(x_1,x_2,x_3)^G$ is $k$-rational if and only if 
$(a,-1)_k=0$. 

{\rm (i-2)} 
When $G$ is $\bQ$-conjugate to 
$\langle\pm\caa,\pm\la_1\rangle$ 
$($resp. $\langle\caa,\la_1,-I_3\rangle)$
we may take $K=k(\sqrt{a},\sqrt{b})$ 
$($resp. $K^{\langle -I_3\rangle}=k(\sqrt{a},\sqrt{b}))$ on which 
$G$ acts by $\pm\caa : \sqrt{a}\mapsto -\sqrt{a}, \sqrt{b}\mapsto \sqrt{b}$, 
$\pm\la_1 : \sqrt{a}\mapsto\sqrt{a}, \sqrt{b}\mapsto-\sqrt{b}$, 
and $K(x_1,x_2,x_3)^G$ is $k$-rational if and only if $(a,-b)_k=0$. 

{\rm (ii)} 
The case where $H=\langle\caa^2,-I_3\rangle\simeq \cC_2\times\cC_2$. 

{\rm (ii-1)} 
When $G$ is $\bQ$-conjugate to $\langle\caa,-I_3\rangle$, 
we may take $K=k(\sqrt{a})$  
on which $G$ acts by $\caa : \sqrt{a}\mapsto -\sqrt{a}$, 
and $K(x_1,x_2,x_3)^G$ is $k$-rational if and only if 
$(a,-1)_k=0$. 

{\rm (ii-2)}
When $G$ is $\bQ$-conjugate to $\langle\caa,\la_1,-I_3\rangle$, 
we may take $K=k(\sqrt{a},\sqrt{b})$ on which 
$G$ acts by $\caa : \sqrt{a}\mapsto -\sqrt{a}, \sqrt{b}\mapsto \sqrt{b}$, 
$\la_1 : \sqrt{a}\mapsto\sqrt{a}, \sqrt{b}\mapsto-\sqrt{b}$, 
and $K(x_1,x_2,x_3)^G$ is $k$-rational if and only if 
$(a,-b)_k=0$. 

Moreover, if $K(x_1,x_2,x_3)^G$ is not $k$-rational, then $k$ is an
infinite field, the Brauer group $\Br(k)$ is non-trivial, and
$K(x_1,x_2,x_3)^G$ is not $k$-unirational.

In particular, the $k$-rationality of $K(x_1,x_2,x_3)^G$ 
does not depend on the $\bQ$-conjugacy class of $G$ and 
the sign of $\pm\caa$ and $\pm\la_1$. 
\end{theorem} 

There exist $15$ finite subgroups $G_{7,j,k}$ 
$(1\leq j\leq 5, 1\leq k\leq 3)$ of $GL_3(\bZ)$ 
which belong to the $7$th crystal system (see Section \ref{sec:notation}). 

\begin{theorem}[The groups $G$ belong to the $7$th crystal system in dimension $3$]
\label{thmain2}
Let $k$ be a field with {\rm char} $k\neq 2$ 
and $G$ be a finite subgroup of $GL_3(\bZ)$ acting on 
$K(x_1,x_2,x_3)$ by purely quasi-monomial $k$-automorphisms.
Assume that $G=G_{7,j,k}$ 
belongs to the $7$th crystal system in dimension $3$ 
and $H=\{\sigma\in G\mid \sigma(\alpha)=\alpha\ {\rm for\ any\ } \alpha\in K\}\neq 1$.\\
{\rm (1)} If $G=G_{7,j,1}$ $(1\leq j\leq 5)$, $G_{7,j,2}$ $(1\leq j\leq 5)$ or 
$G_{7,j,3}$ $(j=1,4)$, 
then $K(x_1,x_2,x_3)^G$ is $k$-rational;\\
{\rm (2)} When $G=G_{7,j,3}$ $(j=2,3,5)$, 
$G_{7,2,3}=\langle \ta_3,\la_3,\cb,-I_3\rangle$ $\simeq$ $\cA_4\times \cC_2$, 
$G_{7,3,3}=\langle \ta_3,\la_3,\cb,-\be_3\rangle$ $\simeq$ $\cS_4$, 
$G_{7,5,3}=\langle \ta_3,\la_3,\cb,\be_3,-I_3\rangle$ $\simeq$ $\cS_4\times \cC_2$ 
where 
\begin{align*}
\ta_3=\left[\begin{array}{ccc}0 & 1 & -1\\ 1 & 0 & -1\\ 0 & 0 & -1\end{array}\right],\ 
\la_3=\left[\begin{array}{ccc}0 & -1 & 1\\ 0 & -1 & 0\\ 1 & -1 & 0\end{array}\right],\ 
\cb=\left[\begin{array}{ccc} 0 & 0 & 1\\ 1 & 0 & 0\\ 0 & 1 & 0\end{array}\right],\ 
\be_3=\left[\begin{array}{ccc}1 & 0 & -1\\ 0 & 1 & -1\\ 0 & 0 & -1\end{array}\right]
\end{align*}
and $I_3$ is the $3\times 3$ identity matrix. 
Then $K(x_1,x_2,x_3)^G$ is 
$k$-rational except possibly for the following cases with 
$H=\langle\ta_3,\la_3\rangle$ or $\langle\ta_3,\la_3,\cb\rangle$: 

{\rm (i)} 
The case where $H=\langle\ta_3,\la_3\rangle\simeq \cC_2\times \cC_2$. 
We have 
\begin{align*}
K(x_1,x_2,x_3)^{G_{7,2,3}}&=K(u_1,u_2,u_3)^{\langle\cb,-I_3\rangle},\\ 
K(x_1,x_2,x_3)^{G_{7,3,3}}&=K(u_1,u_2,u_3)^{\langle\cb,-\beta_3\rangle},\\ 
K(x_1,x_2,x_3)^{G_{7,5,3}}&=K(u_1,u_2,u_3)^{\langle\cb,\beta_3,-I_3\rangle}
\end{align*}
where $K(u_1,u_2,u_3)=K(x_1,x_2,x_3)^H$ and 
\begin{align*}
\cb &: 
u_1\mapsto  u_2,\ u_2\mapsto  u_3,\ u_3\mapsto u_1,\\
-\be_3 &: 
u_1\mapsto \frac{-u_1+u_2+u_3}{u_2u_3},\ 
u_2\mapsto \frac{u_1+u_2-u_3}{u_1u_2},\ 
u_3\mapsto \frac{u_1-u_2+u_3}{u_1u_3},\\
\be_3 &: 
u_1\mapsto u_1,\ u_2\mapsto u_3,\ u_3\mapsto u_2,\\
-I_3 &: 
u_1\mapsto \frac{-u_1+u_2+u_3}{u_2u_3},\ 
u_2\mapsto \frac{u_1-u_2+u_3}{u_1u_3},\ 
u_3\mapsto \frac{u_1+u_2-u_3}{u_1u_2}.
\end{align*}

{\rm (ii)} 
The case where $H=\langle\ta_3,\la_3,\cb\rangle\simeq \cA_4$. 
We have
\begin{align*}
K(x_1,x_2,x_3)^{G_{7,2,3}}&=K(s_1,s_2,s_3)^{\langle-I_3\rangle},\\
K(x_1,x_2,x_3)^{G_{7,3,3}}&=K(s_1,s_2,s_3)^{\langle-\beta_3\rangle},\\
K(x_1,x_2,x_3)^{G_{7,5,3}}&=K(s_1,s_2,s_3)^{\langle\beta_3,-I_3\rangle}
\end{align*}
where $K(s_1,s_2,s_3)=K(x_1,x_2,x_3)^H$ and 
\begin{align*}
-\be_3 &: s_1\mapsto s_1,\ 
s_2\mapsto \frac{1+3s_1^2}{s_2},\ 
s_3\mapsto \frac{-1-6s_1^2-9s_1^4+2s_2+10s_1^2s_2
+4s_1^4s_2-s_2^2-3s_1^2s_2^2}{s_2s_3},\nonumber\\
\be_3 &: s_1\mapsto -s_1,\ s_2\mapsto s_2,\ s_3\mapsto s_3,\label{acts}\\
-I_3 &: s_1\mapsto -s_1,\ 
s_2\mapsto \frac{1+3s_1^2}{s_2},\ 
s_3\mapsto \frac{-1-6s_1^2-9s_1^4+2s_2+10s_1^2s_2
+4s_1^4s_2-s_2^2-3s_1^2s_2^2}{s_2s_3}.\nonumber
\end{align*}
\end{theorem} 

We do not know the rationality of $K(x_1,x_2,x_3)^G$ 
in Theorem \ref{thmain2} (2) (i) with $H\simeq\cC_2\times\cC_2$. 
For the case (2) (ii) of Theorem \ref{thmain2} with $H\simeq \cA_4$, 
we will give the following partial result. 
It turns out that the fixed field $K(x_1,x_2,x_3)^G$ has a conic bundle structure. 

\begin{proposition}\label{prop1}
Let $k$ be a field with {\rm char} $k\neq 2$ and $G=G_{7,j,3}$ $(j=2,3,5)$. 
Assume that $H=\{\sigma\in G\mid \sigma(\alpha)=\alpha\ {\rm for\ any\ } \alpha\in K\}
=\langle\ta_3,\la_3,\cb\rangle\simeq \cA_4$.\\ 
{\rm (1)} $G=G_{7,2,3}=\langle H,-I_3\rangle$. 
There exist $w_1,w_2,w_3\in K(x_1,x_2,x_3)^H$ such that 
$K(x_1,x_2,x_3)^H=K(w_1,w_2,w_3)$ and 
\begin{align*}
-I_3 :&\ w_1\mapsto w_1,\ w_2\mapsto -w_2,\\ 
&\ w_3\mapsto 
\frac{w_1(1+w_1+w_1^2)^2-(2+3w_1+4w_1^2+2w_1^3)w_2^2+(w_1+2)w_2^4}{w_3}.
\end{align*}
In particular, we have 
$K(x_1,x_2,x_3)^{G}=k(X,Y,W_1,W_2)$ 
where $K=k(\sqrt{b})$ and 
\begin{align*}
X^2-bY^2=W_1(1+W_1+W_1^2)^2-b(2+3W_1+4W_1^2+2W_1^3)W_2^2
+b^2(W_1+2)W_2^4.
\end{align*}
{\rm (2)} 
$G=G_{7,3,3}=\langle H,-\be_3\rangle$. 
There exist $u_1,u_2,u_3\in K(x_1,x_2,x_3)^H$ such that 
$K(x_1,x_2,x_3)^H=K(u_1,u_2,u_3)$ and 
\begin{align*}
-\be_3 : u_1\mapsto u_1,\ u_2\mapsto -u_2,\ 
u_3\mapsto \frac{2(5-u_2^2)(u_1^2-u_2^2+1)}{u_3}.
\end{align*}
We have $K(x_1,x_2,x_3)^G=k(X,Y,U_1,U_2)$ where 
$X^2-dY^2=2(5-dU_2^2)(U_1^2-dU_2^2+1)$ and $K=k(\sqrt{d})$. 
Moreover, if $\sqrt{5}\in k$, 
then the following conditions are equivalent:\\
{\rm (i)} 
$K(x_1,x_2,x_3)^G$ is $k$-rational;\\
{\rm (ii)} 
$K(x_1,x_2,x_3)^G$ is $k$-unirational;\\
{\rm (iii)} 
$X^2-dY^2-2v_0^2-2v_1^2+2dv_2^2$ has a non-trivial $k$-zero 
with $K=k(\sqrt{d})$.\\
In particular, if $\sqrt{5}, \sqrt{2}\in k$ or $\sqrt{5}, \sqrt{-1}\in k$, 
then $K(x_1,x_2,x_3)^G$ is $k$-rational.\\
{\rm (3)} $G=G_{7,5,3}=\langle H,\be_3,-I_3\rangle$. 
There exist $p_1,p_2,p_3\in K(x_1,x_2,x_3)^H$ such that 
$K(x_1,x_2,x_3)^H$ $=$ $K(p_1,p_2,p_3)$ and 
\begin{align*}
\be_3 &: %\sqrt{a}\mapsto -\sqrt{a}, 
p_1\mapsto -p_1,\ p_2\mapsto -p_2,\ p_3\mapsto -p_3,\\
-I_3 &: %\sqrt{b}\mapsto -\sqrt{b}, 
p_1\mapsto p_1,\ p_2\mapsto -p_2,\ 
p_3\mapsto \frac{-1-5p_1^2-7p_2^2-(p_1^2-p_2^2)(3p_1^2+17p_2^2)+9(p_1^2-p_2^2)^3}{p_3}.
\end{align*}
In particular, we have 
$K(x_1,x_2,x_3)^{G}=k(X,Y,P_1,P_2)$ 
where $K=k(\sqrt{a},\sqrt{b})$ and 
\begin{align*}
X^2-bY^2=-\tfrac{1}{a}-5P_1^2-7bP_2^2-a(P_1^2-bP_2^2)(3P_1^2+17bP_2^2)+9a^2(P_1^2-bP_2^2)^3.
\end{align*}
{\rm (4)} For $G=G_{7,2,3}=\langle H,-I_3\rangle$, 
we assume that $\sqrt{-3}\in k$. 
Then there exist $t_1,t_2,t_3\in K(x_1,x_2,x_3)^H$ such that 
$K(x_1,x_2,x_3)^H=K(t_1,t_2,t_3)$ and 
\begin{align*}
-I_3 : t_1\mapsto -t_1,\ t_2\mapsto t_2,\ 
t_3\mapsto \frac{(t_1^2+4)(t_1^2-t_2^2+1)}{t_3}.
\end{align*}
In particular, if $\sqrt{-3},\sqrt{-1}\in k$, then 
$K(x_1,x_2,x_3)^G$ is $k$-rational.\\
{\rm (5)} If {\rm char} $k=3$, then 
$K(x_1,x_2,x_3)^G$ is $k$-rational for $G=G_{7,j,3}$ $(j=2,3,5)$. 
\end{proposition}

As an application of Theorems \ref{thmain1} and \ref{thmain2} and 
Proposition \ref{prop1}, 
we get the following theorem which complements to Theorem \ref{th62}: 

\begin{theorem}\label{thmain3}
Let $k$ be a field with {\rm char} $k\neq 2$, 
$G$ be a finite group and $M$ be a $G$-lattice 
such that $G$ acts on $k(M)$ by purely monomial $k$-automorphisms. 
Assume that 
{\rm (i)} $M=M_1\oplus M_2$ as $\bZ[G]$-modules where
${\rm rank}_{\bZ}M_1=3$ and ${\rm rank}_{\bZ} M_2=2$, 
{\rm (ii)} both $M_1$ and $M_2$ are not faithful $G$-lattices. 
Let $N_i=\{\sigma\in G\mid\sigma(\alpha)=\alpha\ 
{\rm for\ any}\ \alpha \in k(M_i)\}$ $(i=1,2)$. 
Then $k(M)^G$ is $k$-rational except 
the following situation: 
$(G/N_1,N_1N_2/N_1,G/N_2,N_1N_2/N_2)\simeq (G_{7,2,3},\cA_4,\cC_4,\cC_2)$, 
$(G_{7,3,3},\cA_4,\cC_4,\cC_2)$ and $(G_{7,5,3},\cA_4,\cD_4,\cC_2)$. 
Moreover, if {\rm char} $k=3$, then $k(M)^G$ is $k$-rational. 
\end{theorem}

This paper is organized as follows. 
In Section \ref{sec:notation}, 
we recall the classification of subgroups of $GL_3(\bZ)$ up to conjugation. 
Section \ref{sepre} contains some rationality results which will be used 
in %the remaining part of 
the paper. 
A technique of conjugacy classes move, which is described in 
Subsection \ref{subsec:conversion}, is useful and will 
be used in the proof of Theorems \ref{thmain1} and \ref{thmain2} 
to the cases of 3rd crystal system (II), 
the 4th crystal system (II) and 
7th crystal system (II) in 
Subsections \ref{subsec:3}, \ref{subsec:4} and \ref{subsec:71} respectively. 
Section \ref{seP1} contains the proof of Theorem \ref{thmain1}. 
The proof of Theorem \ref{thmain2} is given in Section \ref{seP2}. 
In Section \ref{seP3}, the proof of Proposition \ref{prop1} will be given. 
We will prove Theorem \ref{thmain3} in Section \ref{seP4}.

\begin{acknowledgments}
The authors thank Ming-chang Kang for many helpful comments 
and valuable suggestions. 
They also thank the referee for careful reading %of the manuscript 
and useful comments. 
\end{acknowledgments}

%%%%%%%%%%%%%%%%%%%%%%%%%%%%%%%%%%%%%%%%%%%%%%%%%%%%%%%%%%%%%%%%%%%
%%%%%%%%%%%%%%%%%%%%%%%%%%%%%%%%%%%%%%%%%%%%%%%%%%%%%%%%%%%%%%%%%%%
%
\section{Notation} \label{sec:notation}  
%
%%%%%%%%%%%%%%%%%%%%%%%%%%%%%%%%%%%%%%%%%%%%%%%%%%%%%%%%%%%%%%%%%%%
%%%%%%%%%%%%%%%%%%%%%%%%%%%%%%%%%%%%%%%%%%%%%%%%%%%%%%%%%%%%%%%%%%%

Let $\cS_n$ (resp. $\cA_n$, $\cD_n$, $\cC_n$) 
be the symmetric group (resp. the alternating group, the dihedral group, the cyclic group) 
of degree $n$ of order $n!$ (resp. $n!/2$, $2n$, $n$). 
Let $I_3$ be the $3\times 3$ identity matrix. 
We define the following matrices: 
\begin{align*}
\ca&=\left[\begin{array}{ccc} 0 & -1 & 0\\ 1 & -1 & 0\\ 0 & 0 & 1\end{array}\right],& 
\cb&=\left[\begin{array}{ccc} 0 & 0 & 1\\ 1 & 0 & 0\\ 0 & 1 & 0\end{array}\right],&
\caa&= \left[\begin{array}{ccc} 0 & -1 & 0\\ 1 & 0 & 0\\ 0 & 0 & 1\end{array}\right], & 
\cbb&= \left[\begin{array}{ccc} 0 & 1 & 0\\ 0 & 1 & -1\\ -1 & 1 & 0\end{array}\right], 
\end{align*} 
\allowdisplaybreaks[4] 
\begin{align*}
\ta_1&=\left[\begin{array}{ccc}-1 & 0 & 0\\ 0 & -1 & 0\\ 0 & 0 & 1\end{array}\right],&
\la_1&=\left[\begin{array}{ccc}-1 & 0 & 0\\ 0 & 1 & 0\\ 0 & 0 & -1\end{array}\right],&
\be_1&=\left[\begin{array}{ccc}0 & -1 & 0\\ -1 & 0 & 0\\ 0 & 0 & 1\end{array}\right],\\
\ta_2&=\left[\begin{array}{ccc}0 & 1 & 0\\ 1 & 0 & 0\\ -1 & -1 & -1\end{array}\right],&
\la_2&=\left[\begin{array}{ccc}0 & 0 & 1\\ -1 & -1 & -1\\ 1 & 0 & 0\end{array}\right],&
\be_2&=\left[\begin{array}{ccc}1 & 0 & 0\\ 0 & 1 & 0\\ -1 & -1 & -1\end{array}\right],\\
\ta_3&=\left[\begin{array}{ccc}0 & 1 & -1\\ 1 & 0 & -1\\ 0 & 0 & -1\end{array}\right],&
\la_3&=\left[\begin{array}{ccc}0 & -1 & 1\\ 0 & -1 & 0\\ 1 & -1 & 0\end{array}\right],&
\be_3&=\left[\begin{array}{ccc}1 & 0 & -1\\ 0 & 1 & -1\\ 0 & 0 & -1\end{array}\right], \\ 
\alpha&=\left[\begin{array}{ccc} 0 & 1 & 0\\ 1 & 0 & 0\\ 0 & 0 & 1\end{array}\right], &
-I_3 &=\left[\begin{array}{ccc} -1 & 0 & 0\\ 0 & -1 & 0\\ 0 & 0 & -1\end{array}\right].
\end{align*}
There exist exactly $73$ finite subgroups 
contained in $GL_3(\bZ)$ up to conjugation 
which are classified into $7$ crystal systems 
(see \cite[Table 1]{BBNWZ78}). \\

%%%%%%%%%%%%%%%%%%%%%%%%%

\noindent 
The 1st crystal system ($\bZ$-reducible): 
\begin{align*}
G_{1,1,1}=&\{I_3\},&  G_{1,2,1}&=\langle-I_3\rangle\simeq\cC_2.
\end{align*} 

%%%%%%%%%%%%%%%%%%%%%%%%%%

\noindent 
The 2nd crystal system ($\bZ$-reducible): 
\begin{align*}
G_{2,1,1}&=\langle \lambda_1\rangle\simeq\cC_2,& 
G_{2,1,2}&=\langle -\alpha\rangle\simeq\cC_2,\\
G_{2,2,1}&=\langle -\lambda_1\rangle\simeq\cC_2,& 
G_{2,2,2}&=\langle \alpha\rangle\simeq\cC_2,\\
G_{2,3,1}&=\langle \lambda_1,-I_3\rangle\simeq\cC_2\times \cC_2,&
G_{2,3,2}&=\langle -\alpha,-I_3\rangle\simeq\cC_2\times \cC_2. 
\end{align*} 

%%%%%%%%%%%%%%%%%%%%%%%%%%%%%%%%%%%

\noindent 
The 3rd crystal system (I) ($\bZ$-reducible): 
\begin{align*}
G_{3,1,1}&=\langle \tau_1,\lambda_1\rangle\simeq\cC_2\times \cC_2,& 
G_{3,1,2}&=\langle \tau_1,-\alpha\rangle\simeq\cC_2\times \cC_2,\\
G_{3,2,1}&=\langle \tau_1,-\lambda_1\rangle\simeq\cC_2\times \cC_2,&
G_{3,2,2}&=\langle \tau_1,\alpha\rangle\simeq\cC_2\times \cC_2,\\  
G_{3,2,3}&=\langle -\alpha, \beta_1 \rangle\simeq\cC_2\times\cC_2, \\ 
G_{3,3,1}&=\langle \tau_1,\lambda_1,-I_3\rangle\simeq\cC_2\times 
\cC_2\times \cC_2, &
G_{3,3,2}&=\langle \tau_1,-\alpha,-I_3\rangle\simeq\cC_2\times 
\cC_2\times \cC_2.
\end{align*} 

%%%%%%%%%%%%%%%%%%%%%%%%%%%%%%%%%%%

\noindent 
The 3rd crystal system (II) 
($\bZ$-irreducible, but $\bQ$-reducible):
\begin{align*}
G_{3,1,3}&=\langle \tau_2,\lambda_2\rangle\simeq\cC_2\times \cC_2,& 
{}^*G_{3,1,4}&=\langle \tau_3,\lambda_3\rangle\simeq\cC_2\times \cC_2,\\
G_{3,2,4}&=\langle \tau_2,-\lambda_2\rangle\simeq\cC_2\times \cC_2,& 
G_{3,2,5}&=\langle \tau_3,-\lambda_3\rangle\simeq\cC_2\times \cC_2,\\ 
{}^*G_{3,3,3}&=\langle \tau_2,\lambda_2,-I_3\rangle\simeq\cC_2\times 
\cC_2\times \cC_2,& 
{}^*G_{3,3,4}&=\langle \tau_3,\lambda_3,-I_3\rangle\simeq\cC_2\times 
\cC_2\times \cC_2.
\end{align*} 

%%%%%%%%%%%%%%%%%%%%%%%%%%%%%%%%%%%

\noindent 
The 4th crystal system (I) ($\bZ$-reducible): 
\begin{align*}
G_{4,1,1}&=\langle \caa\rangle\simeq\cC_4,& 
G_{4,2,1}&=\langle -\caa\rangle\simeq\cC_4,\\ 
G_{4,3,1}&=\langle \caa,-I_3\rangle\simeq\cC_4\times \cC_2,&
G_{4,4,1}&=\langle \caa,\lambda_1\rangle\simeq\cD_4,\\ 
G_{4,5,1}&=\langle \caa,-\lambda_1\rangle\simeq\cD_4,&  
G_{4,6,1}&=\langle -\caa,\lambda_1\rangle\simeq\cD_4, \\   
G_{4,6,2}&=\langle -\caa,-\lambda_1\rangle\simeq\cD_4, &  
G_{4,7,1}&=\langle \caa,\lambda_1,-I_3\rangle\simeq\cD_4\times \cC_2.
\end{align*} 

%%%%%%%%%%%%%%%%%%%%%%%%%%%%%%%%%%%

\noindent 
The 4th crystal system (II) ($\bZ$-irreducible, but $\bQ$-reducible): 
\begin{align*}
G_{4,1,2}&=\langle \cbb\rangle\simeq\cC_4,&
G_{4,2,2}&=\langle -\cbb\rangle\simeq\cC_4,\\
{}^*G_{4,3,2}&=\langle \cbb,-I_3\rangle\simeq\cC_4\times \cC_2,&
{}^*G_{4,4,2}&=\langle \cbb,\lambda_3\rangle\simeq\cD_4,\\ 
G_{4,5,2}&=\langle \cbb,-\lambda_3\rangle\simeq\cD_4, & 
G_{4,6,3}&=\langle -\cbb,-\lambda_3\rangle\simeq\cD_4, \\  
{}^*G_{4,6,4}&=\langle -\cbb,\lambda_3\rangle\simeq\cD_4, &  
{}^*G_{4,7,2}&=\langle \cbb,\lambda_3,-I_3\rangle\simeq\cD_4\times \cC_2.
\end{align*} 

%%%%%%%%%%%%%%%%%%%%%%%%%%%%%%%%%%%

\noindent 
The 5th crystal system (I) ($\bZ$-irreducible, but $\bQ$-reducible): 
\begin{align*}
G_{5,1,1}&=\langle \cb\rangle\simeq\cC_3,& 
G_{5,2,1}&=\langle \cb,-I_3\rangle\simeq\cC_6,\\ 
G_{5,3,1}&=\langle \cb,-\al\rangle\simeq\cS_3,&
G_{5,4,1}&=\langle \cb,\al\rangle\simeq\cS_3, \\ 
G_{5,5,1}&=\langle \cb,-\al,-I_3\rangle\simeq\cD_6.
\end{align*}

%%%%%%%%%%%%%%%%%%%%%%%%%%%%%%%%%%%

\noindent 
The 5th crystal system (II) ($\bZ$-reducible): 
\begin{align*}
G_{5,1,2}&=\langle \ca\rangle\simeq\cC_3, & 
G_{5,2,2}&=\langle \ca,-I_3\rangle\simeq\cC_6,\\
G_{5,3,2}&=\langle \ca,-\al\rangle\simeq\cS_3,& 
G_{5,3,3}&=\langle \ca,-\be_1\rangle\simeq\cS_3,\\
G_{5,4,2}&=\langle \ca,\be_1\rangle\simeq\cS_3,& 
G_{5,4,3}&=\langle \ca,\al\rangle\simeq\cS_3,\\
G_{5,5,2}&=\langle \ca,-\al,-I_3\rangle\simeq\cD_6,& 
G_{5,5,3}&=\langle \ca,-\be_1,-I_3\rangle\simeq\cD_6.
\end{align*}

%%%%%%%%%%%%%%%%%%%%%%%%%%%%%%%%%%%

\noindent 
The 6th crystal system ($\bZ$-reducible): 
\begin{align*}
G_{6,1,1}&=\langle \ca,\ta_1\rangle\simeq\cC_6,&
G_{6,2,1}&=\langle \ca,-\ta_1\rangle\simeq\cC_6,\\
G_{6,3,1}&=\langle \ca,\ta_1,-I_3\rangle\simeq\cC_6\times\cC_2,&
G_{6,4,1}&=\langle \ca,\ta_1,-\be_1\rangle\simeq\cD_6,\\  
G_{6,5,1}&=\langle \ca,\ta_1,\be_1\rangle\simeq\cD_6,& 
G_{6,6,1}&=\langle \ca,-\ta_1,\be_1\rangle\simeq\cD_6,\\ 
G_{6,6,2}&=\langle \ca,-\ta_1,-\be_1\rangle\simeq\cD_6,&
G_{6,7,1}&=\langle \ca,\ta_1,-\be_1,-I_3\rangle\simeq\cD_6\times\cC_2.
\end{align*}

%%%%%%%%%%%%%%%%%%%%%%%%%%%%%%%%%%%

\noindent 
The 7th crystal system (I) ($\bQ$-irreducible): 
\begin{align*}
G_{7,1,1}&=\langle \ta_1,\la_1,\cb\rangle\simeq \cA_4, &
G_{7,2,1}&=\langle \ta_1,\la_1,\cb,-I_3\rangle\simeq \cA_4\times \cC_2,\\
G_{7,3,1}&=\langle \ta_1,\la_1,\cb,-\be_1\rangle\simeq \cS_4, &
G_{7,4,1}&=\langle \ta_1,\la_1,\cb,\be_1\rangle\simeq \cS_4, \\
G_{7,5,1}&=\langle \ta_1,\la_1,\cb,\be_1,-I_3\rangle\simeq \cS_4\times \cC_2.
\end{align*}

%%%%%%%%%%%%%%%%%%%%%%%%%%%%%%%%%%%

\noindent 
The 7th crystal system (II) ($\bQ$-irreducible): 
\begin{align*}
G_{7,1,2}&=\langle \ta_2,\la_2,\cb\rangle\simeq \cA_4, &
{}^*G_{7,2,2}&=\langle \ta_2,\la_2,\cb,-I_3\rangle\simeq \cA_4\times \cC_2,\\
{}^*G_{7,3,2}&=\langle \ta_2,\la_2,\cb,-\be_2\rangle\simeq \cS_4, &
G_{7,4,2}&=\langle \ta_2,\la_2,\cb,\be_2\rangle\simeq \cS_4, \\
{}^*G_{7,5,2}&=\langle \ta_2,\la_2,\cb,\be_2,-I_3\rangle\simeq \cS_4\times \cC_2.
\end{align*}

%%%%%%%%%%%%%%%%%%%%%%%%%%%%%%%%%%%

\noindent 
The 7th crystal system (III) ($\bQ$-irreducible): 
\begin{align*}
{}^*G_{7,1,3}&=\langle \ta_3,\la_3,\cb\rangle\simeq \cA_4, &
{}^*G_{7,2,3}&=\langle \ta_3,\la_3,\cb,-I_3\rangle\simeq \cA_4\times \cC_2,\\
{}^*G_{7,3,3}&=\langle \ta_3,\la_3,\cb,-\be_3\rangle\simeq \cS_4, & 
{}^*G_{7,4,3}&=\langle \ta_3,\la_3,\cb,\be_3\rangle\simeq \cS_4, \\
{}^*G_{7,5,3}&=\langle \ta_3,\la_3,\cb,\be_3,-I_3\rangle\simeq \cS_4\times \cC_2.
\end{align*}

The $15$ groups ${}^*G_{i,j,k}$ with marked $*$ above are 
in the Kunyavskii's list of \cite[Theorem 1]{Kun87}, i.e. 
the only $15$ groups in $GL_3(\bZ)$ up to conjugation 
for which %the function field 
$K(x_1,x_2,x_3)^{G_{i,j,k}}$ %of $3$-dimensional algebraic tori 
are not retract $k$-rational 
under the faithful action of $G_{i,j,k}$ on $K$ (see Theorem \ref{thm:Kun}). 

%%%%%%%%%%%%%%%%%%%%%%%%%%%%%%%%%%%%%%%%%%%%%%%%%%%%%%%%%%%%%%%%%%%
%%%%%%%%%%%%%%%%%%%%%%%%%%%%%%%%%%%%%%%%%%%%%%%%%%%%%%%%%%%%%%%%%%%
%
\section{Preliminaries}\label{sepre}
%
%%%%%%%%%%%%%%%%%%%%%%%%%%%%%%%%%%%%%%%%%%%%%%%%%%%%%%%%%%%%%%%%%%%
%%%%%%%%%%%%%%%%%%%%%%%%%%%%%%%%%%%%%%%%%%%%%%%%%%%%%%%%%%%%%%%%%%%

%%%%%%%%%%%%%%%%%%%%%%%%%%%%%%%%%%%%%%%%%%%%%%%%%%%%%%%%%%%%%%%%%%%
%
\subsection{Reduction to lower degree} 
%
%%%%%%%%%%%%%%%%%%%%%%%%%%%%%%%%%%%%%%%%%%%%%%%%%%%%%%%%%%%%%%%%%%%

\begin{theorem}[{Miyata \cite[Lemma, page 70]{Miy71}, 
Ahmad, Hajja, Kang \cite[Theorem 3.1]{AHK00}}]\label{thAHK} 
Let $L$ be a field, $L(x)$ be the rational function field with 
one variable $x$ over $L$ and $G$ be a finite group acting on $L(x)$. 
Suppose that, for any $\sigma\in G$, 
$\sigma(L)\subset L$ and $\sigma(x)=a_{\sigma}x+b_{\sigma}$ 
where $a_{\sigma}$, $b_{\sigma} \in L$ and $a_{\sigma}\neq 0$. 
Then $L(x)^G=L^G(f)$ for some polynomial $f\in L[x]^G$. 
\end{theorem} 

\begin{corollary} \label{cor:AHK} 
Let $k$ be a field with $\mathrm{char}$ $k\neq 2$ 
and $G$ be a finite subgroup of $GL_n(\bZ)$ 
acting on $K(x_1, \ldots, x_n)$ by purely quasi-monomial $k$-automorphisms. 
If $G$ is $\bZ$-reducible of $(n-1,1)$-type and 
$K(x_1, \ldots, x_{n-1})^G$ is $k$-rational, 
then $K(x_1, \ldots, x_n)^G$ is $k$-rational. 
\end{corollary} 
\begin{proof} 
Put $x_n'=(x_n-1)/(x_n+1)$ and apply Theorem \ref{thAHK}. 
\end{proof} 

The following lemmas are a restatement of Hilbert $90$ 
(see also \cite[Lemma, page 70]{Miy71}). 
%and \cite[Proposition 1.1]{EM73}). 

\begin{theorem}[Endo, Miyata {\cite[Proposition 1.1]{EM73}}]\label{thEM}
Let $L/k$ be a finite Galois extension of fields 
with Galois group $G$ 
which acts on $L(x_1,\ldots,x_n)$ by $k$-automorphisms. 
Suppose that for any $\sigma\in G$, 
\[
\sigma(\alpha x_i)=\sigma(\alpha)\sum_{j=1}^na_{ij}(\sigma)x_j,
\quad \alpha, a_{ij}(\sigma)\in L. 
\]
Then $L(x_1,\ldots,x_n)^G$ is $k$-rational. 
\end{theorem}

\begin{theorem}[Hajja, Kang {\cite[Theorem 1]{HK95}}]\label{thHK}
Let $L$ be a field and $G$ be a finite group 
acting on $L(x_1,\ldots,x_n)$. 
Suppose that\\
{\rm (i)} $\sigma(L)\subset L$ for any $\sigma\in G$;\\
{\rm (ii)} the restriction of the actions of $G$ to $L$ 
is faithful;\\
{\rm (iii)} for any $\sigma\in G$, 
\begin{align*}
\left(
\begin{array}{c}
\sigma(x_1)\\
\vdots\\
\sigma(x_n)
\end{array}
\right)
=A(\sigma)
\left(
\begin{array}{c}
x_1\\
\vdots\\
x_n
\end{array}
\right)+B(\sigma)
\end{align*}
where $A(\sigma)\in GL_n(L)$ 
and $B(\sigma)$ is an $n\times 1$ matrix over $L$. 
Then there exist $z_1,\ldots,z_n\in L(x_1,\ldots,x_n)$ 
such that $L(x_1,\ldots,x_n)=L(z_1,\ldots,z_n)$ 
with $\sigma(z_i)=z_i$ 
for any $\sigma\in G$ and $1\leq i\leq n$. 
\end{theorem}

%%%%%%%%%%%%%%%%%%%%%%%%%%%%%%%%%%%%%%%%%%%%%%%%%%%%%%%%%%%%%%%%%%
%
\subsection{Explicit transcendental bases} 
%
%%%%%%%%%%%%%%%%%%%%%%%%%%%%%%%%%%%%%%%%%%%%%%%%%%%%%%%%%%%%%%%%%%

\begin{lemma}[{Hashimoto, Hoshi, Rikuna \cite[page 1176]{HHR08}, Hoshi, Kang, Kitayama \cite[Lemma 3.3]{HKK14}, see also \cite[Lemma 3.4]{HKY11}}]\label{lem:tau1} 
Let $k$ be a field and $\tau$ act on $k(x_1,x_2)$ 
by $k$-automorphisms 
\[
\tau : x_1\mapsto \frac{1}{x_1},\ x_2\mapsto \frac{1}{x_2}.
\] 
Then $k(x_1,x_2)^{\langle \tau \rangle}=k(t_1,t_2)$ where 
\begin{align*}
t_1=\frac{x_1x_2+1}{x_1+x_2},\ \ 
t_2=\left\{ 
\begin{array}{lc} 
\frac{x_1x_2-1}{x_1-x_2} & \text{if } \mathrm{char}\ k \neq 2,\vspace*{1mm} \\ 
\frac{x_1(x_2^2+1)}{x_2(x_1^2+1)} & \text{if } \mathrm{char}\ k=2.
\end{array} \right. 
\end{align*}
\end{lemma} 

\begin{lemma}[Hajja, Kang {\cite[Lemma 2.7]{HK94}}, {\cite[Theorem 2.4]{Kan04}}]
\label{lemab}
Let $k$ be a field and $-I_2$ act on $k(x_1,x_2)$ by 
\begin{align*}
-I_2\,:\, x_1\ \mapsto\ \frac{a}{x_1},\ x_2\ \mapsto\ \frac{b}{x_2},\quad a\in k^\times
\end{align*}
where $b=c(x_1+(a/x_1))+d$ such that $c,d\in k$ and at least one of $c$ and $d$ 
is non-zero. 
Then $k(x_1,x_2)^{\langle -I_2\rangle}=k(u_1,u_2)$ where
\begin{align*} 
u_1=\frac{x_1-\frac{a}{x_1}}{x_1x_2-\frac{ab}{x_1x_2}}=\frac{x_2(x_1^2-a)}{x_1^2x_2^2-ab},\quad 
u_2=\frac{x_2-\frac{b}{x_2}}{x_1x_2-\frac{ab}{x_1x_2}}=\frac{x_1(x_2^2-b)}{x_1^2x_2^2-ab}.
\end{align*}
\end{lemma}

\begin{lemma}[{Hoshi, Kitayama, Yamasaki \cite[Theorem 3.13, Remark 3.14]{HKY11}}]\label{lem:mI3}
Let $k$ be a field and $-I_3$ act on $k(x_1,x_2,x_3)$ 
by $k$-automorphism 
\[ -I_3 : x_1\mapsto \frac{1}{x_1},\ x_2\mapsto \frac{1}{x_2},\ x_3\mapsto \frac{1}{x_3}. \] 
Then $k(x_1,x_2,x_3)^{\langle -I_3 \rangle}=k(t_1,t_2,t_3)$ where 
\[ t_1=\frac{x_1x_2+1}{x_1+x_2},\ \ t_2=\frac{x_2x_3+1}{x_2+x_3},\ \ t_3=\frac{x_3x_1+1}{x_3+x_1}. \]   
\end{lemma}

\begin{lemma}[Masuda {\cite{Mas55}, Hoshi, Kang \cite[Theorem 2.2]{HK10}}] 
\label{lem:sigma3B} 
Let $k$ be a field and $\cb$ act on $k(x_1,x_2,x_3)$ 
by $k$-automorphism 
\[
\cb : x_1\mapsto x_2\mapsto x_3\mapsto x_1.
\] 
Then $k(x_1,x_2,x_3)^{\langle \cb \rangle}=k(s_1,u,v)$ where 
\begin{align*} 
s_1 &=s_1(x_1,x_2,x_3)= x_1+x_2+x_3, \\ 
u &=u(x_1,x_2,x_3)= \frac{x_1x_2^2+x_2x_3^2+x_3x_1^2-3x_1x_2x_3}{x_1^2+x_2^2+x_3^2-x_1x_2-x_2x_3-x_3x_1}, \\ 
v &=v(x_1,x_2,x_3)= \frac{x_1^2x_2+x_2^2x_3+x_3^2x_1-3x_1x_2x_3}{x_1^2+x_2^2+x_3^2-x_1x_2-x_2x_3-x_3x_1}. 
\end{align*} 
\end{lemma} 

\begin{lemma}[Hoshi, Kitayama, Yamasaki {\cite[Theorem 3.16]{HKY11}}]  
\label{lem:tau1lambda1} 
Let $k$ be a field with {\rm char} $k \neq 2$ 
and $\tau_1$, $\lambda_1$ 
act on $k(x_1,x_2,x_3)$ by $k$-automorphisms defined by  
\begin{align*} 
\tau_1 &: x_1\mapsto \frac{1}{x_1}, \ x_2\mapsto \frac{1}{x_2}, \ x_3\mapsto x_3, & 
\lambda_1 &: x_1\mapsto \frac{1}{x_1}, \ x_2\mapsto x_2, \ x_3\mapsto \frac{1}{x_3}. 
\end{align*} 
Then $k(x_1,x_2,x_3)^{\langle \tau_1, \lambda_1 \rangle}=k(t_1,t_2,t_3)$ where 
\begin{align*} 
t_1 &= \frac{-x_1+x_2+x_3-x_1x_2x_3}{1-x_1x_2+x_2x_3-x_3x_1}, & 
t_2 &= \frac{x_1-x_2+x_3-x_1x_2x_3}{1-x_1x_2-x_2x_3+x_3x_1}, & 
t_3 &= \frac{x_1+x_2-x_3-x_1x_2x_3}{1+x_1x_2-x_2x_3-x_3x_1}.  
\end{align*} 
\end{lemma}

\begin{lemma}[Hoshi, Kitayama, Yamasaki {\cite[Lemma 3.8]{HKY11}}]\label{lemV42}
Let $k$ be a field and $G_{3,1,4}=\langle\ta_3,\la_3\rangle$ 
act on $k(x_1,x_2,x_3)$ by 
\begin{align*}
\ta_3\,&:\, x_1\mapsto\ x_2\ \mapsto x_1,\ x_3\mapsto\ \frac{c}{x_1x_2x_3}\ \mapsto x_3,\\
\la_3\,&:\, x_1\mapsto\ x_3\ \mapsto x_1,\ x_2\mapsto\ \frac{c}{x_1x_2x_3}\ \mapsto x_2,\quad c\in k^\times. 
\end{align*}
Define $w=c/(x_1x_2x_3)$. 
Then $k(x_1,x_2,x_3)^{G_{3,1,4}}=k(x_1,x_2,x_3)^{\langle \ta_3,\la_3\rangle}=k(v_1,v_2,v_3)$ where 
\begin{align*}
v_1\ &=\ \frac{x_1+x_2-x_3-w}{x_1x_2-x_3w}\ =\ \frac{c-x_1x_2x_3(x_1+x_2-x_3)}{x_3(c-x_1^2x_2^2)},\\
v_2\ &=\ \frac{x_1-x_2-x_3+w}{x_1w-x_2x_3}\ =\ \frac{c-x_1x_2x_3(-x_1+x_2+x_3)}{x_1(c-x_2^2x_3^2)},\\
v_3\ &=\ \frac{x_1-x_2+x_3-w}{x_1x_3-x_2w}\ =\ \frac{c-x_1x_2x_3(x_1-x_2+x_3)}{x_2(c-x_1^2x_3^2)}.
\end{align*}
\end{lemma}

%%%%%%%%%%%%%%%%%%%%%%%%%%%%%%%%%%%%%%%%%%%%%%%%%%%%%%%%%%%%%%%%%%%
%
\subsection{Rationality of quadrics and conic bundles}\label{ssConic} 
%
%%%%%%%%%%%%%%%%%%%%%%%%%%%%%%%%%%%%%%%%%%%%%%%%%%%%%%%%%%%%%%%%%%%

\begin{theorem}[{Ohm \cite[Lemma 5.7]{Ohm94}, see also 
Hajja, Kang, Ohm \cite[Proposition 2.1]{HKO94}}]
\label{thOhm}
Let $k$ be a field with {\rm char} $k\neq 2$ 
and $K$ be a function field of a quadric 
$q=a_0+a_1x_1^2+\cdots+ a_nx_n^2$ over $k$, 
i.e. $K$ is $k$-isomorphic to $k[x_1,\ldots,x_n]/(q)$. 
Then the following conditions are equivalent:\\
{\rm (i)} $a_0x_0^2+a_1x_1^2+\cdots+a_nx_n^2$ has a non-trivial $k$-zero;\\
{\rm (ii)} $a_0+a_1x_1^2+\cdots+a_nx_n^2$ has a $k$-zero;\\
{\rm (iii)} $K$ is $k$-rational;\\
{\rm (iv)} $K$ is $k$-unirational.
\end{theorem}

\begin{lemma}\label{lemXY}
Let $k$ be a field with {\rm char} $k\neq 2$, 
$K=k(\sqrt{d})$ be a quadratic extension of $k$ 
and $\sigma$ act on $K(x_1,\ldots,x_n)$ by $k$-automorphisms 
\begin{align*}
\sigma : \sqrt{d}\mapsto -\sqrt{d},\ 
x_i\mapsto x_i\ (1\leq i\leq n-1),\ 
x_n\mapsto \frac{f(x_1,\ldots,x_{n-1})}{x_n}
\end{align*}
where $f(x_1,\ldots,x_{n-1})$ is a polynomial in $k[x_1,\ldots,x_{n-1}]$. 
Then we have $K(x_1,\ldots,x_n)^{\langle\sigma\rangle}=k(X,Y,x_1,\ldots,x_{n-1})$ 
where $X^2-dY^2=f(x_1,\ldots,x_{n-1})$. 
\end{lemma}
\begin{proof}
Define 
\begin{align*}
X=\frac{1}{2}\left(x_n+\frac{f(x_1,\ldots,x_{n-1})}{x_n}\right),
Y=\frac{1}{2\sqrt{d}}\left(x_n-\frac{f(x_1,\ldots,x_{n-1})}{x_n}\right).
\end{align*}
Then the assertion follows from 
$K(x_1,\ldots,x_n)^{\langle\sigma\rangle}\supset k(X,Y,x_1,\ldots,x_{n-1})$ 
and $[K(x_1,\ldots,x_n) : k(X,Y,x_1,\ldots,x_{n-1})]\leq 2$.
\end{proof}

For the rationality of quadrics and conic bundles, see 
\cite[Section 4]{Kan07}, \cite{Yam}. 

%%%%%%%%%%%%%%%%%%%%%%%%%%%%%%%%%%%%%%%%%%%%%%%%%%%%%%%%%%%%%%%%%%%
%
\subsection{Conjugacy classes move}\label{subsec:conversion} 
%
%%%%%%%%%%%%%%%%%%%%%%%%%%%%%%%%%%%%%%%%%%%%%%%%%%%%%%%%%%%%%%%%%%%
~\\

For the rationality problem of $K(x_1,\ldots,x_n)^G$ under 
purely quasi-monomial actions, 
we may convert some cases into their $GL_n(\bQ)$-conjugation 
(see also \cite[Section 13.2]{HKY11}) 
although the problem is determined up to $GL_n(\bZ)$-conjugation 
(a conjugate of $G$ corresponds to just some base change). 
When the action of $G$ on $K$ is faithful, 
i.e. algebraic $k$-tori case, 
$GL_n(\bQ)$-conjugation corresponds to $k$-isogeny (see Ono \cite[Section 1.3]{Ono61}). 

\begin{theorem}[Kitayama {\cite[Theorem 2.4]{Kit11}}]\label{lem:conversion0} 
Let $k$ be a field with {\rm char} $k \neq 2$. 
Let $G$ and $G'$ be $GL_n(\bQ)$-conjugate subgroups of $GL_n(\bZ)$ 
which act on $K(x_1, \ldots, x_n)$ by purely quasi-monomial 
$k$-automorphisms, 
$P=[p_{i,j}]\in GL_n(\bQ)$ be an integer matrix with 
$G'=P^{-1}GP$ and 
$H'=\{\sigma\in G'\mid \sigma(\alpha)=\alpha\ {\rm for\ any}\ \alpha\in K\}$. 
For $y_j=\prod _{i=1}^n x_i^{p_{i,j}}$ $(j=1, \ldots, n)$, 
the induced purely quasi-monomial action of $G$ 
on $K(y_1, \ldots, y_n)$ is the same action of $G'$
on $K(x_1, \ldots, x_n)$ and $H=\{\sigma\in G\mid \sigma(\alpha)=\alpha\ {\rm for\ any}\ \alpha\in K\}$ where $H=PH'P^{-1}$.
\end{theorem}  
We will apply Theorem \ref{lem:conversion0} as follows. 

\begin{lemma}\label{lem:conversion} 
Let $P_1$, $P_2\in GL_3(\bZ)$ be two matrices defined by 
\begin{align}
P_1=\begin{bmatrix} 1&0&0\\0&1&0\\1&1&2 \end{bmatrix},\ 
P_2=\begin{bmatrix} 1&0&0\\1&2&0\\1&0&2 \end{bmatrix}.\label{matP1P2}
\end{align}
Let $\rho_1$, $\rho_2$ and $\rho_3$ be $K$-automorphisms on $K(x_1,x_2,x_3)$ 
defined by 
\begin{align*} 
\rho_1 &: x_1\mapsto -x_1,\ x_2\mapsto -x_2,\ x_3\mapsto -x_3,\  
\rho_2 : x_1\mapsto -x_1,\ x_2\mapsto x_2,\ x_3\mapsto -x_3,\\ 
\rho_3 &: x_1\mapsto x_1,\ x_2\mapsto -x_2,\ x_3\mapsto -x_3.  
\end{align*} 
Then we have:\\
{\rm (i)} $K(x_1,x_2,x_3)^{P_1^{-1}GP_1}$ is $k$-rational 
if and only if $\big(K(x_1,x_2,x_3)^G\big)^{\langle \rho_1\rangle}$ is $k$-rational;\\
{\rm (ii)} $K(x_1,x_2,x_3)^{P_2^{-1}GP_2}$ is $k$-rational 
if and only if $\big(K(x_1,x_2,x_3)^G\big)^{\langle \rho_2, \rho_3\rangle}$ is $k$-rational. 
\end{lemma} 

\begin{proof}
By Theorem \ref{lem:conversion0}, we have 
\begin{align*} 
&K(x_1,x_2,x_3)^{\langle \rho_1 \rangle}=K(y_1,y_2,y_3), & 
&y_1=x_1x_3, \ y_2=x_2x_3, \ y_3=x_3^2, \\ 
&K(x_1,x_2,x_3)^{\langle \rho_2, \rho_3 \rangle}=K(z_1,z_2,z_3), & 
&z_1=x_1x_2x_3, \ z_2=x_2^2, \ z_3=x_3^2  
\end{align*} 
and 
the action of 
$G$ on $K(y_1,y_2,y_3)$ (resp. $K(z_1,z_2,z_3)$)
is the same as 
the action of $P_1^{-1}GP_1$ (resp. $P_2^{-1}GP_2$) 
on $K(x_1,x_2,x_3)$.  
\end{proof} 

Indeed, we will take a finite subgroup $G$ of $GL_3(\bZ)$ 
such that $G$ is $\bZ$-reducible and 
$P_1^{-1}GP_1$ $($resp. $P_2^{-1}GP_2) \leq GL_3(\bZ)$ is $\bZ$-irreducible. 
Namely, the $k$-rationality of 
$K(x_1,x_2,x_3)^{P_1^{-1}GP_1}$ 
(resp. $K(x_1,x_2,x_3)^{P_2^{-1}GP_2}$) 
for $\bZ$-irreducible $P_1^{-1}GP_1$ (resp. $P_2^{-1}GP_2$) 
can be reduced to the $k$-rationality of 
$\big(K(x_1,x_2,x_3)^G\big)^{\langle \rho_1\rangle}$ 
(resp. $\big(K(x_1,x_2,x_3)^G\big)^{\langle \rho_2, \rho_3\rangle}$) 
for $\bZ$-reducible $G$ 
(see Subsections \ref{subsec:3} and \ref{subsec:4}). 
We will also use this technique for the $7$th crystal system (I), (II) 
although the both of $G$ and $P_1^{-1}GP_1$ are $\bZ$-irreducible (see Subsection \ref{subsec:71}). 

%%%%%%%%%%%%%%%%%%%%%%%%%%%%%%%%%%%%%%%%%%%%%%%%%%%%%%%%%%%%%%%%%%%
%%%%%%%%%%%%%%%%%%%%%%%%%%%%%%%%%%%%%%%%%%%%%%%%%%%%%%%%%%%%%%%%%%%
%
\section{Proof of Theorem \ref{thmain1}}\label{seP1}
%
%%%%%%%%%%%%%%%%%%%%%%%%%%%%%%%%%%%%%%%%%%%%%%%%%%%%%%%%%%%%%%%%%%%
%%%%%%%%%%%%%%%%%%%%%%%%%%%%%%%%%%%%%%%%%%%%%%%%%%%%%%%%%%%%%%%%%%%

In this section, we will give a proof of Theorem \ref{thmain1}.
First, we obtain the following theorem 
by Theorem \ref{thm:HKK2} and Corollary \ref{cor:AHK}. 
\begin{theorem}\label{th41}
If $G$ is a group which belongs to {\rm the 1st crystal system}, 
{\rm the 2nd crystal system}, {\rm the 3rd crystal system (I)}, 
{\rm the 5th crystal system (II)}, 
or {\rm the 6th crystal system}, then $K(x_1,x_2,x_3)^G$ 
under purely quasi-monomial actions of $G$ is $k$-rational. 
\end{theorem} 

Let $H=\{\sigma\in G \mid \sigma(\alpha)=\alpha$ for any $\alpha\in K\}$. 
The remaining cases will be treated below. 

%%%%%%%%%%%%%%%%%%%%%%%%%%%%%%%%%%%%%%%%%%%%%%%%%%%%%%%%%%%
%
\subsection{The 3rd crystal system (II)} \label{subsec:3}
%
%%%%%%%%%%%%%%%%%%%%%%%%%%%%%%%%%%%%%%%%%%%%%%%%%%%%%%%%%%%%

We treat the following six groups: 
\begin{align*}
G_{3,1,3}&=\langle \tau_2,\lambda_2\rangle\simeq\cC_2\times \cC_2,& 
G_{3,1,4}&=\langle \tau_3,\lambda_3\rangle\simeq\cC_2\times \cC_2,\\
G_{3,2,4}&=\langle \tau_2,-\lambda_2\rangle\simeq\cC_2\times \cC_2,& 
G_{3,2,5}&=\langle \tau_3,-\lambda_3\rangle\simeq\cC_2\times \cC_2,\\ 
G_{3,3,3}&=\langle \tau_2,\lambda_2,-I_3\rangle\simeq\cC_2\times 
\cC_2\times \cC_2,& 
G_{3,3,4}&=\langle \tau_3,\lambda_3,-I_3\rangle\simeq\cC_2\times 
\cC_2\times \cC_2.
\end{align*} 
Let $P_1, P_2\in GL_3(\bZ)$ be two matrices defined as in Equation (\ref{matP1P2}). 
Then we may confirm that 
\begin{align*} 
G_{3,1,3} &\sim P_1^{-1}G_{3,1,1}P_1, & 
G_{3,1,4} &\sim P_2^{-1}G_{3,1,1}P_2, \\ 
G_{3,2,4} &\sim P_1^{-1}G_{3,2,1}P_1, & 
G_{3,2,5} &\sim P_2^{-1}G_{3,2,1}P_2, \\ 
G_{3,3,3} &\sim P_1^{-1}G_{3,3,1}P_1, & 
G_{3,3,4} &\sim P_2^{-1}G_{3,3,1}P_2
\end{align*} 
where $\sim$ means $GL_3(\bZ)$-conjugation. 
Let $\rho_1$, $\rho_2$, $\rho_3$ be $K$-automorphisms on 
$K(x_1,x_2,x_3)$ defined by 
\begin{align*} 
\rho_1 &: x_1\mapsto -x_1,\ x_2\mapsto -x_2,\ x_3\mapsto -x_3,\  
\rho_2 : x_1\mapsto -x_1,\ x_2\mapsto x_2,\ x_3\mapsto -x_3,\\ 
\rho_3 &: x_1\mapsto x_1,\ x_2\mapsto -x_2,\ x_3\mapsto -x_3
\end{align*} 
as in Lemma \ref{lem:conversion}. 
By Lemma \ref{lem:conversion}, 
it suffices to consider the $k$-rationality of 
$\big(K(x_1,x_2,x_3)^G\big)^{\langle \rho_1 \rangle}$ and 
$\big(K(x_1,x_2,x_3)^G\big)^{\langle \rho_2, \rho_3 \rangle}$
for $G=G_{3,j,1}$ $(j=1,2,3)$. 
Note that we already know that $K(x_1,x_2,x_3)^G$ is $k$-rational 
by Theorem \ref{th41}. 

The actions of $\tau_1$, $\lambda_1$ and $-I_3$ on $K(x_1,x_2,x_3)$
are given by 
\begin{align*} 
\tau_1 &: x_1\mapsto \frac{1}{x_1},\ x_2\mapsto \frac{1}{x_2},\ x_3\mapsto x_3, & 
\lambda_1 &: x_1\mapsto \frac{1}{x_1},\ x_2\mapsto x_2,\ x_3\mapsto \frac{1}{x_3}, \\ 
-I_3 &: x_1\mapsto \frac{1}{x_1},\ x_2\mapsto \frac{1}{x_2},\ x_3\mapsto \frac{1}{x_3}. 
\end{align*}\\

Case 1:  $G=G_{3,1,1}=\langle \tau_1, \lambda_1 \rangle\simeq \cC_2\times \cC_2$.\\

Let $H$ be a non-trivial proper normal subgroup of $G$. 
By symmetry, we have only to consider $H=\langle \tau_1 \rangle$. 
We can take $\sqrt{a}\in K$ such that 
$K=k(\sqrt{a})$ and $\lambda_1(\sqrt{a})=-\sqrt{a}$. 
Then we have $K(x_1,x_2,x_3)^G=k(y_1,y_2,y_3)$ where 
\[
y_1=\left(\frac{x_1-1}{x_1+1}\right)^2,\
y_2=\frac{\sqrt{a}(x_1-1)(x_2-1)}{(x_1+1)(x_2+1)},\ 
y_3=\frac{\sqrt{a}(x_3-1)}{x_3+1}
\]  
and $\rho_1$, $\rho_2$ and $\rho_3$ act on $k(y_1,y_2,y_3)$ by 
\begin{align*} 
\rho_1 &: y_1\mapsto \frac{1}{y_1},\ y_2\mapsto \frac{a}{y_2},\ y_3\mapsto \frac{a}{y_3},\\ 
\rho_2 &: y_1\mapsto \frac{1}{y_1},\ y_2\mapsto \frac{y_2}{y_1},\ y_3\mapsto \frac{a}{y_3},\ 
\rho_3 : y_1\mapsto y_1,\ y_2\mapsto \frac{ay_1}{y_2},\ y_3\mapsto \frac{a}{y_3}. 
\end{align*} 
Hence the $k$-rationality of $\big(K(x_1,x_2,x_3)^G\big)^{\langle \rho_1 \rangle}$  
and $\big(K(x_1,x_2,x_3)^G\big)^{\langle \rho_2, \rho_3 \rangle}$ 
can be reduced to the $k$-rationality 
for three-dimensional monomial actions on 
$k(y_1,y_2,y_3)$ of $G_{1,2,1}\simeq \cC_2$ and $G_{3,1,2}\simeq \cC_2\times \cC_2$ 
respectively, which has already been settled in \cite{Sal00} and \cite{HKY11}. 
They are both $k$-rational.\\

Case 2: $G=G_{3,2,1}=\langle \tau_1, -\lambda_1 \rangle\simeq \cC_2\times \cC_2$.\\

Because $\ta_1(x_3)=-\la_1(x_3)=x_3$ and 
$\rho_i(x_3)=-x_3$ $(i=1,2,3)$, 
by Theorem \ref{thAHK}, 
the $k$-rationality of $\big(K(x_1,x_2,x_3)^G\big)^{\langle \rho_1 \rangle}$ and 
$\big(K(x_1,x_2,x_3)^G\big)^{\langle \rho_2, \rho_3 \rangle}$ 
can be reduced to the $k$-rationality of 
purely quasi-monomial actions of $G$ on 
$K(x_1,x_2)^{\langle \rho_1 \rangle}$ and 
$K(x_1,x_2)^{\langle \rho_2, \rho_3 \rangle}$. 
They are $k$-rational 
for any non-trivial proper normal subgroup $H$ of $G$ 
by Theorem \ref{thm:HKK2}.\\

Case 3: 
$G=G_{3,3,1}=\langle \tau_1, \lambda_1, -I_3 \rangle\simeq \cC_2\times \cC_2\times \cC_2$.\\

Let $H$ be a non-trivial proper normal subgroup of $G$. 
By symmetry, we have only to consider the following three cases:\\
 
(i) $\langle-\tau_1\lambda_1\rangle\leq H$. 
We have $K(x_1,x_2,x_3)^{\langle -\tau_1\lambda_1 \rangle}=K(y_1,y_2,y_3)$ 
where $y_1=x_1+\frac{1}{x_1}$, $y_2=x_2$, $y_3=x_3$. 
By Theorem \ref{thAHK}, 
the $k$-rationality of $\big(K(x_1,x_2,x_3)^G\big)^{\langle \rho_1 \rangle}$ and 
$\big(K(x_1,x_2,x_3)^G\big)^{\langle \rho_2, \rho_3 \rangle}$ 
can be reduced to the $k$-rationality for 
purely quasi-monomial actions of $G$ on 
$K(y_2,y_3)^{\langle \rho_1 \rangle}$ and 
$K(y_2,y_3)^{\langle \rho_2, \rho_3 \rangle}$. 
They are both $k$-rational by Theorem \ref{thm:HKK2}.\\

(ii) $\langle\tau_1\rangle\leq H$. 
By Lemma \ref{lem:tau1}, 
we have $K(x_1,x_2,x_3)^{\langle \tau_1 \rangle}=K(y_1,y_2,y_3)$ where 
\[
y_1=\frac{x_1x_2+1}{x_1+x_2},\
y_2=\frac{x_1x_2-1}{x_1-x_2},\ 
y_3=x_3.
\] 
The actions of $\lambda_1, -I_3, \rho_1, \rho_2$ and 
$\rho_3$ on $K(y_1,y_2,y_3)$ are given by 
\begin{align*} 
\lambda_1 &: y_1\mapsto \frac{1}{y_1},\ y_2\mapsto \frac{1}{y_2},\ y_3\mapsto \frac{1}{y_3}, & 
-I_3 &: y_1\mapsto y_1,\ y_2\mapsto y_2,\ y_3\mapsto \frac{1}{y_3}, \\ 
\rho_1 &: y_1\mapsto -y_1,\ y_2\mapsto -y_2,\ y_3\mapsto -y_3, \\ 
\rho_2 &: y_1\mapsto y_2,\ y_2\mapsto y_1,\ y_3\mapsto -y_3, & 
\rho_3 &: y_1\mapsto -y_2,\ y_2\mapsto -y_1,\ y_3\mapsto -y_3. 
\end{align*} 

We have $K(y_1,y_2,y_3)^{\langle\rho_1\rangle}=K(z_1,z_2,z_3)$ 
where $z_1=y_1^2, z_2=\frac{y_1y_2-1}{y_1y_2+1}, z_3=y_1y_3$, 
and 
\begin{align*} 
\lambda_1 &: z_1\mapsto \frac{1}{z_1},\ z_2\mapsto -z_2,\ z_3\mapsto \frac{1}{z_3}, & 
-I_3 &: z_1\mapsto z_1,\ z_2\mapsto z_2,\ z_3\mapsto \frac{z_1}{z_3}. 
\end{align*} 
Hence, by Theorem \ref{thAHK}, 
the $k$-rationality of $\big(K(x_1,x_2,x_3)^G\big)^{\langle \rho_1 \rangle}$ 
can be reduced to the $k$-rationality for 
purely quasi-monomial actions of 
$\langle \lambda_1, -I_3 \rangle$ on $K(z_1,z_3)$. 
It is $k$-rational by Theorem \ref{thm:HKK2}. 

We have $K(y_1,y_2,y_3)^{\langle\rho_2\rho_3\rangle}=K(w_1,w_2,w_3)$ 
where $w_1=\frac{y_1-y_2}{y_1+y_2}, w_2=\frac{y_1y_2-1}{y_1y_2+1}, w_3=y_3$ and 
\begin{align*} 
\lambda_1 &: w_1\mapsto -w_1,\ w_2\mapsto -w_2,\ w_3\mapsto \frac{1}{w_3}, & 
-I_3 &: w_1\mapsto w_1,\ w_2\mapsto w_2,\ w_3\mapsto \frac{1}{w_3}, \\ 
\rho_2=\rho_3 &: w_1\mapsto -w_1,\ w_2\mapsto w_2,\ w_3\mapsto -w_3.  
\end{align*} 
By Theorem \ref{thAHK}, 
the $k$-rationality of $\big(K(x_1,x_2,x_3)^G\big)^{\langle \rho_2, \rho_3  \rangle}$ 
can be reduced to the $k$-rationality for 
quasi-monomial actions of $\langle \lambda_1, -I_3, \rho_2 \rangle$ on $K(w_3)$. 
It is $k$-rational by Theorem \ref{thHKKp113}.\\

(iii)\ $H=\langle -I_3 \rangle\simeq \cC_2$. 
We have $K(x_1,x_2,x_3)^H=K(y_1,y_2,y_3)$ where 
\[
y_1=\left(\frac{x_1-1}{x_1+1}\right)^2,\ 
y_2=\frac{(x_1-1)(x_2-1)}{(x_1+1)(x_2+1)},\ 
y_3=\frac{(x_1-1)(x_3-1)}{(x_1+1)(x_3+1)}
\]
and
\begin{align*} 
\tau_1 &: %\sqrt{a}\mapsto -\sqrt{a}, \sqrt{b} \mapsto \sqrt{b},\ 
y_1\mapsto y_1,\ y_2\mapsto y_2,\ y_3\mapsto -y_3,& 
\lambda_1 &: %\sqrt{a}\mapsto \sqrt{a}, \sqrt{b} \mapsto -\sqrt{b},\ 
y_1\mapsto y_1,\ y_2\mapsto -y_2,\ y_3\mapsto y_3,  \\ 
\rho_1 &: %\sqrt{a} \mapsto \sqrt{a}, \sqrt{b}\mapsto\sqrt{b},\ 
y_1\mapsto \frac{1}{y_1},\ y_2\mapsto \frac{1}{y_2},\ y_3\mapsto \frac{1}{y_3}, \\  
\rho_2 &: %\sqrt{a} \mapsto \sqrt{a}, \sqrt{b}\mapsto\sqrt{b},\ 
y_1\mapsto \frac{1}{y_1},\ y_2\mapsto \frac{y_2}{y_1},\ y_3\mapsto \frac{1}{y_3},& 
\rho_3 &: %\sqrt{a} \mapsto \sqrt{a}, \sqrt{b}\mapsto\sqrt{b},\ 
y_1\mapsto y_1,\ y_2\mapsto \frac{y_1}{y_2},\ y_3\mapsto \frac{y_1}{y_3}.
\end{align*} 
We may take $\sqrt{a}, \sqrt{b} \in K$ such that $K=k(\sqrt{a},\sqrt{b})$ and 
\begin{align*} 
\tau_1 &: \sqrt{a}\mapsto -\sqrt{a},\ \sqrt{b} \mapsto \sqrt{b},\ 
y_1\mapsto y_1,\ y_2\mapsto y_2,\ y_3\mapsto -y_3,  \\  
\lambda_1 &: \sqrt{a}\mapsto \sqrt{a}\, \sqrt{b} \mapsto -\sqrt{b},\ 
y_1\mapsto y_1,\ y_2\mapsto -y_2,\ y_3\mapsto y_3.
\end{align*} 
Then $K(y_1,y_2,y_3)^{\langle \tau_1, \lambda_1 \rangle}=k(z_1,z_2,z_3)$ 
where $z_1=y_1$, $z_2=\sqrt{b} y_2$, $z_3=\sqrt{a} y_3$, and 
\begin{align*} 
\rho_1 &: z_1\mapsto \frac{1}{z_1},\ z_2\mapsto \frac{b}{z_2},\ z_3\mapsto \frac{a}{z_3}, \\ 
\rho_2 &: z_1\mapsto \frac{1}{z_1},\ z_2\mapsto \frac{z_2}{z_1},\ z_3\mapsto \frac{a}{z_3}, & 
\rho_3 &: z_1\mapsto z_1,\ z_2\mapsto \frac{bz_1}{z_2},\ z_3\mapsto \frac{az_1}{z_3}. 
\end{align*} 
Hence the $k$-rationality of 
$\big(K(x_1,x_2,x_3)^G\big)^{\langle \rho_1 \rangle}$ 
and $\big(K(x_1,x_2,x_3)^G\big)^{\langle \rho_2, \rho_3 \rangle}$ 
can be reduced to the $k$-rationality of three-dimensional 
monomial actions of $G_{1,2,1}$ and $G_{3,1,3}$ respectively,  
which has been settled in \cite{Sal00} and \cite{HKY11}. 
They are both $k$-rational. 

%%%%%%%%%%%%%%%%%%%%%%%%%%%%%%%%%%%%%%%%%%%%%%%%%%%%%%%%%%%
%
\subsection{The 4th crystal system (I), (II)} \label{subsec:4}
%
%%%%%%%%%%%%%%%%%%%%%%%%%%%%%%%%%%%%%%%%%%%%%%%%%%%%%%%%%%%%

We treat the following sixteen groups:  
\begin{align*}
G_{4,1,1}&=\langle \caa\rangle\simeq\cC_4,&
G_{4,2,1}&=\langle -\caa\rangle\simeq\cC_4,\\
G_{4,3,1}&=\langle \caa,-I_3\rangle\simeq\cC_4\times \cC_2,& 
G_{4,4,1}&=\langle \caa,\lambda_1\rangle\simeq\cD_4,\\
G_{4,5,1}&=\langle \caa,-\lambda_1\rangle\simeq\cD_4, &   
G_{4,6,1}&=\langle -\caa,\lambda_1\rangle\simeq\cD_4, \\
G_{4,6,2}&=\langle -\caa,-\lambda_1\rangle\simeq\cD_4, &
G_{4,7,1}&=\langle \caa,\lambda_1,-I_3\rangle\simeq\cD_4\times \cC_2
\end{align*}
which are $\bZ$-reducible and belong to the 4th crystal system (I), and 
\begin{align*} 
G_{4,1,2}&=\langle \cbb\rangle\simeq\cC_4,&
G_{4,2,2}&=\langle -\cbb\rangle\simeq\cC_4,\\
G_{4,3,2}&=\langle \cbb,-I_3\rangle\simeq\cC_4\times \cC_2,& 
G_{4,4,2}&=\langle \cbb,\lambda_3\rangle\simeq\cD_4,\\
G_{4,5,2}&=\langle \cbb,-\lambda_3\rangle\simeq\cD_4, &   
G_{4,6,3}&=\langle -\cbb,-\lambda_3\rangle\simeq\cD_4, \\
G_{4,6,4}&=\langle -\cbb,\lambda_3\rangle\simeq\cD_4, &
G_{4,7,2}&=\langle \cbb,\lambda_3,-I_3\rangle\simeq\cD_4\times \cC_2
\end{align*}
which are $\bZ$-irreducible and belong to the 4th crystal system (II).

Let $P_1\in GL_3(\bZ)$ be a matrix 
defined as Equation (\ref{matP1P2}) in Lemma \ref{lem:conversion}. 
We may confirm that 
\[
G_{4,j,2} \sim P_1^{-1}G_{4,j,1}P_1 \hspace{20pt} 
\text{for } 1\leq j\leq 5,\ j=7,
\]  
\[
G_{4,6,3} \sim P_1^{-1}G_{4,6,1}P_1, \hspace{20pt} 
G_{4,6,4} \sim P_1^{-1}G_{4,6,2}P_1
\]
where $\sim$ means $GL_3(\bZ)$-conjugation. 
By Lemma \ref{lem:conversion}, 
it suffices to consider the $k$-rationality of 
$K(x_1,x_2,x_3)^G$ and 
$\big(K(x_1,x_2,x_3)^G\big)^{\langle \rho_1 \rangle}$ where 
\[
\rho_1 : x_1\mapsto -x_1, x_2\mapsto -x_2, x_3\mapsto -x_3
\] 
for $G=G_{4,j,1}$ ($1\leq j\leq 7$) and $G_{4,6,2}$ 
which are $\bZ$-reducible, contain $\caa$ or $-\caa$ 
and belong to the $4$th crystal system (I). 
 
The actions of $\caa$, $\caa^2$, $\lambda_1$, $-\lambda_1$, 
$-I_3$ and $\rho_1$ on $K(x_1, x_2, x_3)$ are given by  
\begin{align*} 
\caa &: x_1\mapsto x_2,\ x_2\mapsto \frac{1}{x_1},\ x_3\mapsto x_3, & 
-\caa &: x_1\mapsto \frac{1}{x_2},\ x_2\mapsto x_1,\ x_3\mapsto \frac{1}{x_3},\\
\caa^2 &: x_1\mapsto \frac{1}{x_1},\ x_2\mapsto \frac{1}{x_2},\ x_3\mapsto x_3, \\ 
\lambda_1 &: x_1\mapsto \frac{1}{x_1},\ x_2\mapsto x_2,\ x_3\mapsto \frac{1}{x_3}, & 
-\lambda_1 &: x_1\mapsto x_1,\ x_2\mapsto \frac{1}{x_2},\ x_3\mapsto x_3, \\ 
-I_3 &: x_1\mapsto \frac{1}{x_1},\ x_2\mapsto \frac{1}{x_2},\ x_3\mapsto \frac{1}{x_3}, & 
\rho_1 &: x_1\mapsto -x_1,\ x_2\mapsto -x_2,\ x_3\mapsto -x_3. 
\end{align*} 
The non-trivial proper normal subgroups $H$ of each $G=G_{4,j,1}$ 
$(1\leq j\leq 7)$ and $G_{4,6,2}$ are: 
\begin{align*} 
G_{4,1,1} &: H=\langle \caa^2 \rangle,  \\ 
G_{4,2,1} &: H=\langle \caa^2 \rangle,  \\ 
G_{4,3,1} &: H=\langle \pm \caa^2 \rangle, 
\langle \pm \caa \rangle, 
\langle -I_3 \rangle, 
\langle \caa^2, -I_3 \rangle,\\ 
G_{4,4,1} &: H=\langle \caa^2 \rangle, \langle \caa^2, \lambda_1 \rangle, 
\langle \caa^2, \caa\lambda_1 \rangle, \langle \caa \rangle,  \\ 
G_{4,5,1} &: H=\langle \caa^2 \rangle, \langle \caa^2,-\lambda_1 \rangle, 
\langle \caa^2, -\caa\lambda_1 \rangle, \langle \caa \rangle,  \\ 
G_{4,6,1} &: H=\langle \caa^2 \rangle, \langle \caa^2, \lambda_1 \rangle, 
\langle \caa^2, -\caa\lambda_1 \rangle, \langle -\caa \rangle,  \\ 
G_{4,6,2} &: H=\langle \caa^2 \rangle, \langle \caa^2, -\lambda_1 \rangle, 
\langle \caa^2, \caa\lambda_1 \rangle, \langle -\caa \rangle,  \\ 
%%%%%%%%%%%%%%%%%%%%%%%%%%%%%%%%%%%%%%%%%%%%%%%%%%%%%%%%%%%%%%%%%%
G_{4,7,1} &: H=\langle \pm\caa^2 \rangle, 
\langle \caa^2, \pm\lambda_1 \rangle, 
\langle \caa^2, \pm \caa\lambda_1 \rangle, 
\langle \pm \caa \rangle, 
\langle \pm \caa, \pm \lambda_1 \rangle,\\
&\hspace{37pt} 
\langle -I_3 \rangle, \langle \caa^2, -I_3 \rangle, 
\langle \caa^2, \lambda_1, -I_3 \rangle, 
\langle \caa^2, \caa\lambda_1, -I_3 \rangle, 
\langle \caa, -I_3 \rangle.\\ 
\end{align*}

%%%%%%%%%%%%%%%%%%%%%%%%%%%%%%%%%%%%%%%%%%%%%%%%%%%%%%%%%%%%%%%%
Case 1: $\langle\caa^2\rangle\leq H$.\\
 
By Lemma \ref{lem:tau1}, we have 
\[
K(x_1,x_2,x_3)^{\langle \caa^2 \rangle} = K(y_1,y_2,y_3)
\] 
where 
\begin{align*} 
y_1 = \frac{x_1x_2+1}{x_1+x_2},\ 
y_2 =\frac{x_1x_2-1}{x_1-x_2},\ 
y_3 =x_3.
\end{align*}  
The actions of $\caa$, $-\caa$, $\la_1$, $-\la_1$, $-I_3$ and $\rho_1$
on $K(y_1,y_2,y_3)$ are given by 
\begin{align*} 
\caa &: y_1\mapsto \frac{1}{y_1},\ y_2\mapsto -\frac{1}{y_2},\ y_3\mapsto y_3, &
-\caa &: y_1\mapsto \frac{1}{y_1},\ y_2\mapsto -\frac{1}{y_2},\ y_3\mapsto \frac{1}{y_3},\\
\lambda_1 &: y_1\mapsto \frac{1}{y_1},\ y_2\mapsto \frac{1}{y_2},\ y_3\mapsto \frac{1}{y_3},&
-\lambda_1 &: y_1\mapsto \frac{1}{y_1},\ y_2\mapsto \frac{1}{y_2},\ y_3\mapsto y_3, \\
-I_3 &: y_1\mapsto y_1,\ y_2\mapsto y_2,\ y_3\mapsto \frac{1}{y_3}, & 
\rho_1 &: y_1\mapsto -y_1,\ y_2\mapsto -y_2,\ y_3\mapsto -y_3. 
\end{align*} 
Define 
\begin{align*} 
z_1 = \frac{y_2-1}{y_2+1},\ 
z_2 = \frac{y_1-1}{y_1+1},\ 
z_3 = \frac{y_3-1}{y_3+1}. 
\end{align*}
Then $K(x_1,x_2,x_3)^{\langle\caa^2\rangle}=K(y_1,y_2,y_3)=K(z_1,z_2,z_3)$ and 
the actions of $\pm \caa$, $\pm\lambda_1$, $-I_3$ and $\rho_1$ 
on $K(z_1,z_2,z_3)$ are given by 
\begin{align*} 
\pm\caa &: z_1\mapsto -\frac{1}{z_1},\ z_2\mapsto -z_2,\ z_3\mapsto \pm z_3, &
\pm \lambda_1 &: z_1\mapsto -z_1,\ z_2\mapsto -z_2,\ z_3\mapsto \mp z_3, \\ 
-I_3 &: z_1\mapsto z_1,\ z_2\mapsto z_2,\ z_3\mapsto -z_3. 
\end{align*}\\

We will consider the $k$-rationality of 
$K(x_1,x_2,x_3)^G=K(z_1,z_2,z_3)^{G/\langle\caa^2\rangle}$ for each $H\leq G$. \\

{\rm (i)} $H=\langle\caa^2\rangle$. 
By Theorem \ref{thAHK}, 
the $k$-rationality of $K(z_1,z_2,z_3)^G$ is reduced to 
that of $K(z_1)^G$ which has been settled in Theorem \ref{thHKKp113}. 
Indeed, we have the following:

For $G=G_{4,j,1}=\langle\pm\caa\rangle$ $(j=1,2)$, 
(resp. $G_{4,3,1}=\langle\caa,-I_3\rangle$)
we may take $K=k(\sqrt{a})$ 
(resp. $K^{\langle -I_3\rangle}=k(\sqrt{a})$) 
on which $G$ acts by $\pm\caa : \sqrt{a}\mapsto -\sqrt{a}$. 
By Theorem \ref{thHKKp113} (2), 
$K(x_1,x_2,x_3)^G$ is $k$-rational if and only if 
$k(\sqrt{a})(z_1)^{\langle\pm\caa\rangle}$ is $k$-rational 
if and only if $(a,-1)_k=0$. 

For $G=\langle\pm\caa,\pm\la_1\rangle$, i.e. 
$G=G_{4,4,1}$, $G_{4,5,1}$, $G_{4,6,1}$ or $G_{4,6,2}$, 
(resp. $G_{4,7,1}=\langle\caa,\la_1,-I_3\rangle$)
we may take $K=k(\sqrt{a},\sqrt{b})$ 
(resp. $K^{\langle -I_3\rangle}=k(\sqrt{a},\sqrt{b})$) 
on which 
$G$ acts by $\pm\caa : \sqrt{a}\mapsto -\sqrt{a}, \sqrt{b}\mapsto \sqrt{b}$, 
$\pm\la_1 : \sqrt{a}\mapsto\sqrt{a}, \sqrt{b}\mapsto-\sqrt{b}$, 
and $K(x_1,x_2,x_3)^G$ is $k$-rational 
if and only if 
$K(z_1)^{\langle\pm\caa,\pm\la_1\rangle}$ is $k$-rational 
if and only if $k(\sqrt{a})(\sqrt{b}z_1)^{\langle\pm\caa\rangle}$ is $k$-rational 
if and only if $(a,-b)_k=0$. \\
%%%%%%%%%%%%%%%%%%%%%%%%%%%

{\rm (ii)} $H=\langle\caa^2,-I_3\rangle$. 
We have $K(x_1,x_2,x_3)^H=K(z_1,z_2,z_3)^{\langle -I_3\rangle}=K(z_1,z_2,z_3^2)$. 

For $G=G_{4,3,1}=\langle\caa,-I_3\rangle$, 
we may take $K=k(\sqrt{a})$  
on which $G$ acts by $\caa : \sqrt{a}\mapsto -\sqrt{a}$, 
and $K(x_1,x_2,x_3)^G$ is $k$-rational 
if and only if 
$k(\sqrt{a})(z_1)^{\langle\caa\rangle}$ is $k$-rational 
if and only if 
$(a,-1)_k=0$. 

For $G=G_{4,7,1}=\langle\caa,\la_1,-I_3\rangle$, 
we may take $K=k(\sqrt{a},\sqrt{b})$ on which 
$G$ acts by $\caa : \sqrt{a}\mapsto -\sqrt{a}, \sqrt{b}\mapsto \sqrt{b}$, 
$\la_1 : \sqrt{a}\mapsto\sqrt{a}, \sqrt{b}\mapsto-\sqrt{b}$, 
and $K(x_1,x_2,x_3)^G$ is $k$-rational if and only if 
$K(z_1)^{\langle\caa,\la_1\rangle}$ is $k$-rational 
if and only if $k(\sqrt{a})(\sqrt{b}z_1)^{\langle\caa\rangle}$ is $k$-rational 
if and only if $(a,-b)_k=0$. \\

{\rm (iii)} $H=\langle\caa^2,\pm\la_1\rangle$ or 
$\langle\caa^2,\la_1,-I_3\rangle$. 
For $H=\langle\caa^2,\la_1\rangle$, we define 
\begin{align*}
w_1=z_1z_2,\ w_2=\frac{z_2}{z_1},\ w_3=\frac{z_3}{z_1}.
\end{align*}
Then we have 
$K(x_1,x_2,x_3)^{\langle\caa^2,\la_1\rangle}
=K(z_1,z_2,z_3)^{\langle \la_1\rangle}=K(w_1,w_2,w_3)$ 
and 
\begin{align*} 
\pm\caa &: w_1\mapsto w_2,\ w_2\mapsto w_1,\ w_3\mapsto \mp\frac{w_1w_3}{w_2}, 
&-I_3 &: w_1\mapsto w_1,\ w_2\mapsto w_2,\ w_3\mapsto -w_3. 
\end{align*}
Hence, by Theorem \ref{thAHK} and Theorem \ref{thEM} 
(or Theorem \ref{thHK}), $K(x_1,x_2,x_3)^G$ is $k$-rational 
for $G=G_{4,4,1}=\langle\caa,\la_1\rangle$, 
$G_{4,6,1}=\langle-\caa,\la_1\rangle$, 
$G_{4,7,1}=\langle \caa,\la_1,-I_3\rangle$. 

Similarly, we see that $K(x_1,x_2,x_3)^G$ is $k$-rational 
when $H=\langle\caa^2,-\la_1\rangle$ or $\langle\caa^2,\la_1,-I_3\rangle$.\\

{\rm (iv)} $H=\langle\caa^2,\pm\caa\la_1\rangle$ or 
$\langle\caa^2,\caa\la_1,-I_3\rangle$. 
For $H=\langle\caa^2,\caa\la_1\rangle$, we define 
\begin{align*}
w_1=\frac{z_1+1}{z_1-1}z_3,\ w_2=\frac{z_1-1}{z_1+1}z_3,\ w_3=z_2.
\end{align*}
Then we have 
$K(x_1,x_2,x_3)^{\langle\caa^2,\caa\la_1\rangle}
=K(z_1,z_2,z_3)^{\langle \caa\la_1\rangle}=K(w_1,w_2,w_3)$ 
and 
\begin{align*} 
\pm\caa =\pm\la_1 &: w_1\mapsto \mp w_2,\ w_2\mapsto \mp w_1,\ w_3\mapsto -w_3, 
&-I_3 &: w_1\mapsto -w_1,\ w_2\mapsto -w_2,\ w_3\mapsto w_3. 
\end{align*}
Hence, by Theorem \ref{thEM}, $K(x_1,x_2,x_3)^G$ is $k$-rational 
for $G=G_{4,4,1}=\langle\caa,\la_1\rangle$, 
$G_{4,6,2}=\langle-\caa,-\la_1\rangle$, 
$G_{4,7,1}=\langle \caa,\la_1,-I_3\rangle$. 

Similarly, $K(x_1,x_2,x_3)^G$ is $k$-rational 
when $H=\langle\caa^2,-\caa\la_1\rangle$ or 
$\langle\caa^2,\caa\la_1,-I_3\rangle$.\\

{\rm (v)} $H=\langle\pm\caa\rangle$ or 
$\langle\caa,-I_3\rangle$. 
For $H=\langle\caa\rangle$, we define 
\begin{align*}
w_1=z_1-\frac{1}{z_1},\ w_2=\left(z_1+\frac{1}{z_1}\right)\bigg/z_2,\ w_3=z_3.
\end{align*}
Then we have 
$K(x_1,x_2,x_3)^{\langle\caa\rangle}
=K(z_1,z_2,z_3)^{\langle \caa\rangle}=K(w_1,w_2,w_3)$ 
and 
\begin{align*} 
\pm\la_1 &: w_1\mapsto - w_1,\ w_2\mapsto w_2,\ w_3\mapsto \mp w_3, 
&-I_3 &: w_1\mapsto w_1,\ w_2\mapsto w_2,\ w_3\mapsto -w_3. 
\end{align*}
Hence, by Theorem \ref{thEM}, $K(x_1,x_2,x_3)^G$ is $k$-rational 
for $G=G_{4,3,1}=\langle\caa,-I_3\rangle$, 
$G_{4,4,1}=\langle\caa,\la_1\rangle$, 
$G_{4,5,1}=\langle\caa,-\la_1\rangle$, 
$G_{4,7,1}=\langle \caa,\la_1,-I_3\rangle$. 

Similarly, $K(x_1,x_2,x_3)^G$ is $k$-rational when 
$H=\langle-\caa\rangle$ or $\langle\caa,-I_3\rangle$.\\

{\rm (vi)} $H=\langle\pm\caa,\pm\la_1\rangle$. 
We have $G=G_{4,7,1}=\langle\caa,\la_1,-I_3\rangle$. 
For $H=\langle\caa,\la_1\rangle$, 
define 
\begin{align*}
w_1=\left(z_1-\frac{1}{z_1}\right)z_3,\ 
w_2=\left(z_1+\frac{1}{z_1}\right)\bigg/z_2,\ 
w_3=z_3^2.
\end{align*}
Then 
$K(x_1,x_2,x_3)^G
=K(z_1,z_2,z_3)^{\langle \caa,\la_1\rangle}=K(w_1,w_2,w_3)$ 
and 
\begin{align*} 
-I_3 &: w_1\mapsto -w_1,\ w_2\mapsto w_2,\ w_3\mapsto w_3. 
\end{align*}
Hence, by Theorem \ref{thEM}, $K(x_1,x_2,x_3)^G$ is $k$-rational.

Similarly, $K(x_1,x_2,x_3)^G$ is $k$-rational 
when $H=\langle\caa,-\la_1\rangle$ or $\langle-\caa,\pm\la_1\rangle$.\\

%%%%%%%%%%%%%%%%%%%%%%%%%%%%%%%%%%%

By (i)--(vi), we obtained the $k$-rationality criterion of $K(x_1,x_2,x_3)^G$.\\

On the other hand, we have 
\[
K(y_1,y_2,y_3)^{\langle \rho_1 \rangle}
=K\left(\frac{y_1}{y_3}, y_1y_3, y_1y_2\right) 
=K\left(\frac{y_1-y_3}{y_1+y_3}, \frac{y_1y_3-1}{y_1y_3+1}, 
\frac{y_1y_2-1}{y_1y_2+1}\right)=K(w_1,w_2,w_3)
\] 
where 
\begin{align*} 
w_1=\frac{y_1y_2-1}{y_1y_2+1},\ 
w_2=\frac{y_1-y_3}{y_1+y_3}+\frac{y_1y_3-1}{y_1y_3+1},\ 
w_3=\frac{y_1-y_3}{y_1+y_3}-\frac{y_1y_3-1}{y_1y_3+1}. 
\end{align*}  
The actions of $\caa$, $\lambda_1$ and $-I_3$ 
on $K(w_1, w_2, w_3)$ are given by  
\begin{align*} 
\pm\caa &: w_1\mapsto -\frac{1}{w_1},\ w_2\mapsto -w_2,\ w_3\mapsto  \pm w_3, & 
\pm\lambda_1 &: w_1\mapsto -w_1,\ w_2\mapsto -w_2,\ w_3\mapsto \mp w_3, \\ 
-I_3 &: w_1\mapsto w_1,\ w_2\mapsto w_2,\ w_3\mapsto -w_3. 
\end{align*} 
These actions are the same as the above on $K(z_1,z_2,z_3)$.

By Theorem \ref{thHKKp113}, 
we may obtain Theorem \ref{thmain1} 
for the cases where $\langle\caa^2\rangle\leq H$.\\

%%%%%%%%%%%%%%%%%%%%%%%%%%%%%%%%%%%%%%%%%%%%%%%%%%%%%%%%%%%%%%%%
Case 2: $H=\langle -I_3 \rangle\simeq \cC_2$.\\

There exist two cases: $G=G_{4,3,1}=\langle\caa,-I_3\rangle\simeq \cC_4\times \cC_2$, 
$G_{4,7,1}=\langle\caa,\la_1,-I_3\rangle\simeq \cD_4\times \cC_2$.  
By Lemma \ref{lem:mI3}, we have 
\[ K(x_1, x_2, x_3)^{\langle -I_3 \rangle}=K(y_1, y_2, y_3) \] 
where 
\begin{align*} 
y_1=\frac{x_1x_2+1}{x_1+x_2},\  
y_2=\frac{x_2x_3+1}{x_2+x_3},\ 
y_3=\frac{x_3x_1+1}{x_3+x_1}. 
\end{align*}  
The actions of $\caa$, $\lambda_1$ and $\rho_1$ 
on $K(y_1,y_2,y_3)$ are given by 
\begin{align*} 
\caa^* &: y_1\mapsto \frac{1}{y_1},\ y_2\mapsto \frac{1}{y_3},\ y_3\mapsto y_2, & 
\lambda_1^* &: y_1\mapsto \frac{1}{y_1},\ y_2\mapsto \frac{1}{y_2},\ y_3\mapsto y_3, \\ 
\rho_1 &: y_1\mapsto -y_1,\ y_2\mapsto -y_2,\ y_3\mapsto -y_3
\end{align*} 
where $\caa^*$ (resp. $\la_1^*$) stands for that 
the action of $\caa$ (resp. $\la_1$) on $K$ is faithful. 

By Theorem \ref{thAHK}, 
the $k$-rationality of $K(x_1,x_2,x_3)^G$ is reduced to 
that of $K(y_2,y_3)^{\langle \caa \rangle}$ 
and $K(y_2,y_3)^{\langle \caa, \lambda_1 \rangle}$, 
which is the rationality problem of two-dimensional algebraic tori. 
By Theorem \ref{thm:Vos}, they are both $k$-rational. 

On the other hand, we have 
\[
K(y_1, y_2, y_3)^{\langle \rho_1 \rangle}=K(z_1,z_2,z_3)
\] 
where $z_1=y_1^2$, $z_2=y_1y_2$, $z_3=y_1y_3$. 
The actions of $\caa$ and $\lambda_1$ 
on $K(z_1, z_2, z_3)$ are given by 
\begin{align*} 
\caa^* &: z_1\mapsto \frac{1}{z_1},\ z_2\mapsto \frac{1}{z_3},\ z_3\mapsto \frac{z_2}{z_1}, & 
\lambda_1^* &: z_1\mapsto \frac{1}{z_1},\ z_2\mapsto \frac{1}{z_2},\ z_3\mapsto \frac{z_3}{z_1}.
\end{align*} 
Hence the $k$-rationality of 
$\big(K(x_1,x_2,x_3)^G\big)^{\langle \rho_1 \rangle}$ 
for $G=G_{4,3,1}$ (resp. $G_{4,7,1}$) 
is reduced to the $k$-rationality of 
three-dimensional algebraic torus for $G_{4,2,2}$ 
(resp. $G_{4,6,3}$). 
By Theorem \ref{thm:Kun}, they are all $k$-rational (see also the last paragraph of Section \ref{sec:notation}).\\

%%%%%%%%%%%%%%%%%%%%%%%%%%%%%%%%%%%%%%%%%%%%%%%%%%%%%%%%%%%%%%%%
Case 3: $H=\langle -\caa^2 \rangle\simeq \cC_2$.\\

There exist two cases: $G=G_{4,3,1}=\langle\caa,-I_3\rangle\simeq \cC_4\times \cC_2$, 
$G_{4,7,1}=\langle\caa,\la_1,-I_3\rangle\simeq \cD_4\times \cC_2$.  
We have 
\[
K(x_1, x_2, x_3)^{\langle -\caa^2 \rangle}=K(y_1, y_2, y_3)
\] 
where 
\begin{align*} 
y_1 = x_1,\ 
y_2 = x_2,\ 
y_3 = \left(x_1+\frac{1}{x_1}+x_2+\frac{1}{x_2}\right)\left(x_3+\frac{1}{x_3}\right). 
\end{align*}  
The actions of $\caa$, $\lambda_1$, $-I_3$ and $\rho_1$  
on $K(y_1,y_2,y_3)$ are given by 
\begin{align*} 
\caa^* &: y_1\mapsto y_2,\ y_2\mapsto \frac{1}{y_1},\ y_3\mapsto y_3, & 
\lambda_1^* &: y_1\mapsto \frac{1}{y_1},\ y_2\mapsto y_2,\ y_3\mapsto y_3, \\ 
-I_3^* &: y_1\mapsto \frac{1}{y_1},\ y_2\mapsto \frac{1}{y_2},\ y_3\mapsto y_3, & 
\rho_1 &: y_1\mapsto -y_1,\ y_2\mapsto -y_2,\ y_3\mapsto y_3 
\end{align*} 
where $\caa^*$ (resp. $\la_1^*$, $-I_3^*$) stands for that 
the action of $\caa$ (resp. $\la_1$, $-I_3$) on $K$ is faithful. 

By Theorem \ref{thAHK}, 
the $k$-rationality of $K(x_1,x_2,x_3)^G$ can be reduced to 
that of $K(y_1,y_2)^{\langle \caa, -I_3 \rangle}$ 
and $K(y_1,y_2)^{\langle \caa, \lambda_1, -I_3 \rangle}$, 
which is the rationality problem of two-dimensional algebraic tori. 
By Theorem \ref{thm:Vos}, they are both $k$-rational. 

%%%%%%%%%%%%%%%%%%%%%%%%%%%%%%%%%%%%%%%%%%%%

On the other hand, we have 
\[
K(y_1,y_2,y_3)^{\langle \rho_1 \rangle}=K(z_1,z_2,z_3)
\] 
where $z_1=y_2/y_1$, $z_2=y_1y_2$, $z_3=y_3$. 
The actions of $\caa$, $-I_3$ and $\lambda_1$ on $K(z_1,z_2,z_3)$ 
are 
\begin{align*} 
\caa^* &: z_1\mapsto \frac{1}{z_2}, \ z_2\mapsto z_1, \ z_3\mapsto z_3, & 
\lambda^*_1 &: z_1\mapsto z_2, \ z_2\mapsto z_1, \ z_3\mapsto z_3, \\ 
-I_3^* &: z_1\mapsto \frac{1}{z_1}, \ z_2\mapsto \frac{1}{z_2}, \ z_3\mapsto z_3. 
\end{align*}  
By Theorem \ref{thAHK}, the $k$-rationality 
of $\big(K(x_1,x_2,x_3)^G\big)^{\langle \rho_1 \rangle}$ 
is reduced to the $k$-rationality of $K(z_1,z_2)^{\langle \caa, -I_3 \rangle}$ 
and $K(z_1,z_2)^{\langle \caa, \lambda_1, -I_3 \rangle}$, 
which is the rationality problem of two-dimensional algebraic torus. 
By Theorem \ref{thm:Vos} again, they are all $k$-rational. 

%%%%%%%%%%%%%%%%%%%%%%%%%%%%%%%%%%%%%%%%%%%%%%%%%%%%%%%%%%%%%%%%%%%
%
\subsection{The 5th crystal system (I)} \label{subsec:5} 
%
%%%%%%%%%%%%%%%%%%%%%%%%%%%%%%%%%%%%%%%%%%%%%%%%%%%%%%%%%%%%%%%%%%%

We consider the following five groups:  
\begin{align*}
G_{5,1,1}&=\langle \cb\rangle\simeq\cC_3,&
G_{5,2,1}&=\langle \cb,-I_3\rangle\simeq\cC_6,\\
G_{5,3,1}&=\langle \cb,-\alpha\rangle\simeq\cS_3,&
G_{5,4,1}&=\langle \cb,\alpha\rangle\simeq\cS_3,&
G_{5,5,1}&=\langle \cb,\alpha,-I_3\rangle\simeq\cD_6.
\end{align*}
The actions of $\cb$, $-I_3$, $\alpha$ and $-\alpha$ 
on $K(x_1, x_2, x_3)$ are given by 
\begin{align*} 
\cb &: x_1\mapsto x_2,\ x_2\mapsto x_3,\ x_3\mapsto x_1, & 
-I_3 &: x_1\mapsto \frac{1}{x_1},\ x_2\mapsto \frac{1}{x_2},\ x_3\mapsto \frac{1}{x_3}, \\ 
\alpha &: x_1\mapsto x_2,\ x_2\mapsto x_1,\ x_3\mapsto x_3, & 
-\alpha &: x_1\mapsto \frac{1}{x_2},\ x_2\mapsto \frac{1}{x_1},\ x_3\mapsto \frac{1}{x_3}. 
\end{align*} 
The non-trivial proper normal subgroups $H$ 
of each $G=G_{5,j,1}$ are as follows: 
\begin{align*} 
G_{5,1,1} &: \ \text{none}, \\ 
G_{5,2,1} &: \ H=\langle \cb \rangle, \langle -I_3 \rangle, \\ 
G_{5,3,1} &: \ H=\langle \cb \rangle, \\ 
G_{5,4,1} &: \ H=\langle \cb \rangle, \\ 
G_{5,5,1} &: \ H=\langle \cb \rangle, \langle -I_3 \rangle, 
\langle \cb, \alpha \rangle, 
\langle \cb, -\alpha \rangle, \langle \cb, -I_3 \rangle.  
\end{align*} 

%%%%%%%%%%%%%%%%%%%%%%%%%%%%%%%%%%%%%%%%%%%%%%%%%%%%%%%%%%%%%%%%
Case 1: $\langle\cb\rangle\leq H$.\\

Define 
\begin{align*} 
X_1 = \frac{x_1-1}{x_1+1},\  
X_2 = \frac{x_2-1}{x_2+1},\  
X_3 = \frac{x_3-1}{x_3+1}.
\end{align*} 
Then the actions of $\cb$, $-I_3$, $\alpha$ and $-\alpha$ 
on $K(X_1, X_2, X_3)$ are given by 
\begin{align*} 
\cb &: X_1\mapsto X_2,\ X_2\mapsto X_3,\ X_3\mapsto X_1, & 
-I_3 &: X_1\mapsto -X_1,\ X_2\mapsto -X_2,\ X_3\mapsto -X_3, \\ 
\alpha &: X_1\mapsto X_2,\ X_2\mapsto X_1,\ X_3\mapsto X_3, & 
-\alpha &: X_1\mapsto -X_2,\ X_2\mapsto -X_1,\ X_3\mapsto -X_3 
\end{align*} 
and $\cb$ acts on $K$ trivially. 
It follows from Lemma \ref{lem:sigma3B} that 
\[
K(X_1, X_2, X_3)^{\langle \cb \rangle} = K(y_1,y_2,y_3)
\] 
where 
\begin{align*} 
y_1 &= X_1+X_2+X_3, \\ 
y_2 &= \frac{X_1X_2^2+X_2X_3^2+X_3X_1^2-3X_1X_2X_3}
{X_1^2+X_2^2+X_3^2-X_1X_2-X_2X_3-X_3X_1},\\ 
y_3 &= \frac{X_1^2X_2+X_2^2X_3+X_3^2X_1-3X_1X_2X_3}
{X_1^2+X_2^2+X_3^2-X_1X_2-X_2X_3-X_3X_1}. 
\end{align*} 
The actions of $-I_3$, $\alpha$ and $-\alpha$ 
on $K(y_1,y_2,y_3)$ are given by 
\begin{align*} 
-I_3 &: y_1\mapsto -y_1,\ y_2\mapsto -y_2,\ y_3\mapsto -y_3, & 
\alpha &: y_1\mapsto y_1,\ y_2\mapsto y_3,\ y_3\mapsto y_2, \\ 
-\alpha &: y_1\mapsto -y_1,\ y_2\mapsto -y_3,\ y_3\mapsto -y_2. 
\end{align*} 
Define 
\[
Y_2=\frac{y_2-1}{y_2+1},\ Y_3=\frac{y_3-1}{y_3+1}.
\]
Then, by Theorem \ref{thAHK},  
the $k$-rationality of $K(x_1,x_2,x_3)^G$ for each $G=G_{5,j,1}$ 
is reduced to that of $K(Y_2,Y_3)^G$, 
which is the rationality problem 
of two-dimensional purely quasi-monomial actions.  
By Theorem \ref{thm:HKK2}, they are all $k$-rational.\\

%%%%%%%%%%%%%%%%%%%%%%%%%%%%%%%%%%%%%%%%%%%%%%%%%%%%%%%%%%%%%%%%
Case 2: $H=\langle -I_3\rangle\simeq \cC_2$.\\

There exist only two cases 
$G=G_{5,2,1}=\langle\cb,-I_3\rangle\simeq \cC_6$, 
$G_{5,5,1}=\langle\cb,\al,-I_3\rangle\simeq \cD_6$. 
By Lemma \ref{lem:mI3}, we have 
\[
K(x_1, x_2, x_3)^{\langle -I_3 \rangle}=K(y_1, y_2, y_3)
\] 
where 
\begin{align*} 
y_1 = \frac{x_1x_2+1}{x_1+x_2},\ 
y_2 = \frac{x_2x_3+1}{x_2+x_3},\  
y_3 = \frac{x_3x_1+1}{x_3+x_1}.
\end{align*} 
The actions of $\cb$ and $\alpha$ on $K(y_1, y_2, y_3)$ are given by 
\begin{align*} 
\cb^* &: y_1\mapsto y_2,\ y_2\mapsto y_3,\ y_3\mapsto y_1, & 
\alpha^* &: y_1\mapsto y_1,\ y_2\mapsto y_3,\ y_3\mapsto y_2.   
\end{align*} 
Since $\langle \cb \rangle$ and $\langle \cb, \alpha \rangle$ 
act on $K$ faithfully,  
the $k$-rationality of $K(x_1,x_2,x_3)^G$ 
for $G_{5,2,1}$ (resp. $G_{5,5,1}$) is reduced to 
the rationality problem of three-dimensional algebraic torus  
corresponding to $G_{5,1,1}$ (resp. $G_{5,4,1}$). 
By Theorem \ref{thm:Kun}, they are both $k$-rational. 

%%%%%%%%%%%%%%%%%%%%%%%%%%%%%%%%%%%%%%%%%%%%%%%%%%%%%%%%%%%%%%%%%%%%

\section{Proof of Theorem \ref{thmain2}}\label{seP2}

We consider the following $15$ groups $G_{7,j,k}$ 
($1\leq j\leq 5$, $1\leq k\leq 3)$ which belong to the 
$7$th crystal system in dimension $3$: 
\begin{align*}
\hspace*{1.5cm}
G_{7,1,k}&=\langle \tau_k,\lambda_k,\cb\rangle&\hspace{-1cm} 
&\simeq \cA_4&\hspace{-1cm} 
&\simeq (\cC_2\times \cC_2)\rtimes \cC_3, \\
G_{7,2,k}&=\langle \tau_k,\lambda_k,\cb,-I_3\rangle&\hspace{-1cm} 
&\simeq \cA_4\times \cC_2&\hspace{-1cm} 
&\simeq (\cC_2\times \cC_2\times \cC_2)\rtimes \cC_3,\\
G_{7,3,k}&=\langle \tau_k,\lambda_k,\cb,-\beta_k\rangle&\hspace{-1cm} 
&\simeq \cS_4&\hspace{-1cm} 
&\simeq (\cC_2\times \cC_2)\rtimes \cS_3, \\
G_{7,4,k}&=\langle \tau_k,\lambda_k,\cb,\beta_k\rangle&\hspace{-1cm} 
&\simeq \cS_4&\hspace{-1cm} 
&\simeq (\cC_2\times \cC_2)\rtimes \cS_3, \\
G_{7,5,k}&=\langle \tau_k,\lambda_k,\cb,\beta_k,-I_3\rangle&\hspace{-1cm} 
&\simeq \cS_4\times \cC_2&\hspace{-1cm} 
&\simeq (\cC_2\times \cC_2\times \cC_2)\rtimes \cS_3.
\end{align*}

The non-trivial proper normal subgroups $H$ of each $G=G_{7,j,k}$ 
$(1\leq k\leq 3)$ are: 
\begin{align*} 
G_{7,1,k} &: H=\langle \tau_k, \lambda_k \rangle, \\ 
G_{7,2,k} &: H=\langle \tau_k, \lambda_k \rangle, 
\langle \tau_k, \lambda_k, \cb \rangle, 
\langle -I_3 \rangle, \langle \tau_k, \lambda_k, -I_3 \rangle,\\ 
G_{7,3,k} &: H=\langle \tau_k, \lambda_k \rangle, 
\langle \tau_k, \lambda_k, \cb \rangle,  \\ 
G_{7,4,k} &: H=\langle \tau_k, \lambda_k \rangle, 
\langle \tau_k, \lambda_k, \cb \rangle,  \\ 
G_{7,5,k} &: H=\langle \tau_k, \lambda_k \rangle,  
\langle \tau_k, \lambda_k, \cb \rangle, 
\langle -I_3 \rangle, 
\langle \tau_k, \lambda_k, -I_3 \rangle, \\ 
&\hspace{37pt} 
\langle \tau_k, \lambda_k, \cb, \beta_k \rangle, 
\langle \tau_k, \lambda_k, \cb, -\beta_k \rangle, 
\langle \tau_k, \lambda_k, \cb, -I_3 \rangle.  
\end{align*} 

%%%%%%%%%%%%%%%%%%%%%%%%%%%%%%%%%%%%%%%%%%%%%%%%%%%%%%%%%%%%%%%%%
%
\subsection{The 7th crystal system (I), (II)}  \label{subsec:71}
%
%%%%%%%%%%%%%%%%%%%%%%%%%%%%%%%%%%%%%%%%%%%%%%%%%%%%%%%%%%%%%%%%%

Let $P_1\in GL_3(\bZ)$ be the matrix defined as Equation (\ref{matP1P2}) 
in Lemma \ref{lem:conversion}. 
Then we may confirm that 
\[
G_{7,j,2} \sim P_1^{-1}G_{7,j,1}P_1 \ \ \text{ for each } 1\leq j\leq 5
\]
where $\sim$ means $GL_3(\bZ)$-conjugation. 
Let $\rho_1$ be $K$-automorphism on $K(x_1,x_2,x_3)$ defined by 
\begin{align*} 
\rho_1 &: x_1\mapsto -x_1,\ x_2\mapsto -x_2,\ x_3\mapsto -x_3
\end{align*} 
as in Lemma \ref{lem:conversion}. 
By Lemma \ref{lem:conversion},  
it suffices to prove the $k$-rationality of 
$K(x_1,x_2,x_3)^G$ and $\big(K(x_1,x_2,x_3)^G\big)^{\langle \rho_1 \rangle}$ 
for $G=G_{7,j,1}$ ($1\leq j\leq 5$) 
which belong to the 7th crystal system (I).\\

%%%%%%%%%%%%%%%%%%%%%%%%%%%%%%%%%%%%%%%%%%%%%%%%%%%%%%%%%%%%%%%%
Case 1: $\langle \tau_1, \lambda_1 \rangle \leq H$.\\

The actions of $\tau_1$, $\lambda_1$, $\cb$, $-I_3$, $\beta_1$ 
and $\rho_1$ on $K(x_1,x_2,x_3)$ are given by 
\begin{align*} 
\tau_1 &: x_1\mapsto \frac{1}{x_1},\ x_2\mapsto \frac{1}{x_2},\ x_3\mapsto x_3 , & 
\lambda_1 &: x_1\mapsto \frac{1}{x_1},\ x_2\mapsto x_2,\ x_3\mapsto \frac{1}{x_3}, \\ 
\cb &: x_1\mapsto x_2,\ x_2\mapsto x_3,\ x_3\mapsto x_1, & 
-I_3 &: x_1\mapsto \frac{1}{x_1},\ x_2\mapsto \frac{1}{x_2},\ x_3\mapsto \frac{1}{x_3}, \\ 
\beta_1 &: x_1\mapsto \frac{1}{x_2},\ x_2\mapsto \frac{1}{x_1},\ x_3\mapsto x_3, & 
\rho_1 &: x_1\mapsto -x_1,\ x_2\mapsto -x_2,\ x_3\mapsto -x_3. 
\end{align*}
By Lemma \ref{lem:tau1lambda1}, we have 
\[
K(x_1,x_2,x_3)^{\langle \tau_1, \lambda_1 \rangle}=K(y_1,y_2,y_3)
\] 
where 
\begin{align*} 
y_1 = \frac{-x_1+x_2+x_3-x_1x_2x_3}{1-x_1x_2+x_2x_3-x_3x_1},\ 
y_2 = \frac{x_1-x_2+x_3-x_1x_2x_3}{1-x_1x_2-x_2x_3+x_3x_1},\ 
y_3 = \frac{x_1+x_2-x_3-x_1x_2x_3}{1+x_1x_2-x_2x_3-x_3x_1}.
\end{align*}  
The actions of $\cb$, $-I_3$, $\beta_1$ and $\rho_1$ 
on $K(y_1,y_2,y_3)$ are given by 
\begin{align*}
\cb &: y_1\mapsto y_2,\ y_2\mapsto y_3,\ y_3\mapsto y_1, & 
-I_3 &: y_1\mapsto \frac{1}{y_1},\ y_2\mapsto \frac{1}{y_2},\ y_3\mapsto \frac{1}{y_3}, \\ 
\beta_1 &: y_1\mapsto y_2,\ y_2\mapsto y_1,\ y_3\mapsto y_3, & 
\rho_1 &: y_1\mapsto -y_1,\ y_2\mapsto -y_2,\ y_3\mapsto -y_3. 
\end{align*} 
Hence the $k$-rationality of $K(x_1,x_2,x_3)^{G_{7,j,1}}$ 
is reduced to that of $K(y_1,y_2,y_3)^{G_{5,j,1}}$. 
By Theorem \ref{thmain1}, they are all $k$-rational. 

On the other hand, 
we have $K(y_1,y_2,y_3)^{\langle \rho_1 \rangle}=K(z_1,z_2,z_3)$ 
where $z_1=y_2y_3$, $z_2=y_1y_3$, $z_3=y_1y_2$. 
The actions of $\cb$, $-I_3$ and $\beta_1$ on $K(z_1,z_2,z_3)$ 
are the same as the actions on $K(y_1,y_2,y_3)$. 
Hence the $k$-rationality of $\big(K(y_1,y_2,y_3)^{G_{7,j,1}}\big)^{\langle \rho_1 \rangle}$ 
is reduced to the $k$-rationality of $K(z_1,z_2,z_3)^{G_{5,j,1}}$. 
By Theorem \ref{thmain1} again, they are all $k$-rational. \\

%%%%%%%%%%%%%%%%%%%%%%%%%%%%%%%%%%%%%%%%%%%%%%%%%%%%%%%%%%%%%%%%
Case 2: $H=\langle -I_3\rangle\simeq\cC_2$.\\

There exist two cases: 
$G=G_{7,2,1}=\langle\ta_1,\la_1,\cb,-I_3\rangle$ $\simeq$ $\cA_4\times\cC_2$, 
$G_{7,5,1}=\langle\ta_1,\la_1,\cb,\be_1,-I_3\rangle$ $\simeq$ $\cS_4\times\cC_2$. 
By Lemma \ref{lem:mI3}, we have 
\[
K(x_1,x_2,x_3)^{\langle -I_3 \rangle}=K(y_1,y_2,y_3)
\] 
where 
\begin{align*} 
y_1 = \frac{x_1x_2+1}{x_1+x_2},\  
y_2 = \frac{x_2x_3+1}{x_2+x_3},\ 
y_3 = \frac{x_3x_1+1}{x_3+x_1}.  
\end{align*} 
The actions of $\tau_1$, $\lambda_1$, $\cb$, $\beta_1$ and $\rho_1$ 
on $K(y_1,y_2,y_3)$ are given by 
\begin{align*} 
\tau_1^* &: y_1\mapsto y_1,\ y_2\mapsto \frac{1}{y_2},\ y_3\mapsto \frac{1}{y_3}, & 
\lambda_1^* &: y_1\mapsto \frac{1}{y_1},\ y_2\mapsto \frac{1}{y_2},\ y_3\mapsto y_3, \\ 
\cb^* &: y_1\mapsto y_2,\ y_2\mapsto y_3,\ y_3\mapsto y_1, & 
\beta_1^* &:  y_1\mapsto y_1,\ y_2\mapsto \frac{1}{y_3},\ y_3\mapsto \frac{1}{y_2}, \\ 
\rho_1 &: y_1\mapsto -y_1,\ y_2\mapsto -y_2,\ y_3\mapsto -y_3. 
\end{align*} 
Hence, as in Case $1$, the $k$-rationality of $K(x_1,x_2,x_3)^G$ and 
$\big(K(x_1,x_2,x_3)^G\big)^{\langle \rho_1 \rangle}$ 
for $G=G_{7,2,1}$ (resp. $G_{7,5,1}$)  
can be reduced to the $k$-rationality of three-dimensional algebraic torus 
for $G_{7,1,k}$ (resp. $G_{7,4,k}$) with $k=1,2$. 
By Theorem \ref{thm:Kun}, 
they are all $k$-rational. 

%%%%%%%%%%%%%%%%%%%%%%%%%%%%%%%%%%%%%%%%%%

\subsection{The 7th crystal system (III)}

We treat the remaining five cases: 
\begin{align*}
\hspace*{1.5cm}
G_{7,1,3}&=\langle \tau_3,\lambda_3,\cb\rangle&\hspace{-1cm} 
&\simeq \cA_4&\hspace{-1cm} 
&\simeq (\cC_2\times \cC_2)\rtimes \cC_3, \\
G_{7,2,3}&=\langle \tau_3,\lambda_3,\cb,-I_3\rangle&\hspace{-1cm} 
&\simeq \cA_4\times \cC_2&\hspace{-1cm} 
&\simeq (\cC_2\times \cC_2\times \cC_2)\rtimes \cC_3,\\
G_{7,3,3}&=\langle \tau_3,\lambda_3,\cb,-\beta_3\rangle&\hspace{-1cm} 
&\simeq \cS_4&\hspace{-1cm} 
&\simeq (\cC_2\times \cC_2)\rtimes \cS_3, \\
G_{7,4,3}&=\langle \tau_3,\lambda_3,\cb,\beta_3\rangle&\hspace{-1cm} 
&\simeq \cS_4&\hspace{-1cm} 
&\simeq (\cC_2\times \cC_2)\rtimes \cS_3, \\
G_{7,5,3}&=\langle \tau_3,\lambda_3,\cb,\beta_3,-I_3\rangle&\hspace{-1cm} 
&\simeq \cS_4\times \cC_2&\hspace{-1cm} 
&\simeq (\cC_2\times \cC_2\times \cC_2)\rtimes \cS_3.
\end{align*}
The actions of 
$\ta_3$, $\la_3$, $\cb$, $-\be_3$, $\be_3$ and $-I_3$  
on $K(x_1,x_2,x_3)$ are given by 
\begin{align*}
\ta_3 &: x_1\mapsto x_2,\ x_2\mapsto x_1,\ x_3\mapsto \frac{1}{x_1x_2x_3},&
\la_3 &: x_1\mapsto x_3,\ x_2\mapsto \frac{1}{x_1x_2x_3},\ x_3\mapsto x_1,\\
\cb &: x_1\mapsto  x_2,\ x_2\mapsto x_3,\ x_3\mapsto x_1,& 
-\be_3 &: x_1\mapsto \frac{1}{x_1},\ x_2\mapsto  \frac{1}{x_2},\ 
x_3\mapsto x_1x_2x_3,\\
\be_3 &: x_1\mapsto x_1,\ x_2\mapsto x_2,\ x_3\mapsto \frac{1}{x_1x_2x_3},& 
-I_3 &: x_1\mapsto \frac{1}{x_1},\ 
x_2\mapsto \frac{1}{x_2},\ x_3\mapsto \frac{1}{x_3}.
\end{align*} 
We will split the problem into seven cases according 
to each $H\leq G_{7,j,3}$.\\

%%%%%%%%%%%%%%%%%%%%%%%%%%%%%%%%%%%%%%%%%%%%%%%%
Case 1: $H=\langle-I_3\rangle\simeq \cC_2\leq G=G_{7,j,3}$, $(j=2,5)$.\\

Step 1. 
By Lemma \ref{lem:mI3}, we have
\begin{align*}
K(x_1,x_2,x_3)^{\langle-I_3\rangle}=K(u_1,u_2,u_3)
\end{align*}
where 
\begin{align*}
u_1=\frac{x_1x_2+1}{x_1+x_2},\ 
u_2=\frac{x_2x_3+1}{x_2+x_3},\
u_3=\frac{x_1x_3+1}{x_1+x_3}.
\end{align*}
Define
\[
v_1=\frac{u_1+1}{u_1-1},\ 
v_2=\frac{u_2+1}{u_2-1},\ 
v_3=\frac{u_3+1}{u_3-1}.
\]
Then $K(u_1,u_2,u_3)=K(v_1,v_2,v_3)$ and 
the actions of $\ta_3$, $\la_3$, $\cb$ and $-\be_3$ 
on $K(v_1,v_2,v_3)$ are 
\begin{align*}
\ta_3^* &: 
v_1\mapsto v_1,\ 
v_2\mapsto -\frac{v_1v_2+v_1v_3+v_2v_3+v_1v_2v_3}{v_2(1+v_1+v_2+v_3)},\ 
v_3\mapsto -\frac{v_1v_2+v_1v_3+v_2v_3+v_1v_2v_3}{v_3(1+v_1+v_2+v_3)},\\
\la_3^*&:
v_1\mapsto -\frac{v_1v_2+v_1v_3+v_2v_3+v_1v_2v_3}{v_1(1+v_1+v_2+v_3)},\ 
v_2\mapsto -\frac{v_1v_2+v_1v_3+v_2v_3+v_1v_2v_3}{v_2(1+v_1+v_2+v_3)},\ 
v_3\mapsto v_3,\\
\cb^*&:
v_1\mapsto v_2,\ 
v_2\mapsto v_3,\ 
v_3\mapsto v_1,\\
-\be_3^*&:v_1\mapsto v_1,\ 
v_2\mapsto -\frac{v_1v_2+v_1v_3+v_2v_3+v_1v_2v_3}{v_3(1+v_1+v_2+v_3)},\ 
v_3\mapsto -\frac{v_1v_2+v_1v_3+v_2v_3+v_1v_2v_3}{v_2(1+v_1+v_2+v_3)}.
\end{align*}
We also define 
\begin{align*}
t_1&=\frac{2v_2(v_1+1)(v_1+v_2)(v_1+v_3)}{(v_1-v_2)(2v_1v_2+v_1^2v_2+
v_1v_2^2+v_1v_3+v_2v_3+
2v_1v_2v_3)},\\
t_2&=-\frac{(v_1+v_3)(-2v_2-v_1v_2-v_2^2+v_1v_3-
v_2v_3)}{(v_2+v_3)(2v_1+v_1^2+v_1v_2+v_1v_3-
v_2v_3)},\\
t_3&=\frac{2v_1(v_2+1)(v_1+v_2)(v_2+v_3)}{(v_1-
v_2)(2v_1v_2+v_1^2v_2+v_1v_2^2+v_1v_3+v_2v_3+2v_1v_2v_3)}.
\end{align*}
Then we may see that 
$K(v_1,v_2,v_3)=K(t_1,t_2,t_3)$. 
This may be confirmed directly as 
\begin{align*}
v_1&=\frac{(t_1 - t_3+2) (t_1 + t_3+2) (t_1 - t_2 t_3)}{D},\\
v_2&=\frac{(t_1 - t_3+2) (t_1 + t_3-2) (t_1 - t_2 t_3)}{D},\\
v_3&=\frac{(t_2-1) (2 - t_1 - t_3) (2 + t_1 + t_3) (t_1 - t_2 t_3)}{(t_2+1) D}
\end{align*}
where 
$D=4 t_1 - t_1^3 + 4 t_2 t_3 + t_1^2 t_2 t_3 + t_1 t_3^2 - t_2 t_3^3$. 
We also see that 
\begin{align*}
t_1&=-\frac{(x_1x_2+1)(x_1x_3-1)(x_2x_3-1)}{(x_1-x_3)(x_1x_2^2x_3-1)},\\
t_2&=\frac{(x_2x_3-1)(1-2x_1+x_1x_2+x_1x_3-2x_1x_2x_3+x_1^2x_2x_3)}{(x_1x_2-1)(1-2x_3+x_1x_3+x_2x_3-2x_1x_2x_3+x_1x_2x_3^2)},\\
t_3&=-\frac{(x_1x_2-1)(x_1x_3-1)(x_2x_3+1)}{(x_1-x_3)(x_1x_2^2x_3-1)}.
\end{align*}
Define
\begin{align*}
p_1'=\frac{-4t_2^2+t_1^2t_2^2-2t_1t_2^3t_3-t_3^2+
2t_2^2t_3^2}{(t_1t_2-t_3)^2},\ 
p_2=-\frac{2}{t_1t_2-t_3},\ 
p_3=\frac{1}{t_2}.
\end{align*}
Then we have $K(t_1,t_2,t_3)=K(p_1',p_2,p_3)$ and 
\[
\ta_3^*(p_1')=p_1',\ 
\la_3^*(p_1')=p_1',\ 
\cb^*(p_1')=\frac{-1+p_3^2+p_1'p_3^2}{p_2^2},\ 
-\be_3^*(p_1')=\frac{p_1'p_3^2}{p_2^2}.
\]
Hence we put 
\[
p_1=\sum_{\sigma\in\langle\cb,-\be_3\rangle}\sigma(p_1').
\]
Then $K(p_1,p_2,p_3)=K(p_1',p_2,p_3)=K(t_1,t_2,t_3)$ and 
\begin{align}
\ta_3^* &: %\sqrt{a}\mapsto-\sqrt{a}, 
p_1\mapsto p_1,\ 
p_2\mapsto -p_2,\ 
p_3\mapsto -p_3,\nonumber\\
\la_3^* &: %\sqrt{b}\mapsto-\sqrt{b}, 
p_1\mapsto  p_1,\ 
p_2\mapsto -p_2,\ 
p_3\mapsto p_3,\label{actG}\\
\cb^* &: %\te\mapsto\te', 
p_1\mapsto p_1,\
p_2\mapsto \frac{1}{p_3},\ 
p_3\mapsto \frac{p_2}{p_3},\nonumber\\
-\be_3^* &: %\sqrt{d}\mapsto-\sqrt{d},
p_1\mapsto p_1,\ 
p_2\mapsto -p_3,\ 
p_3\mapsto -p_2.\nonumber
\end{align}

Step 2. 
Let $K^\prime=K(\sqrt{a})$ $(a\in k)$ be a quadratic extension of $K$ 
and $K^\prime(X,Y,Z)$ be the rational function field 
with $3$ variables $X,Y,Z$ over $K^\prime$. 

We now consider the following action of 
$G^\prime=\langle \gamma,\ta_3,\lambda_3,\cb,-\be_3\rangle
\simeq \langle\gamma,\ta_3,\lambda_3\rangle\rtimes\langle\cb,-\ta_3\be_3\rangle
\simeq (\cC_2\times\cC_2\times\cC_2)\rtimes \cS_3\simeq\cC_2\wr\cS_3$ 
on $K^\prime(X,Y,Z)$: 
\begin{align*}
\gamma^* &: %\sqrt{a}\mapsto-\sqrt{a}, 
X\mapsto -X,\ 
Y\mapsto -Y,\ 
Z\mapsto -Z,\\
\ta_3^* &: %\sqrt{a}\mapsto-\sqrt{a}, 
X\mapsto X,\ 
Y\mapsto -Y,\ 
Z\mapsto -Z,\\
\la_3^* &: %\sqrt{b}\mapsto-\sqrt{b}, 
X\mapsto  X,\ 
Y\mapsto -Y,\ 
Z\mapsto Z,\\
\cb^* &: %\te\mapsto\te', 
X\mapsto Z,\ 
Y\mapsto X,\ 
Z\mapsto Y,\\
-\be_3^* &: %\sqrt{d}\mapsto-\sqrt{d},
X\mapsto X,\ 
Y\mapsto -Z,\ 
Z\mapsto -Y
\end{align*}
where %$G^\prime=\langle \gamma,G \rangle$, 
$(K^{\prime})^{\langle\gamma\rangle}=K$ and 
$\cC_2\wr\cS_3$ is the wreath product of $\cC_2$ and $\cS_3$. 

Define 
\[
P_1=X,\ P_2=\frac{Y}{X},\ P_3=\frac{Z}{X}. 
\]
Then $\gamma,\ta_3,\lambda_3,\cb$ and $-\be_3$ 
act on $K(P_1,P_2,P_3)$ by 
\begin{align*}
\gamma^* &: %\sqrt{a}\mapsto-\sqrt{a}, 
P_1\mapsto -P_1,\ 
P_2\mapsto P_2,\ 
P_3\mapsto P_3,\\
\ta_3^* &: %\sqrt{a}\mapsto-\sqrt{a}, 
P_1\mapsto P_1,\ 
P_2\mapsto -P_2,\ 
P_3\mapsto -P_3,\\
\la_3^* &: %\sqrt{b}\mapsto-\sqrt{b}, 
P_1\mapsto  P_1,\ 
P_2\mapsto -P_2,\ 
P_3\mapsto P_3,\\
\cb^* &: %\te\mapsto\te', 
P_1\mapsto P_1 P_3,\ 
P_2\mapsto \frac{1}{P_3},\ 
P_3\mapsto \frac{P_2}{P_3},\\
-\be_3^* &: %\sqrt{d}\mapsto-\sqrt{d},
P_1\mapsto P_1,\ 
P_2\mapsto -P_3,\ 
P_3\mapsto -P_2.
\end{align*}
We may take $K^\prime=K(\sqrt{a})$ where 
$\gamma(\sqrt{a})=-\sqrt{a}$, $\ta_3(\sqrt{a})=\la_3(\sqrt{a})
=\cb(\sqrt{a})=-\be_3(\sqrt{a})=\sqrt{a}$. 
Define $P_1^\prime=P_1 \sqrt{a}$. 
Then we have 
$K^\prime(P_1,P_2,P_3)^{\langle\gamma\rangle}=K(P_1^\prime,P_2,P_3)$ 
and $G$ acts on $K(P_1^\prime,P_2,P_3)$ by 
\begin{align*}
\ta_3^* &: %\sqrt{a}\mapsto-\sqrt{a}, 
P_1^\prime\mapsto P_1^\prime,\ 
P_2\mapsto -P_2,\ 
P_3\mapsto -P_3,\\
\la_3^* &: %\sqrt{b}\mapsto-\sqrt{b}, 
P_1^\prime\mapsto  P_1^\prime,\ 
P_2\mapsto -P_2,\ 
P_3\mapsto P_3,\\
\cb^* &: %\te\mapsto\te', 
P_1^\prime\mapsto P_1^\prime P_3,\ 
P_2\mapsto \frac{1}{P_3},\ 
P_3\mapsto \frac{P_2}{P_3},\\
-\be_3^* &: %\sqrt{d}\mapsto-\sqrt{d},
P_1^\prime\mapsto P_1^\prime,\ 
P_2\mapsto -P_3,\ 
P_3\mapsto -P_2.
\end{align*}

By Theorem \ref{thHK}, there exists $P_1^{\prime\prime}\in K(P_1^\prime,P_2,P_3)$ 
such that $K(P_1^\prime,P_2,P_3)=K(P_1^{\prime\prime},P_2,P_3)$ and 
$\sigma(P_1^{\prime\prime})=P_1^{\prime\prime}$ for any $\sigma\in G$. 
Hence the action of $G$ on $K(P_1^{\prime\prime},P_2,P_3)$ coincides 
with that on $K(p_1,p_2,p_3)$ given in (\ref{actG}). 

On the other hand, 
$K(P_1^{\prime\prime},P_2,P_3)^G=K^{\prime}(X,Y,Z)^{G^\prime}$ 
is $k$-rational by Theorem \ref{thEM} (or Theorem \ref{thHK}). 
This implies that 
$K(x_1,x_2,x_3)^G=K(p_1,p_2,p_3)^{G/H}$ is $k$-rational for $G=G_{7,2,3}$ and $G_{7,5,3}$. \\

%%%%%%%%%%%%%%%%%%%%%%%%%%%%%%%%%%%%%%%%%%%%%%%%
Case 2: $H=\langle\ta_3,\la_3\rangle\simeq \cC_2\times \cC_2\leq G=G_{7,j,3}$, 
$(j=1,2,3,4,5)$.\\

By Lemma \ref{lemV42}, we have 
$K(x_1,x_2,x_3)^{\langle \ta_3,\la_3\rangle}=K(u_1,u_2,u_3)$ where 
\begin{align*}
\cb^* &: %\te\mapsto\te',
u_1\mapsto  u_2,\ u_2\mapsto  u_3,\ u_3\mapsto u_1,\\
-\be_3^* &: %\de\mapsto\de',
u_1\mapsto \frac{-u_1+u_2+u_3}{u_2u_3},\ 
u_2\mapsto \frac{u_1+u_2-u_3}{u_1u_2},\ 
u_3\mapsto \frac{u_1-u_2+u_3}{u_1u_3},\\
\be_3^* &: %\ga\mapsto-\ga,
u_1\mapsto u_1,\ u_2\mapsto u_3,\ u_3\mapsto u_2,\\
-I_3^* &: %\ep\mapsto-\ep,
u_1\mapsto \frac{-u_1+u_2+u_3}{u_2u_3},\ 
u_2\mapsto \frac{u_1-u_2+u_3}{u_1u_3},\ 
u_3\mapsto \frac{u_1+u_2-u_3}{u_1u_2}.
\end{align*}
(see also \cite[page 106]{HKY11}).  
For $G=G_{7,1,3}=\langle \tau_3,\lambda_3,\cb\rangle$ 
and $G_{7,4,3}=\langle \tau_3,\lambda_3,\cb,\beta_3\rangle$, 
it follows from Theorem \ref{thEM} (or Theorem \ref{thHK}) that 
$K(x_1,x_2,x_3)^G$ is $k$-rational. 

However, we do not know the rationality 
of $K(x_1,x_2,x_3)^G$ for $G=G_{7,2,3}$, $G_{7,3,3}$, $G_{7,5,3}$. 
\\

%%%%%%%%%%%%%%%%%%%%%%%%%%%%%%%%%%%%%%%%%%%%%%%%
Case 3: $H=\langle\ta_3,\la_3,-I_3\rangle\simeq \cC_2\times \cC_2\times \cC_2
\leq G=G_{7,j,3}$, $(j=2,5)$.\\

By Lemma \ref{lemV42}, we have 
$K(x_1,x_2,x_3)^{\langle \ta_3,\la_3\rangle}=K(u_1,u_2,u_3)$ 
where $u_1,u_2,u_3$ and the actions 
of $\cb,-\be_3,\be_3,-I_3$ on $K(u_1,u_2,u_3)$ 
are given as in Case $2$. 
We see that 
$K(x_1,x_2,x_3)^{\langle\ta_3,\la_3,-I_3\rangle}=K(t_1,t_2,t_3)$ 
where $t_i=u_i+(-I_3)(u_i)$ $(1\leq i\leq 3)$. 
The actions of $\cb$ and $\be_3$ on 
$K(t_1,t_2,t_3)$ are given by 
\begin{align}
\cb^* &: %\te\mapsto \te', 
t_1\mapsto  t_2,\ t_2\mapsto t_3,\ t_3\mapsto t_1, 
&\be_3^* &: %\ga\mapsto -\ga, 
t_1\mapsto t_1,\ t_2\mapsto t_3,\ t_3\mapsto t_2.\label{actt}
\end{align}
(see \cite[page 106]{HKY11}). 
By Theorem \ref{thEM} (or Theorem \ref{thHK}), 
$K(x_1,x_2,x_3)^G$ is $k$-rational for $G=G_{7,2,3}, G_{7,5,3}$.\\

%%%%%%%%%%%%%%%%%%%%%%%%%%%%%%%%%%%%%%%%%%%%%%%%
Case 4: $H=\langle\ta_3,\la_3,\cb\rangle\simeq \cA_4\leq G=G_{7,j,3}$, $(j=2,3,4,5)$.\\

By \cite[page 290]{HK10}, there exist $A,B,C\in K(x_1,x_2,x_3)^{\langle\cb\rangle}$ 
such that 
$K(x_1,x_2,x_3)^{\langle\ta_3,\la_3,\cb\rangle}$ $=$ 
$K(u_1,u_2,u_3)^{\langle\cb\rangle}$ $=$ $K(A,B,C)$ 
and $-\be_3,\be_3,-I_3$ act on $K(A,B,C)$ by
\begin{align*}
-\be_3^* &: A\mapsto \frac{-A+5C-7AC^2+27C^3}{1-AC+7C^2+AC^3},\ 
B\mapsto \frac{4(1-AC+7C^2+AC^3)}{aB(1-A^2+4AC)(1+3C^2)},\ 
C\mapsto C,\\
\be_3^* & : A\mapsto -A,\ B\mapsto -B,\ C\mapsto -C,\\
-I_3^* &: A\mapsto -\frac{-A+5C-7AC^2+27C^3}{1-AC+7C^2+AC^3},\ 
B\mapsto -\frac{4(1-AC+7C^2+AC^3)}{aB(1-A^2+4AC)(1+3C^2)},\ C\mapsto -C.
\end{align*}
Define
\[
s_1=C,\ 
s_2=\frac{1-AC+7C^2+AC^3}{1+3C^2},\ 
s_3=\frac{2C(C^2-1)(1+3C^2)}{B(1-AC+7C^2+AC^3)}.
\]
Then $K(A,B,C)=K(s_1,s_2,s_3)$ 
and $-\be_3,\be_3,-I_3$ act on $K(s_1,s_2,s_3)$ by
\begin{align}
-\be_3^* &: %\de\mapsto-\de, 
s_1\mapsto s_1,\ 
s_2\mapsto \frac{1+3s_1^2}{s_2},\ 
s_3\mapsto \frac{-1-6s_1^2-9s_1^4+2s_2+10s_1^2s_2
+4s_1^4s_2-s_2^2-3s_1^2s_2^2}{s_2s_3},\nonumber\\
\be_3^* &: %\ga\mapsto-\ga, 
s_1\mapsto -s_1,\ s_2\mapsto s_2,\ s_3\mapsto s_3,\label{acts}\\
-I_3^* &: %\ep\mapsto-\ep, 
s_1\mapsto -s_1,\ 
s_2\mapsto \frac{1+3s_1^2}{s_2},\ 
s_3\mapsto \frac{-1-6s_1^2-9s_1^4+2s_2+10s_1^2s_2
+4s_1^4s_2-s_2^2-3s_1^2s_2^2}{s_2s_3}.\nonumber
\end{align}

When $G=G_{7,4,3}=\langle H,\be_3\rangle$, 
it follows from Theorem \ref{thEM} (or Theorem \ref{thHK}) that 
$K(x_1,x_2,x_3)^G$ $=$ $K(s_1,s_2,s_3)^{\langle\be_3\rangle}$ is $k$-rational.\\

We note that
\[
-\be_3(s_3)=-I_3(s_3)=
\frac{1}{s_3}\left(2(1+5s_1^2+2s_1^4)-
(1+3s_1^2)\left(s_2+\frac{1+3s_1^2}{s_2}\right)\right).
\]

We will consider the rationality for 
the cases of $G_{7,2,3}$, $G_{7,3,3}$, $G_{7,5,3}$ 
in Section \ref{seP3}.\\

%%%%%%%%%%%%%%%%%%%%%%%%%%%%%%%%%%%%%%%%%%%%%%%%
Case 5: $H=\langle\ta_3,\la_3,\cb,-I_3\rangle\simeq \cA_4\times \cC_2\leq G=G_{7,5,3}$.\\

{}From Equation (\ref{actt}) and Lemma \ref{lem:sigma3B},  
$K(x_1,x_2,x_3)^{\langle \ta_3,\la_3,\cb,-I_3\rangle}=K(s_1,u,v)$ 
where $s_1=s_1(t_1,t_2,t_3)$, $u=u(t_1,t_2,t_3)$ 
and $v=v(t_1,t_2,t_3)$. 
The action of $\beta_3$ on $K(s_1,u,v)=k(\ga)(s_1,u,v)$ is given by 
\begin{align*}
\be_3 : \ga\mapsto -\ga,\ s_1\mapsto s_1,\ u\mapsto v,\ v\mapsto u.
\end{align*}
Hence $K(x_1,x_2,x_3)^G=k(s_1,u+v,\ga(u-v))$ 
is $k$-rational for $G=G_{7,5,3}$.\\

%%%%%%%%%%%%%%%%%%%%%%%%%%%%%%%%%%%%%%%%%%%%%%%%
Case 6: $H=\langle\ta_3,\la_3,\cb,-\be_3\rangle\simeq \cS_4\leq G=G_{7,5,3}$.\\

By Equation (\ref{acts}), we have 
$K(x_1,x_2,x_3)^H=K(s_1,s_2,s_3)^{\langle-\beta_3\rangle}
=k(\de)(s_1,s_2,s_3)^{\langle-\beta_3\rangle}$ 
on which $-\be_3$ acts by 
\begin{align*}
-\be_3 : \de\mapsto-\de,\ &\ s_1\mapsto s_1,\ s_2\mapsto \frac{1+3s_1^2}{s_2},\\ 
&\ s_3\mapsto \frac{-1-6s_1^2-9s_1^4+2s_2+10s_1^2s_2
+4s_1^4s_2-s_2^2-3s_1^2s_2^2}{s_2s_3}.
\end{align*}
By Lemma \ref{lemab}, we get  
$K(x_1,x_2,x_3)^H=K(t_1,t_2,t_3)$ where 
\begin{align*}
t_1&=s_1,\ 
t_2=\frac{s_2-\frac{a}{s_2}}{s_2s_3-\frac{ab}{s_2s_3}},\ 
t_3=\frac{s_3-\frac{b}{s_3}}{s_2s_3-\frac{ab}{s_2s_3}},\ 
a=1+3s_1^2,\\
b&= \frac{-1-6s_1^2-9s_1^4+2s_2+10s_1^2s_2+4s_1^4s_2-s_2^2-3s_1^2s_2^2}{s_2}\\
&=-A (1 + 3 s_1^2) + 4s_1^2 (1 + s_1^2),\\
A&=s_2+\frac{1+3s_1^2}{s_2}-2.
\end{align*}
Then the action of $\beta_3$ on $K(t_1,t_2,t_3)$ is 
\begin{align*}
\beta_3 : \ga\mapsto -\ga,\ t_1\mapsto -t_1,\ t_2\mapsto t_2,\ t_3\mapsto t_3.
\end{align*}
Hence $K(x_1,x_2,x_3)^G=k(\ga t_1,t_2,t_3)$ is $k$-rational.\\ 

%%%%%%%%%%%%%%%%%%%%%%%%%%%%%%%%%%%%%%%%%%%%%%%%
Case 7: $H=\langle\ta_3,\la_3,\cb,\be_3\rangle\simeq \cS_4\leq G=G_{7,5,3}$.\\

By Equation (\ref{acts}), 
$K(x_1,x_2,x_3)^H=K(t_1,t_2,t_3)$ where $t_1=s_1^2, t_2=s_2, t_3=s_3$. 
Define 
\begin{align*}
u_1=\frac{1+3t_1}{t_2}+t_2-4,\ 
u_2=\frac{1+3t_1}{t_2}-t_2,\ 
u_3=6t_3.
\end{align*} 
Then $K(t_1,t_2,t_3)=K(u_1,u_2,u_3)=k(\de)(u_1,u_2,u_3)$ and 
\begin{align*}
-\be_3 &: \de\mapsto-\de,\ 
u_1\mapsto u_1,\ 
u_2\mapsto -u_2,\ 
u_3\mapsto \frac{10u_1^2+7u_1^3+u_1^4-18u_2^2
-7u_1u_2^2-2u_1^2u_2^2+u_2^4}{u_3}.
\end{align*}
Define 
\begin{align*}
v_1=\frac{7u_1+2u_1^2-2u_2^2}{3u_1},\ 
v_2=\frac{2u_2}{\de u_1},\ 
v_3=\frac{2u_3}{3u_1}.
\end{align*}
Then $K(u_1,u_2,u_3)=K(v_1,v_2,v_3)$ and 
\begin{align*}
-\be_3 &: \de\mapsto-\de,\ 
v_1\mapsto v_1,\ 
v_2\mapsto v_2,\ 
v_3\mapsto \frac{-1+v_1^2-2dv_2^2}{v_3}
\end{align*}
where $d=\de^2\in k$. 
By Lemma \ref{lemXY}, we have 
$K(x_1,x_2,x_3)^G=K(X,Y,v_1,v_2)$ where $X^2-dY^2=-1+v_1^2-2dv_2^2$. 
Hence $K(x_1,x_2,x_3)^G=k(X+v_1,Y,v_2)$ is $k$-rational.

%%%%%%%%%%%%%%%%%%%%%%%%%%%%%%%%%%%%%%%%%%%%%%%%%%%%
\section{Proof of Proposition \ref{prop1}}\label{seP3}

We consider the three cases: 
\[
G_{7,2,3}=\langle H,-I_3\rangle,\ 
G_{7,3,3}=\langle H,-\be_3\rangle,\ 
G_{7,5,3}=\langle H,\be_3,-I_3\rangle
\]
where 
$H=\langle\ta_3,\la_3,\cb\rangle\simeq \cA_4$.\\

%%%%%%%%%%%%%%%%%%%%%%%%%%%%%%%%%%%%%%%%%%%%%%%%%%
{\it Proof of Proposition \ref{prop1} (1).} 

We consider the case where $G=G_{7,2,3}=\langle H,-I_3\rangle\simeq \cA_4\times \cC_2$. 

%%%%%%%%%%%%%%%%%%%%%%%%%%%%%%%%%%%%%%
By Equation (\ref{acts}), we have 
$K(x_1,x_2,x_3)^H=K(s_1,s_2,s_3)$ 
and $-\be_3,\be_3,-I_3$ act on $K(s_1,s_2,s_3)$ by
\begin{align*}
-\be_3^* &: %\de\mapsto-\de, 
s_1\mapsto s_1,\ 
s_2\mapsto \frac{1+3s_1^2}{s_2},\ 
s_3\mapsto \frac{-1-6s_1^2-9s_1^4+2s_2+10s_1^2s_2
+4s_1^4s_2-s_2^2-3s_1^2s_2^2}{s_2s_3},\nonumber\\
\be_3^* &: %\ga\mapsto-\ga, 
s_1\mapsto -s_1,\ s_2\mapsto s_2,\ s_3\mapsto s_3,\\
-I_3^* &: %\ep\mapsto-\ep, 
s_1\mapsto -s_1,\ 
s_2\mapsto \frac{1+3s_1^2}{s_2},\ 
s_3\mapsto \frac{-1-6s_1^2-9s_1^4+2s_2+10s_1^2s_2
+4s_1^4s_2-s_2^2-3s_1^2s_2^2}{s_2s_3}.\nonumber
\end{align*}
Define 
\begin{align*}
t_1&=\frac{1+3s_1^2-2s_2-2s_1^2s_2+s_2^2}
{(1+s_1-s_2)(1+3s_1^2-s_2-s_1s_2)},\\
t_2&=\frac{s_1(1+3s_1^2-s_2^2)}{1+3s_1^2-2s_2-2s_1^2s_2+s_2^2},\\
t_3&=\frac{2(1+3s_1^2-2s_2+s_2^2)s_3}
{(1+s_1-s_2)(1+3s_1^2-2s_2-2s_1^2s_2+s_2^2)}.
\end{align*}
Then $K(x_1,x_2,x_3)^H=K(s_1,s_2,s_3)=K(t_1,t_2,t_3)$ because 
\begin{align*}
s_1&=\frac{1-t_1^2+t_1^2t_2^2}{-1+4t_1-t_1^2+t_1^2t_2^2},\\
s_2&=\frac{2(-1+2t_1-2t_1^2+t_1^3+2t_1^2t_2-t_1^3t_2-
t_1^3t_2^2+t_1^3t_2^3)}{-1+5t_1-5t_1^2+t_1^3-t_1t_2+
4t_1^2t_2-t_1^3t_2+t_1^2t_2^2-t_1^3t_2^2+t_1^3t_2^3},\\
s_3&=\frac{2(t_1t_3-t_1^3t_3+t_1^3t_2^2t_3)}{(1-t_1+t_1t_2)(-1+
4t_1-t_1^2+t_1^2t_2^2)^2}.
\end{align*}
The actions of $-\be_3,\be_3,-I_3$ on $K(t_1,t_2,t_3)$ 
are given by
\begin{align*}
-\be_3^* &: %\de\mapsto-\de,
t_1\mapsto t_1,\ 
t_2\mapsto -t_2,\ 
t_3\mapsto \frac{2\left(5-t_2^2+8(t_2^2-1)t_1-(t_2^2-1)(5-t_2^2)t_1^2\right)}
{t_3},\\
\be_3^*&:%\ga\mapsto-\ga, 
t_1\mapsto -\frac{1}{t_1(t_2^2-1)},\ 
t_2\mapsto -t_2,\ 
t_3\mapsto -\frac{t_3}{t_1(t_2-1)},\\
-I_3^* &: %\ep\mapsto-\ep, 
t_1\mapsto -\frac{1}{t_1(t_2^2-1)},\ 
t_2\mapsto t_2,\ 
t_3\mapsto \frac{2\left(5-t_2^2+8(t_2^2-1)t_1-(t_2^2-1)(5-t_2^2)t_1^2\right)}
{t_1(t_2+1)t_3}.
\end{align*}
We also define 
\[
u_1=\frac{3+t_2^2}{4-5t_1+t_1t_2^2},\ 
u_2=t_2,\ 
u_3=-\frac{(5-t_2^2)t_3}{4-5t_1+t_1t_2^2}.
\]
Then $K(t_1,t_2,t_3)=K(u_1,u_2,u_3)$ 
and $-\be_3,\be_3,-I_3$ act on $K(u_1,u_2,u_3)$  by 
\begin{align}
-\be_3^* &: %\de\mapsto-\de,
u_1\mapsto u_1,\ 
u_2\mapsto -u_2,\ 
u_3\mapsto \frac{2(5-u_2^2)(1+u_1^2-u_2^2)}{u_3},\nonumber\\
\be_3^*&:%\ga\mapsto-\ga, 
u_1\mapsto -\frac{(u_2^2-1)(3-4u_1+u_2^2)}{4+3u_1-4u_2^2+u_1u_2^2},\ 
u_2\mapsto -u_2,\ 
u_3\mapsto -\frac{(u_2+1)(5-u_2^2)u_3}{4+3u_1-4u_2^2+u_1u_2^2},\label{eqb3}\\
-I_3^* &: %\ep\mapsto-\ep, 
u_1\mapsto -\frac{(u_2^2-1)(3-4u_1+u_2^2)}{4+3u_1-4u_2^2+u_1u_2^2},\ 
u_2\mapsto u_2,\ 
u_3\mapsto \frac{2(u_2-1)(1+u_1^2-u_2^2)(5-u_2^2)^2}
{(4+3u_1-4u_2^2+u_1u_2^2)u_3}.\nonumber
\end{align}
Define 
\[
v_1=-\frac{4+3u_1-4u_2^2+u_1u_2^2}{(u_2+1)(5-u_2^2)},\ 
v_2=-\frac{u_2-1}{u_2+1},\ 
v_3=\frac{(3+u_2^2)u_3}{(u_2+1)(4+3u_1-4u_2^2+u_1u_2^2)}.
\]
Then $K(x_1,x_2,x_3)^H=K(u_1,u_2,u_3)=K(v_1,v_2,v_3)$ 
and $-\be_3,\be_3,-I_3$ act on $K(v_1,v_2,v_3)$  by 
\begin{align}
-\be_3^* &: %\de\mapsto-\de,
v_1\mapsto \frac{v_1}{v_2},\ 
v_2\mapsto \frac{1}{v_2},\ 
v_3\mapsto \frac{2(v_1^2+v_2+4v_1v_2+3v_1^2v_2+3v_2^2+4v_1v_2^2+
v_1^2v_2^2+v_2^3)}{v_1^2v_2v_3},\nonumber\\
\be_3^* &: v_1\mapsto \frac{1}{v_1},\ 
v_2\mapsto \frac{1}{v_2},\ 
v_3\mapsto \frac{v_1v_3}{v_2^2},\label{actv}\\
-I_3^* &: v_1\mapsto \frac{v_2}{v_1},\ 
v_2\mapsto v_2,\ 
v_3\mapsto \frac{2(v_1^2+v_2+4v_1v_2+3v_1^2v_2+3v_2^2+4v_1v_2^2+v_1^2v_2^2+v_2^3)}{v_1v_3}.\nonumber
\end{align}
We note that 
\begin{align*}
-I_3(v_3)=\frac{1}{v_3}
\left(8v_2(v_2+1)+2(1+3v_2+v_2^2)\left(v_1+\frac{v_2}{v_1}\right)\right).
\end{align*}
Define 
\[
w_1=\frac{1}{2}\left(v_1+\frac{v_2}{v_1}\right),\ 
w_2=\frac{1}{2}\left(v_1-\frac{v_2}{v_1}\right),\ 
w_3=\frac{v_3}{2}.
\]
Then $K(w_1,w_2,w_3)=K(v_1,v_2,v_3)$ 
and $-\be_3,\be_3,-I_3$ act on $K(w_1,w_2,w_3)$  by 
\begin{align*}
-\be_3^* &: %\de\mapsto-\de,
w_1\mapsto \frac{w_1}{w_1^2-w_2^2},\ 
w_2\mapsto \frac{w_2}{w_1^2-w_2^2},\\
&\quad w_3\mapsto \frac{w_1(1+w_1+w_1^2)^2-(2+3w_1+4w_1^2+2w_1^3)w_2^2
+(w_1+2)w_2^4}{(w_1+w_2)(w_1^2-w_2^2)w_3},\\
\be_3^*&:
w_1\mapsto \frac{w_1}{w_1^2-w_2^2},\ 
w_2\mapsto -\frac{w_2}{w_1^2-w_2^2},\ 
w_3\mapsto \frac{w_3}{(w_1-w_2)(w_1^2-w_2^2)},\\
-I_3^*&:
w_1\mapsto w_1,\ 
w_2\mapsto -w_2,\\
&\quad w_3\mapsto\frac{w_1(1+w_1+w_1^2)^2-(2+3w_1+4w_1^2+2w_1^3)w_2^2
+(w_1+2)w_2^4}{w_3}.
\end{align*}

We may take $K=k(\sqrt{d})$ with $-I_3(\sqrt{d})=-\sqrt{d}$. 
Define $W_1=w_1$, $W_2=w_2/\sqrt{d}$. 
By Lemma \ref{lemXY}, we have 
$K(x_1,x_2,x_3)^{G}=k(\sqrt{d})(w_1,w_2,w_3)^{\langle-I_3\rangle}=k(X,Y,W_1,W_2)$ 
where 
\begin{align*}
X^2-dY^2=W_1(1+W_1+W_1^2)^2-d(2+3W_1+4W_1^2+2W_1^3)W_2^2
+d^2(W_1+2)W_2^4.
\end{align*}
This completes the proof of Proposition \ref{prop1} (1). \qed\\

%%%%%%%%%%%%%%%%%%%%%%%%%%%%%%%%%%%%%%%%%%%%%%%%%%
{\it Proof of Proposition \ref{prop1} (2).}

We consider the case where $G=G_{7,3,3}=\langle H,-\be_3\rangle\simeq \cS_4$. 

%%%%%%%%%%%%%%%%%%%%%%%%%%%%%%%%%%%%%%%%%%%%%%%%%%%%%%%%%%%%%

{}From Equation (\ref{eqb3}), we have 
$K(x_1,x_2,x_3)^G=K(u_1,u_2,u_3)^{\langle-\be_3\rangle}$ where 
$K=k(\sqrt{d})$ and 
\begin{align*}
-\be_3 &: \sqrt{d}\mapsto-\sqrt{d},\ 
u_1\mapsto u_1,\ 
u_2\mapsto -u_2,\ 
u_3\mapsto \frac{2(5-u_2^2)(1+u_1^2-u_2^2)}{u_3}.
\end{align*}

Now we assume that $\sqrt{5}\in k$. 

Define 
\begin{align*}
U_1=u_1,\ U_2=\frac{u_2}{\sqrt{d}},\ U_3=\frac{u_3}{u_2-\sqrt{5}}.
\end{align*}
Then $K(u_1,u_2,u_3)=K(U_1,U_2,U_3)$ and 
\begin{align*}
-\be_3 &: \sqrt{d}\mapsto-\sqrt{d},\ 
U_1\mapsto U_1,\ 
U_2\mapsto U_2,\ 
U_3\mapsto \frac{2(1+U_1^2-dU_2^2)}{U_3}.
\end{align*}
By Lemma \ref{lemXY}, 
we get $K(U_1,U_2,U_3)^{\langle-\be_3\rangle}=k(X,Y,U_1,U_2)$ 
where $X^2-dY^2=2(1+U_1^2-dU_2^2)$. 
Hence Proposition \ref{prop1} (2) follows from Theorem \ref{thOhm}. \qed\\

%%%%%%%%%%%%%%%%%%%%%%%%%%%%%%%%%%%%%%%%%%%%%%%%%%
{\it Proof of Proposition \ref{prop1} (3).}

We consider the case where $G=G_{7,5,3}=\langle H,\be_3,-I_3\rangle\simeq \cS_4\times \cC_2$.

By Equation (\ref{actv}), we have $K(x_1,x_2,x_3)^H=K(v_1,v_2,v_3)$. 
Define 
\begin{align*}
p_1=-\frac{v_1 (v_2-1)}{(v_1-1) (v_1 - v_2)},\ 
p_2=\frac{v_1^2 - v_2}{(v_1-1) (v_1 - v_2)},\ 
p_3=\frac{4 v_1^2 (v_1 + v_2^2) v_3}{(v_1 - v_2)^3 (v_1 v_2+1)}.
\end{align*}
It follows from 
\begin{align*}
v_1&=\frac{p_1+p_2+1}{p_1+p_2-1},\ 
v_2=\frac{1+2p_1+p_1^2-p_2^2}{1-2p_1+p_1^2-p_2^2},\\
v_3&=-\frac{2(3p_1+p_1^3+p_2+p_1^2p_2-p_1p_2^2-p_2^3)p_3}
{(p_1-p_2-1)^2(-3p_1+3p_1^2-p_1^3+p_1^4+p_2+2p_1p_2+p_1^2p_2-
p_2^2+p_1p_2^2-2p_1^2p_2^2-p_2^3+p_2^4)}
\end{align*}
that $K(v_1,v_2,v_3)=K(p_1,p_2,p_3)$ and 
\begin{align*}
\be_3 : \sqrt{a}\mapsto -\sqrt{a},\ \sqrt{b}\mapsto \sqrt{b},\ 
&p_1\mapsto -p_1,\ p_2\mapsto -p_2,\ p_3\mapsto -p_3,\\
-I_3 : \sqrt{a}\mapsto \sqrt{a},\ \sqrt{b}\mapsto -\sqrt{b},\ 
&p_1\mapsto p_1,\ p_2\mapsto -p_2,\\ 
&p_3\mapsto \frac{-1-5p_1^2-7p_2^2-(p_1^2-p_2^2)(3p_1^2+17p_2^2)+9(p_1^2-p_2^2)^3}{p_3}
\end{align*}
where $K=k(\sqrt{a},\sqrt{b})$. 
Define 
\[
P_1=\frac{p_1}{\sqrt{a}},\ 
P_2=\frac{p_2}{\sqrt{a}\sqrt{b}},\ 
P_3=\frac{p_3}{\sqrt{a}}.
\]
Then, by Lemma \ref{lemXY}, we have 
$K(x_1,x_2,x_3)^{G}=k(X,Y,P_1,P_2)$ 
 and 
\begin{align*}
X^2-bY^2=-\tfrac{1}{a}-5P_1^2-7bP_2^2-a(P_1^2-bP_2^2)(3P_1^2+17bP_2^2)+9a^2(P_1^2-bP_2^2)^3.
\end{align*}
This completes the proof of Proposition \ref{prop1} (3). \qed\\

%%%%%%%%%%%%%%%%%%%%%%%%%%%%%%%%%%%%%%%%%%%%%%%%%%
{\it Proof of Proposition \ref{prop1} (4).} 

We assume that $\sqrt{-3}\in k$. 
Take $G=G_{7,2,3}=\langle H,-I_3\rangle\simeq \cA_4\times \cC_2$.

By Equation (\ref{acts}), we have $K(x_1,x_2,x_3)^H=K(s_1,s_2,s_3)$. 
Define 
\begin{align*}
t_1&=2s_1,\\
t_2&=\frac{-s_1-7s_1^3+2s_1s_2+2s_1^3s_2+(1+5s_1^2+2s_1^4-
s_2-3s_1^2s_2)\sqrt{-3}}{(s_1^2-1)(1-s_2+s_1\sqrt{-3})},\\
t_3&=\frac{4(s_1^2+1)s_2s_3}{(s_1^2-1)(1-s_2+s_1\sqrt{-3})}.
\end{align*}
Then $K(x_1,x_2,x_3)^H=K(s_1,s_2,s_3)=K(t_1,t_2,t_3)$ because 
\begin{align*}
s_1&=\frac{t_1}{2},\\
s_2&=\frac{4t_1+7t_1^3-8t_2+2t_1^2t_2+(-8-10t_1^2-t_1^4-4t_1t_2+t_1^3t_2)\sqrt{-3}}
{2(4t_1+t_1^3-4t_2+t_1^2t_2-(4+3t_1^2)\sqrt{-3})},\\
s_3&=-\frac{(t_1^2-4)t_1t_3(2+t_1\sqrt{-3})}
{2(-4t_1-7t_1^3+8t_2-2t_1^2t_2+(8+10t_1^2+t_1^4+4t_1t_2-t_1^3t_2)\sqrt{-3})}.
\end{align*}
The action of $-I_3$ on $K(t_1,t_2,t_3)$ becomes 
\begin{align*}
-I_3 : \sqrt{b}\mapsto -\sqrt{b},\ 
t_1\mapsto -t_1,\ t_2\mapsto t_2,\ t_3\mapsto \frac{(t_1^2+4)(t_1^2-t_2^2+1)}{t_3}.
\end{align*}
where $K=k(\sqrt{b})$. 

Now we also assume that $\sqrt{-1}\in k$. 
Define 
\begin{align*}
u_1=\frac{t_1}{\sqrt{b}},\ 
u_2=t_2,\ 
u_3=\frac{t_3}{t_1+2\sqrt{-1}}.
\end{align*}
Then $K(t_1,t_2,t_3)=K(u_1,u_2,u_3)$ and 
\begin{align*}
-I_3 : \sqrt{b}\mapsto -\sqrt{b},\ 
u_1\mapsto u_1,\ u_2\mapsto u_2,\ u_3\mapsto -\frac{bu_1^2-u_2^2+1}{u_3}.
\end{align*}
By Lemma \ref{lemXY}, we have 
$K(u_1,u_2,u_3)^{\langle -I_3\rangle}=k(X,Y,u_1,u_2)$ where $X^2-bY^2=-bu_1^2+u_2^2-1$. 
Hence $K(x_1,x_2,x_3)^G=K(u_1,u_2,u_3)^{\langle -I_3\rangle}=k(X,Y+u_1,u_2)$ 
is $k$-rational. \qed\\

%%%%%%%%%%%%%%%%%%%%%%%%%%%%%%%%%%%%%%%%%%%%%%%%%%
{\it Proof of Proposition \ref{prop1} (5).} 

We assume that char $k=3$. 

By Equation (\ref{acts}), we have 
$K(x_1,x_2,x_3)^H=K(s_1,s_2,s_3)$ where $H=\langle\ta_3,\la_3,\cb\rangle$ 
and 
\begin{align*}
-\be_3^* &: 
s_1\mapsto s_1,\ 
s_2\mapsto \frac{1}{s_2},\ 
s_3\mapsto \frac{-1-s_2+s_1^2s_2+s_1^4s_2-s_2^2}{s_2s_3},\\
\be_3^* &: 
s_1\mapsto -s_1,\ 
s_2\mapsto s_2,\ 
s_3\mapsto s_3,\\
-I_3^* &:  
s_1\mapsto -s_1,\ 
s_2\mapsto \frac{1}{s_2},\ 
s_3\mapsto \frac{-1-s_2+s_1^2s_2+s_1^4s_2-s_2^2}{s_2s_3}.
\end{align*}
Define 
\[
t_1=s_1,\ 
t_2=\frac{(s_1^2-1)(s_2-1)}{s_1(s_2+1)},\ 
t_3=\frac{s_3}{s_1(s_2+1)}.
\]
Then $K(s_1,s_2,s_3)=K(t_1,t_2,t_3)$ and 
\begin{align*}
-\be_3^* &: 
t_1\mapsto t_1,\ 
t_2\mapsto -t_2,\ 
t_3\mapsto \frac{t_1^2-t_2^2+1}{t_3},\\
\be_3^*&:
t_1\mapsto -t_1,\ 
t_2\mapsto -t_2,\ 
t_3\mapsto -t_3,\\
-I_3^*&:
t_1\mapsto -t_1,\ 
t_2\mapsto t_2,\ 
t_3\mapsto -\frac{t_1^2-t_2^2+1}{t_3}.
\end{align*}
Hence $K(x_1,x_2,x_3)^{G_{7,3,3}}=K(t_1,t_2,t_3)^{\langle-\be_3\rangle}$ 
and $K(x_1,x_2,x_3)^{G_{7,2,3}}=K(t_1,t_2,t_3)^{\langle-I_3\rangle}$ 
are $k$-rational. 
For $G_{7,5,3}=\langle H,\be_3,-I_3\rangle$, 
we may take $K=k(\sqrt{a},\sqrt{b})$ with 
$\be_3(\sqrt{a})=-\sqrt{a}, \be_3(\sqrt{b})=\sqrt{b}$. 
Define $T_1=\sqrt{a}t_1$, $T_2=\sqrt{a}t_2$, $T_3=\sqrt{a}t_3$. 
Then $K(t_1,t_2,t_3)^{\langle\be_3\rangle}=k(\sqrt{b})(T_1,T_2,T_3)$ 
and 
\begin{align*}
-I_3&: \sqrt{b}\mapsto-\sqrt{b},\ 
T_1\mapsto -T_1,\ 
T_2\mapsto T_2,\ 
T_3\mapsto -\frac{T_1^2-T_2^2+a}{T_3}.
\end{align*}
By Lemma \ref{lemXY}, we have 
$K(T_1,T_2,T_3)^{\langle -I_3\rangle}=k(X,Y,U_1,U_2)$ where 
$U_1=T_1/\sqrt{b}$, $U_2=T_2$ and $X^2-bY^2=-bU_1^2+U_2^2-a$. 
Hence $K(x_1,x_2,x_3)^{G_{7,2,5}}=K(T_1,T_2,T_3)^{\langle-I_3\rangle}=k(X,Y+U_1,U_2)$ 
is $k$-rational. \qed

%%%%%%%%%%%%%%%%%%%%%%%%%%%%%%%%%%%%%%%%%%%%%%%%%%%%%%%%%%%%%%%

\section{Proof of Theorem \ref{thmain3}}\label{seP4}

Write $M_1=\bigoplus_{1\leq i\leq 3}\bZ x_i$, 
$M_2=\bigoplus_{1\leq i\leq 2}\bZ y_j$ and 
$k(M)=k(x_1,x_2,x_3,y_1,y_2)$. 
Then $N_1=\{\sigma\in G\mid\sigma(x_i)=x_i$ $(i=1,2,3)\}$ and 
$N_2=\{\sigma\in G\mid\sigma(y_j)=y_j$ $(j=1,2)\}$. 

By Theorem \ref{thp12}, without loss of generality, 
we may assume that $M$ is a faithful $G$-lattice, 
i.e. $N=\{\sigma\in G\mid \sigma(x_i)=x_i, \sigma(y_j)=y_j$ 
$(i=1,2,3,j=1,2)\}=N_1\cap N_2=\{1\}$. \\

Step 1. Case $(G/N_2,N_1N_2/N_2)\not\simeq(\cC_4,\cC_2)$, $(\cD_4,\cC_2)$.\\ 

We have $k(M)^G=K(y_1,y_2)^{G/N_2}$ and $G/N_2$ acts on 
$K(y_1,y_2)$ by purely quasi-monomial $k$-automorphisms 
where $K=k(x_1,x_2,x_3)^{N_2}$. 
By Theorem \ref{thm:HKK2}, 
if $(G/N_2,N_1N_2/N_2)\not\simeq(\cC_4,\cC_2)$, $(\cD_4,\cC_2)$, 
then $k(M)^G=K(y_1,y_2)^{G/N_2}$ is $K^{G/N_2}$-rational. 
It follows from Theorem \ref{thHKHR} that $K^{G/N_2}=k(x_1,x_2,x_3)^G$ 
is $k$-rational. 
Hence $k(M)$ is $k$-rational if 
$(G/N_2,N_1N_2/N_2)\not\simeq(\cC_4,\cC_2)$, $(\cD_4,\cC_2)$. \\

Step 2. Case $(G/N_1,N_1N_2/N_1)\not\simeq$ $(G_{7,j,3},\cA_4)$, 
$(G_{7,j,3},\cC_2\times \cC_2)$ $(j=2,3,5)$.\\

As in Step 1, we have 
$k(M)^G=K(x_1,x_2,x_3)^{G/N_1}$ and $G/N_1$ acts on 
$K(x_1,x_2,x_3)$ by purely quasi-monomial $k$-automorphisms 
where $K=k(y_1,y_2)^{N_1}$. 
It follows from Theorem \ref{thHaj87} that $K^{G/N_1}=k(y_1,y_2)^G$ 
is $k$-rational. 
Thus, by Theorem \ref{thmain2}, 
if $(G/N_1,N_1N_2/N_1)\not\simeq (G_{4,j,k}, \cC_2), 
(G_{4,j,k},\cC_2\times \cC_2)$, $(G_{7,j,3},\cA_4)$, 
$(G_{7,j,3},\cC_2\times \cC_2)$ $(j=2,3,5)$, then 
$k(M)^G$ is $K^{G/N_1}$-rational, and hence $k$-rational. 

For the case where $G/N_1=G_{4,j,k}$, 
the rationality does not depend 
on the $\bQ$-conjugacy class by Theorem \ref{thmain2}. 
Indeed, the actions of the groups which belong to 4th crystal system (I) 
and the groups which belong to 4th crystal system (II) on 
$K(x_1,x_2,x_3)^{\langle\caa^2\rangle}=K(z_1,z_2,z_3)$ 
coincide (see Case 1 in Subsection \ref{subsec:4}). 
Hence we may assume that $G/N_1$ belongs to 
the 4th crystal system (I) and 
$M_1=M_{1,1}\oplus M_{1,2}$ where 
${\rm rank}_{\bZ}M_{1,1}=1$ and ${\rm rank}_{\bZ} M_{1,2}=2$. 
Therefore, it follows from \cite[Proposition 5.2]{HKK14} 
that $k(M)^G$ is $k$-rational. \\

Step 3. 
By Steps 1 and 2, the remaining cases that we should treat are 
$(G/N_1,N_1N_2/N_1)\simeq$ $(G_{7,j,3},\cA_4)$, 
$(G_{7,j,3},\cC_2\times \cC_2)$ $(j=2,3,5)$ and 
$(G/N_2,N_1N_2/N_2)\simeq(\cC_4,\cC_2)$, $(\cD_4,\cC_2)$. 

Note that $N_1N_2/N_1\simeq N_2$ and $N_1N_2/N_2\simeq N_1$ 
because $N_1\cap N_2=\{1\}$. 
Thus $\#G=\# G_{7,j,3}\times \# N_1=\# G/N_2\times \# N_2$ 
where $N_1\simeq \cC_2$ and $N_2\simeq \cA_4$ or $\cC_2\times\cC_2$. 
Hence we see that only the cases where 
$(G/N_1,N_2,G/N_2,N_1)\simeq (G_{7,2,3},\cA_4,\cC_4,\cC_2)$, 
$(G_{7,3,3},\cA_4,\cC_4,\cC_2)$ or $(G_{7,5,3},\cA_4,\cD_4,\cC_2)$ may occur. 
The last statement follows from Step 2 and Proposition \ref{prop1} (5).\qed

%%%%%%%%%%%%%%%%%%%%%%%%%%%%%%%%%%%%%%%%%%%%%%%%%%%%%%%%%%%%%%%
%


\begin{thebibliography}{BBNWZ78}
\bibitem[AHK00]{AHK00} H. Ahmad, M. Hajja, M. Kang, {\it Rationality of some projective linear actions}, J. Algebra \textbf{228} (2000) 643--658.
\bibitem[BBNWZ78]{BBNWZ78} H. Brown, R. B\"ulow, J. Neub\"user, H. Wondratschek, H. Zassenhaus. {\it Crystallographic Groups of Four-Dimensional Space}, John Wiley, New York, 1978. 
\bibitem[CHKK10]{CHKK10}
H. Chu, S.-J. Hu, M. Kang, B. E. Kunyavskii, 
\textit{Noether's problem and the unramified Brauer groups 
for groups of order $64$},
Int. Math. Res. Not. IMRN \textbf{2010}, 2329--2366.
\bibitem[CTS07]{CTS07}
J.-L. Colliot-Th\'el\`ene, J.-J. Sansuc, 
\textit{The rationality problem for fields of invariants under linear
algebraic groups \rm{(}with special regards to the Brauer
groups\rm{)}}, in ``Proc. International Conference, Mumbai, 2004"
edited by V. Mehta, Narosa Publishing House, 2007.
\bibitem[Dra83]{Dra83}
P. K. Draxl, \textit{Skew fields}, London Math. Soc. Lecture Note
Series vol. 81, Cambridge Univ. Press, Cambridge, 1983.
\bibitem[EM73]{EM73} S. Endo, T. Miyata, 
{\it Invariants of finite abelian groups}, 
J. Math. Soc. Japan \textbf{25} (1973) 7--26. 
\bibitem[Haj87]{Haj87} M. Hajja, {\it Rationality of finite groups of monomial automorphisms of $k(x,y)$}, J. Algebra \textbf{109} (1987) 46--51. 
\bibitem[HK92]{HK92} M. Hajja, M. Kang, {\it Finite group actions on rational function fields}, J. Algebra \textbf{149} (1992) 139--154. 
\bibitem[HK94]{HK94} M. Hajja, M. Kang, {\it Three-dimensional purely monomial group actions}, J. Algebra \textbf{170} (1994) 805--860. 
\bibitem[HK95]{HK95} M. Hajja, M. Kang, 
{\it Some actions of symmetric groups}, 
J. Algebra \textbf{177} (1995) 511--535. 
\bibitem[HKO94]{HKO94} M. Hajja, M. Kang, J. Ohm, 
{\it Function fields of conics as invariant subfields}, 
J. Algebra \textbf{163} (1994) 383--403.
\bibitem[HHR08]{HHR08} K. Hashimoto, A. Hoshi, Y. Rikuna, {\it Noether's problem and $\bQ$-generic polynomials for the normalizer of the $8$-cycle in $S\sb 8$ and its subgroups}, Math. Comp. \textbf{77} (2008) 1153--1183. 
\bibitem[Hos14]{Hos14} A. Hoshi, 
{\it Rationality problem for quasi-monomial actions}, 
Algebraic number theory and related topics 2012, 203--227, 
RIMS K\^oky\^uroku Bessatsu, B51, Res. Inst. Math. Sci. (RIMS), Kyoto, 2014. 
\bibitem[HK10]{HK10} A. Hoshi, M. Kang, {\it Twisted symmetric group actions}, 
Pacific J. Math. \textbf{248} (2010) 285--304. 
\bibitem[HKK14]{HKK14} A. Hoshi, M. Kang, H. Kitayama, 
{\it Quasi-monomial actions and some $4$-dimensional rationality problems}, 
J. Algebra \textbf{403} (2014) 363--400.
\bibitem[HKY11]{HKY11} A. Hoshi, H. Kitayama, A. Yamasaki, {\it Rationality problem of three-dimensional monomial group actions}, J. Algebra \textbf{341} (2011) 45--108. 
\bibitem[HR08]{HR08} A. Hoshi, Y. Rikuna, {\it Rationality problem of three-dimensional purely monomial group actions: the last case}, Math. Comp. \textbf{77} (2008) 1823--1829. 
\bibitem[HY17]{HY17} A. Hoshi, A. Yamasaki, 
Rationality problem for algebraic tori,
Mem. Amer. Math. Soc. \textbf{248} (2017) no. 1176, v+215 pp.
\bibitem[Kan04]{Kan04} M. Kang, {\it Rationality problem of $\rm GL\sb 4$ group actions}, 
Adv. Math. \textbf{181} (2004) 321--352.
\bibitem[Kan07]{Kan07} M. Kang, 
{\it Some rationality problems revisited}, 
in ``Proceedings of the 4th ICCM, Hangzhou, 2007", edited by Lizhen
Ji, Kefeng Liu, Lo Yang and Shing-Tung Yau, Higher Education Press
(Beijing) and International Press (Somerville), 2007.
\bibitem[Kan09]{Kan09}
M. Kang, \textit{Retract rationality and Noether's problem}, 
Int. Math. Res. Not. IMRN \textbf{2009}, no. 15, 2760--2788.
\bibitem[Kan12]{Kan12} M. Kang, {\it Retract rational fields}, J. Algebra \textbf{349} (2012) 22--37. 
\bibitem[Kit11]{Kit11} H. Kitayama, {\it The rationality problem for purely monomial group actions}, Pacific J. Math. \textbf{253} (2011) 95--102. 
\bibitem[Kun87]{Kun87} B. E. Kunyavskii, 
{\it Three-dimensional algebraic tori} (Russian), 
in Investigations in number theory, 90--111, Saratov. Gos. Univ., Saratov, 
1987. English translation: Selecta Math. Soviet. \textbf{9} (1990) 1--21.
\bibitem[MT86]{MT86}
Y. I. Manin, M. A. Tsfasman, 
\textit{Rational varieties: algebra, geometry and arithmetic}, 
Russian Math. Survey \textbf{41} (1986) 51--116.
\bibitem[Mas55]{Mas55} K. Masuda, {\it On a problem of Chevalley}, Nagoya Math. J. \textbf{8} (1955) 59--63. 
\bibitem[Miy71]{Miy71} T. Miyata, {\it Invariants of certain groups I}, 
Nagoya Math. J. \textbf{41} (1971) 69--73. 
\bibitem[Ohm94]{Ohm94} J. Ohm, 
{\it Function fields of conics, a theorem of Amitsur-MacRae, 
and a problem of Zariski}, 
Algebraic geometry and its applications 
(West Lafayette, IN, 1990), 333--363, Springer, New York, 1994. 
\bibitem[Ono61]{Ono61}
T. Ono, {\it Arithmetic of algebraic tori}, 
Ann. of Math. (2) \textbf{74} (1961) 101--139. 
\bibitem[Pro10]{Pro10}
Y. G. Prokhorov, \textit{Fields of invariants of finite linear
groups}, in ``Cohomological and geometric approaches to
rationality problems", edited by F. Bogomolov and Y. Tschinkel,
Progress in Math. vol. 282, Birkh\"auser, Boston, 2010.
%\bibitem[Sal84]{Sal84}
%D. J. Saltman, \textit{Retract rational fields and cyclic Galois
%extensions}, Israel J. Math. \textbf{47} (1984), 165--215. 
%\bibitem[Sal87]{Sal87}
%D. J. Saltman, \textit{Multiplicative field invariants}, 
%J. Algebra \textbf{106} (1987) 221--238.
\bibitem[Sal90a]{Sal90a}
D. J. Saltman, 
\textit{Twisted multiplicative field invariants, 
Noether's problem and Galois extensions}, 
J. Algebra \textbf{131} (1990) 535--558.
\bibitem[Sal90b]{Sal90b}
D. J. Saltman, 
\textit{Multiplicative field invariants and the Brauer group}, 
J. Algebra \textbf{133} (1990) 533--544.
\bibitem[Sal00]{Sal00} 
D. J. Saltman, 
{\it A nonrational field, answering a question of Hajja}, 
in: Algebra and Number Theory, 
in: Lect. Notes Pure Appl. Math. \textbf{208} (2000) 263--271. 
\bibitem[Swa83]{Swa83}
R. G. Swan, \textit{Noether's problem in Galois theory}, in ``Emmy
Noether in Bryn Mawr", edited by B. Srinivasan and J. Sally,
Springer-Verlag, Berlin, 1983.
\bibitem[Vos67]{Vos67} V. E. Voskresenskii, 
{\it On two-dimensional algebraic tori II} (in Russian), 
Izv. Akad. Nauk SSSR Ser. Mat. \textbf{31} (1967) 711--716. 
English translation: Math. USSR Izv. \textbf{1} (1967) 691--696.
\bibitem[Yam12]{Yam12} A. Yamasaki, {\it Negative solutions to three-dimensional monomial Noether problem}, J. Algebra \textbf{370} (2012) 46--78. 
\bibitem[Yam]{Yam} 
A. Yamasaki, 
{\it Rationality problem of conic bundles}, arXiv:1408.2233.
\end{thebibliography}
\end{document}